\documentclass[11pt]{article}
\usepackage[utf8]{inputenc}
\usepackage[T1]{fontenc}
\usepackage[english]{babel}
\usepackage{amsmath,amssymb,amsthm,mathtools,bm,array,booktabs,enumitem}
\usepackage[a4paper,margin=1in]{geometry}
\IfFileExists{lmodern.sty}{\usepackage{lmodern}}{}
\IfFileExists{microtype.sty}{\usepackage{microtype}}{}
\usepackage[hidelinks]{hyperref}
\hypersetup{
 pdftitle={Occupation-Based Propagation of Chaos for McKean--Vlasov Forward--Backward Systems with Jumps and Monotone Sources},
 pdfauthor={Kayembe Tshiswaka Tcheick; Mabela Matendo Rostin; Ntumba Badibanga Simon; Bosonga Bofeki Jean-Pierre},
 pdfkeywords={McKean--Vlasov equations, propagation of chaos, jump BSDEs, maximal monotone operators, occupation estimates, Yosida approximation}}
\allowdisplaybreaks
\numberwithin{equation}{section}
\theoremstyle{plain}
\newtheorem{theorem}{Theorem}[section]
\newtheorem{lemma}[theorem]{Lemma}
\newtheorem{proposition}[theorem]{Proposition}
\newtheorem{corollary}[theorem]{Corollary}
\theoremstyle{definition}
\newtheorem{definition}[theorem]{Definition}
\newtheorem{hypothesis}[theorem]{Assumption}
\newtheorem{remark}[theorem]{Remark}
\newcommand{\R}{\mathbb R}
\newcommand{\E}{\mathbb E}
\renewcommand{\P}{\mathbb P}
\newcommand{\Pcal}{\mathcal P}
\newcommand{\Lcal}{\mathcal L}
\newcommand{\F}{\mathbb F}
\newcommand{\Wtwo}{W_2}
\newcommand{\one}{\mathbf 1}

\begin{document}
\raggedbottom
\title{Occupation-Based Propagation of Chaos for McKean--Vlasov Forward--Backward Systems with Jumps and Monotone Sources}
\author{Kayembe Tshiswaka Tcheick$^{1}$ \and Mabela Matendo Rostin$^{1}$ \and Ntumba Badibanga Simon$^{1}$ \and Bosonga Bofeki Jean-Pierre$^{2}$}
\date{}
\maketitle

\begin{center}
\small
$^{1}$ Universit\'e de Kinshasa, Facult\'e des Sciences et Technologies, Mention Math\'ematiques, Statistique et Informatique, Kinshasa, Democratic Republic of the Congo.\\
$^{2}$ Universit\'e de Kinshasa, Facult\'e des Sciences \`Economiques et de Gestion, Mention \`Economie Math\'ematique, Kinshasa, Democratic Republic of the Congo.\\[0.4em]
Corresponding author: Kayembe Tshiswaka Tcheick; \texttt{tcheick.kayembe@unikin.ac.cd}
\end{center}

\begin{abstract}
We study propagation of chaos for triangular McKean--Vlasov forward--backward
stochastic differential equations with Brownian and Poisson noise when the
backward equation contains an accumulated law-dependent maximal monotone
source whose minimal-norm selection is discontinuous across switching
interfaces. The main difficulty is that a quadratic approximation of the
forward state does not directly yield a Lipschitz estimate for the selected
source. We address this difficulty through occupation estimates for the
limiting forward process. For the asymmetric loss-aversion (ALA) prototype, a
first-moment occupation condition of order $r^\theta$ yields the joint
squared-error bound
$C\{\tau_d(N)^{\theta/(\theta+2)}+\lambda^2+\lambda^\theta\}$.
Under a second-moment occupation condition, direct control of the accumulated
source improves the estimate to
$C\{\tau_d(N)^{\theta/(\theta+1)}+\lambda^2+\lambda^{2\theta}\}$.
No occupation estimate for the particle system is required. The error controls
the forward and backward states, the accumulated source, and all Brownian and
Poisson martingale integrands, uniformly over the model family in the stated
modelwise sense. We also obtain direct propagation for the minimal selection,
including the selected case $\lambda=0$, and verify the stronger occupation
condition for several jump-driven classes, including purely discontinuous and
nonlinear examples.
\end{abstract}


\section*{Introduction}
\addcontentsline{toc}{section}{Introduction}

We consider a triangular McKean--Vlasov forward--backward system with Brownian
and Poisson noise in which a law-dependent maximal monotone operator is
evaluated along the forward state and enters the backward equation through the
cumulative finite-variation process
\[
\mathsf K_t^P=\int_0^t
\langle\varrho,A^0(X_{s-}^P,\mu_{s-}^P)\rangle\,ds,
\qquad \mu_s^P=\Lcal^P(X_s^P).
\]
Here $A^0$ is the minimal-norm selection. The central approximation problem is
to replace the McKean--Vlasov state by an interacting particle system while,
when needed, replacing $A^0$ by its Yosida approximation $A^\lambda$. The
obstruction is structural: $A^0$ is discontinuous across switching interfaces,
so a quadratic state error does not directly provide a Lipschitz source error.

The paper develops an occupation-based way of overcoming this obstruction. A
regularized particle comparison based only on the Lipschitz constant of
$A^\lambda$ produces the generic factor $1+\lambda^{-2}$. For the ALA
prototype, we instead separate the switching interfaces, transfer the bad set
to the limiting forward process, and estimate the time accumulated near those
interfaces. Under the first-moment condition \ref{HoccD}, this gives
\begin{equation}
\mathcal D_{N,\lambda}^{\rm reg}
\le C\{\tau_d(N)^{\theta/(\theta+2)}
+\lambda^2+\lambda^\theta\}.
\label{eq:intro-first}
\end{equation}
Under the stronger second-moment condition \ref{Hocc2}, direct control of the
cumulative source improves the estimate to
\begin{equation}
\mathcal D_{N,\lambda}^{\rm reg}
\le C\{\tau_d(N)^{\theta/(\theta+1)}
+\lambda^2+\lambda^{2\theta}\}.
\label{eq:intro-second}
\end{equation}
The full error $\mathcal D_{N,\lambda}^{\rm reg}$, defined in
\eqref{eq:regerrorfull}, contains the supremum norms of $X,Y,\mathsf K$ and
the energy norms of all Brownian and Poisson martingale components. No
occupation estimate for the particles is assumed. When $\theta=1$,
\eqref{eq:intro-second} becomes
$C\{\tau_d(N)^{1/2}+\lambda^2\}$, whereas the first-moment estimate gives
$C\{\tau_d(N)^{1/3}+\lambda\}$.

The strengthened estimate is obtained by controlling the accumulated source
itself rather than first estimating its instantaneous
$L^2(dt\otimes P)$ error. The argument isolates the singular interface of
$A^0$, compares particle and limiting forward processes, combines occupation
of the limiting process with the bad-coupling event, and transfers the
resulting bound to $(Y,Z,U,\mathsf K)$ through backward stability. This removes
the generic $\lambda^{-2}$ amplification from the joint estimate. The same
comparison also treats the selected system directly: the case $\lambda=0$ is
proved without a limiting substitution into an estimate containing
$\lambda^{-1}$.

McKean--Vlasov equations and their particle approximations go back to
McKean \cite{McKean1966} and Sznitman \cite{Sznitman1991}; forward--backward
mean-field systems are central in control and mean-field games
\cite{CarmonaDelarue2018}. For BSDEs and BSDEs with jumps we refer to
\cite{PardouxPeng1990,TangLi1994,BarlesBuckdahnPardoux1997}. Mean-field
systems with jumps have been studied, among others, in
\cite{Li2018,Erny2022,BaoLiuWang2026,ShenYu2026,QinEtAl2026}. A general
mean-field BSDE treatment under discontinuous martingales is available in the
preprint of Papapantoleon, Saplaouras and Theodorakopoulos
(arXiv:2408.13758). Coupled McKean--Vlasov FBSDEs with jumps appear in
\cite{LiMin2021}; recent preprints include Liu and Zhang (arXiv:2601.19084)
and Feng and Garry (arXiv:2608.23203).

The variational component is related to stochastic variational inequalities.
Jump-driven SVIs were considered in \cite{Zalinescu2014,
MaticiucRascanuSlominski2017}. Propagation of chaos for diffusive
McKean--Vlasov SVIs, including oblique subgradients, is available in
\cite{NingWu2026,DuanWu2026}, and mean-field forward--backward stochastic
variational inequalities with oblique subgradients are studied in
\cite{DuanWuFBSVI2026}. Backward propagation of chaos with reflection,
including jump settings, is also known; see
\cite{LinXu2025,DjehicheDumitrescuZeng2025,LiNing2026}. In the present paper,
however, $\mathsf K$ is not a Skorokhod regulator: it accumulates a nonsmooth
selection depending on the law of the forward process.

Occupation estimates provide the specific mechanism used here. Threshold and
integral-functional estimates have substantial precedents
\cite{Avikainen2007,GanychenkoKnopovaKulik2015,AltmeyerChorowski2018}; a
generalized Avikainen estimate is also developed in Taguchi's preprint
(arXiv:2001.05608). Particularly relevant is
\cite{LeobacherReisingerStockinger2022}, where occupation near a
discontinuity is used for McKean--Vlasov particle and numerical
approximations. There the discontinuity is in the forward drift; here
occupation of the limiting forward process controls the error of a
law-dependent monotone source accumulated in the backward equation. Recent
occupation-measure viewpoints include \cite{TissotDaguette2026,Friesen2026}.
To the best of our knowledge, the cited literature does not provide the
cumulative-source mechanism and the joint estimate \eqref{eq:intro-second} for
the McKean--Vlasov jump forward--backward setting considered here.

The scope of the claims is deliberately limited. We do not claim the first
McKean--Vlasov stochastic variational inequality, the first mean-field
forward--backward variational inequality, the first McKean--Vlasov FBSDE with
jumps, or the first backward propagation-of-chaos result, and no optimality of
the rates is asserted. The probabilistic family may be non-dominated through
the quadratic variation of the continuous martingale, but the jump compensator
$\Pi(de)dt$ is common to all models and the solutions are constructed model by
model. A model-dependent compensator would require a different framework,
such as the 2BSDE-with-jumps setting of
\cite{KaziTaniPossamaiZhou2015a,KaziTaniPossamaiZhou2015b}.

The earlier Brownian preprint by Kayembe, Mabela, Bosonga and Mbuyi
(arXiv:2606.30526) treats a related law-dependent nonsmooth structure. The present proofs are independent of that
paper. In particular, the earlier work does not establish the direct
selected-particle estimate or the cumulative occupation argument yielding a
joint $N$--$\lambda$ bound without the generic $\lambda^{-2}$ amplification.

The paper is organized as follows. Section~1 introduces the Brownian--Poisson
framework, the law-dependent operator and the ALA prototype. Sections~2--4
establish modelwise well-posedness, stability and Yosida convergence.
Section~5 treats the particle system and the regularized propagation-of-chaos
estimate. Section~6 records the generic joint particle--Yosida route, where
the $\lambda^{-2}$ amplification remains visible. Section~7 develops direct
propagation for the selected source and the occupation-based cumulative
estimate that removes this amplification. Section~8 verifies the occupation
conditions for several jump-driven classes. The final section states the
scope and limitations of the results.

\section{Brownian--Poisson framework}

\subsection{Family of models and functional spaces}

Fix $T>0$, $d,m,r\ge1$, $E=\R^r\setminus\{0\}$, and a Lévy measure $\Pi$
such that
\begin{equation}
\int_E|e|^2\Pi(de)<\infty.
\label{eq:Pi2}
\end{equation}
This second-moment assumption is maintained throughout the paper; our
conventions for Lévy measures and compensated integrals are those of
\cite{Applebaum2009}.

For $p\ge1$, let $\mathcal P_p(\R^n)$ denote the set of probability measures
$\mu$ on $\R^n$ such that
\[
M_p(\mu)^p:=\int_{\R^n}|x|^p\mu(dx)<\infty.
\]
The space $L_\Pi^p(E;\R^n)$ is endowed with the norm
\[
\|u\|_{L_\Pi^p}^p=\int_E|u(e)|^p\Pi(de);
\]
we write $\|u\|_\Pi=\|u\|_{L_\Pi^2}$. Matrix norms are Hilbert--Schmidt
norms. The letter $C$ denotes a deterministic constant that may change from
line to line; unless stated otherwise, it depends only on
\[
T,d,m,r,q,\underline a,\overline a,L_F,C_F,L_B,C_B,
L_K,C_K,C_A,|\varrho|,
\]
and never on $P$, $N$, or $\lambda$. Moment estimates and convergence rates
also depend on the uniform moment bounds and on the parameters appearing in
the occupation or regularity assumptions that are invoked. For an estimate
at fixed $\lambda$, $C_\lambda$ denotes a constant that may depend on
$\lambda$. Thus $C_\lambda=C(1+\lambda^{-2})$ is admissible notation; in the
estimates used to balance the limits $N\to\infty$ and $\lambda\downarrow0$,
we keep the factor $1+\lambda^{-2}$ explicit. The multiplicative constant
$C$ remains uniform in $\lambda$.

More precisely, we may take
\[
\Omega=\R^d\times C_0([0,T];\R^m)\times\mathcal M,
\]
where $C_0$ denotes continuous paths starting from zero and $\mathcal M$ the
integer-valued measures on $(0,T]\times E$ that are finite on
$(0,T]\times\{|e|>1/n\}$ for every $n\ge1$. The sigma-field on $\mathcal M$
is generated by the evaluation maps. The coordinates are $\xi$, $B$, and
$N$. The raw canonical filtration is
\[
\mathcal F_t^0=\sigma\!\left(\xi,B_s,N((0,s]\times D):
0\le s\le t,\ D\in\mathcal B(E),\ \operatorname{dist}(D,0)>0\right).
\]
For each probability measure $P$ in the nonempty family $\Pcal$, let
$\mathcal N^P$ be the ideal of all subsets of $P$-null sets in
$\mathcal F_T^0$. Extend $P$ to its completion and define
\[
\mathcal F_T^P=\mathcal F_T^0\vee\mathcal N^P,\qquad
\mathcal F_t^P=\bigcap_{t<u\le T}(\mathcal F_u^0\vee\mathcal N^P)
\quad(t<T),\qquad \F^P=(\mathcal F_t^P)_{0\le t\le T}.
\]
The filtration is complete and right-continuous, and $\xi$ is
$\mathcal F_0^P$-measurable. Stochastic integrands are predictable with
respect to $\F^P$ (and with respect to
$\mathcal P(\F^P)\otimes\mathcal B(E)$ for $U$). All martingale statements
below refer to this augmented filtration, separately under each $P$. We assume
that $B$ is a continuous local martingale and that
\begin{align}
\langle B\rangle_t&=\int_0^t a_s^Pds,\qquad
\underline a\le a_s^P\le\overline a,\quad dt\otimes P\text{-a.e.},
\label{eq:qv}\\
\widetilde N^P(dt,de)&:=N(dt,de)-\Pi(de)dt,
\label{eq:compensated}
\end{align}
where $\underline a,\overline a\in\mathbb S_m^{>0}$ do not depend on $P$,
and the matrix inequalities are understood in the Loewner order. Under each
$P$, $N$ is a Poisson random measure relative to $\F^P$ with compensator
exactly $\Pi(de)dt$. Finally, every square-integrable local martingale under
$P$ admits the representation
\begin{equation}
M_t=M_0+\int_0^t Z_s\,dB_s+
\int_0^t\int_EU_s(e)\widetilde N^P(ds,de).
\label{eq:PRP}
\end{equation}
This representation property is an additional assumption on the admissible
models; it does not follow from the usual augmentation alone. Without it, a
BSDE on a more general filtration may require an additional orthogonal
martingale. Our results concern the representation \eqref{eq:PRP}, without
such a component.

A family $\Pcal$ is called dominated if there exists a probability measure
$Q$ such that $P\ll Q$ for all $P\in\Pcal$, and non-dominated otherwise.
Variation of $a^P$ indeed allows non-domination, but does not imply it by
itself. Here is a precise example. Let $\mathfrak A$ be an uncountable set of
positive definite matrices contained in $[\underline a,\overline a]$. For
$a\in\mathfrak A$, let $P^a$ be the canonical law under which $\xi=0$,
$B=a^{1/2}W$, while $N$ is an independent Poisson random measure with
compensator $\Pi(de)dt$. Define
\[
Q_n(B):=\sum_{j=1}^{2^n}
(B_{jT2^{-n}}-B_{(j-1)T2^{-n}})
(B_{jT2^{-n}}-B_{(j-1)T2^{-n}})^\top,
\qquad E_a:=\{\lim_nQ_n(B)=Ta\}.
\]
These sets are defined by the same Borel maps under all models. Under $P^a$,
independent Gaussian increments yield
$\E^{P^a}\|Q_n-Ta\|^2\le C_aT^22^{-n}$. For every $\epsilon>0$,
\[
\sum_nP^a(\|Q_n-Ta\|>\epsilon)
\le C_aT^2\epsilon^{-2}\sum_n2^{-n}<\infty.
\]
Applying Borel--Cantelli to $\epsilon=1/j$ gives $P^a(E_a)=1$. Since $T>0$,
the sets $E_a$ are pairwise disjoint. If a probability measure $Q$ dominated
all $P^a$, every such set would have strictly positive $Q$-mass, which is
impossible for an uncountable family of disjoint measurable sets. Hence
$\{P^a:a\in\mathfrak A\}$ is non-dominated. Throughout the paper,
\eqref{eq:PRP} is imposed model by model, and no common aggregator is
constructed.

The canonical process $B$ is therefore not assumed to be a standard Brownian
motion. For each $P$, the matrix $(a_t^P)^{-1/2}$ is well defined
$dt\otimes P$-a.e. and
\[
W_t^P:=\int_0^t(a_s^P)^{-1/2}\,dB_s
\]
is a continuous local martingale with
$\langle W^P\rangle_t=tI_m$. Lévy's characterization theorem shows that
$W^P$ is a Brownian motion under $P$; conversely,
$B_t=\int_0^t(a_s^P)^{1/2}dW_s^P$. This identification is strictly model by
model.

For $z\in\R^{1\times m}$ and $u:E\to\R$, set
\[
|z|_{a_t^P}^2:=za_t^Pz^\top,\qquad
\|u\|_\Pi^2:=\int_E|u(e)|^2\Pi(de).
\]
Under $P$, define
\begin{align*}
\mathbb S^2(P)&=\{Y\text{ adapted c\`adl\`ag}:
\E^P\sup_{t\le T}|Y_t|^2<\infty\},\\
\mathbb H_B^2(P)&=\{Z\text{ predictable}:
\E^P\int_0^T|Z_t|_{a_t^P}^2dt<\infty\},\\
\mathbb H_\Pi^2(P)&=\{U\text{ predictable}:
\E^P\int_0^T\|U_t\|_\Pi^2dt<\infty\}.
\end{align*}
For $p\ge2$, $\mathbb S^p(P;\R^n)$ is defined similarly, with norm
$(\E^P\sup_{t\le T}|V_t|^p)^{1/p}$ and values in $\R^n$. The expression
``uniform in the model'' means that the supremum of the corresponding norms
over $P\in\Pcal$ is finite. It does not mean that the processes admit a
version common to all models.

A strong solution under $P$ means here a solution adapted to the completed
filtration generated by $(\xi,B,N)$ and constructed from these given
coordinates. All stochastic identities are understood up to indistinguishability
under the fixed model.

Families of solutions may be collected in the spaces
\[
\mathfrak S^2(\Pcal)=\{(V^P)_P:\sup_P\|V^P\|_{\mathbb S^2(P)}<\infty\},
\]
with analogous definitions for $\mathfrak H_B^2$ and
$\mathfrak H_\Pi^2$. The norm is the supremum of the modelwise norms. These
spaces are complete: a Cauchy sequence converges in each complete space for
fixed $P$, the uniform Cauchy inequality then passes to the limit for each
$P$, and finally to the supremum. This construction requires neither
measurability in $P$ nor a common version of the processes.

\begin{lemma}[Uniform martingale estimates]\label{lem:mart}
There exists $C<\infty$, independent of $P$, such that
\begin{align}
\E^P\sup_{t\le T}\left|\int_0^tZ_s\,dB_s\right|^2
&\le C\E^P\int_0^T|Z_s|_{a_s^P}^2ds,\label{eq:BDGB}\\
\E^P\sup_{t\le T}\left|\int_0^t\int_EU_s(e)
\widetilde N^P(ds,de)\right|^2
&\le C\E^P\int_0^T\|U_s\|_\Pi^2ds.\label{eq:BDGN}
\end{align}
The two stochastic integrals are orthogonal in $L^2(P)$.
\end{lemma}

\begin{proof}
The bracket of the first martingale is
$\int_0^\cdot|Z_s|_{a_s^P}^2ds$. The Burkholder--Davis--Gundy inequality
gives \eqref{eq:BDGB}. For the second martingale, the isometry for
compensated integrals followed by Doob's inequality gives
\[
\E^P\sup_{t\le T}\left|\int_0^t\int_EU_s(e)\widetilde N^P(ds,de)\right|^2
\le4\E^P\int_0^T\int_E|U_s(e)|^2\Pi(de)ds.
\]
A continuous martingale and a purely discontinuous martingale have zero
quadratic covariation; orthogonality follows by localization and then by
$L^2(P)$ convergence.
\end{proof}

\begin{lemma}[Martingale inequality of order $p$]\label{lem:martp}
Let $p\ge2$ and
\[
M_t=\int_0^t Z_s\,dB_s+
\int_0^t\int_E U_s(e)\widetilde N^P(ds,de).
\]
There exist $0<c_p\le C_p<\infty$, independent of $P$, such that
\begin{align}
c_p\E^P\Bigg[&\left(\int_0^T|Z_s|_{a_s^P}^2ds\right)^{p/2}
+\left(\int_0^T\|U_s\|_\Pi^2ds\right)^{p/2}
+\int_0^T\|U_s\|_{L_\Pi^p}^pds\Bigg]
&\le \E^P\sup_{t\le T}|M_t|^p,\label{eq:martplower}\\
\E^P\sup_{t\le T}|M_t|^p
&\le C_p\E^P\Bigg[
\left(\int_0^T|Z_s|_{a_s^P}^2ds\right)^{p/2}
+\left(\int_0^T\|U_s\|_\Pi^2ds\right)^{p/2}
+\int_0^T\|U_s\|_{L_\Pi^p}^pds\Bigg].\label{eq:martpupper}
\end{align}
The inequalities hold first for simple integrands and, by completion,
whenever one of the corresponding sides is finite.
\end{lemma}

\begin{proof}
Write $M^c=\int Z\,dB$ and $M^d=\int U\,d\widetilde N^P$. For the
continuous part, the two-sided BDG inequality gives
\[
\E^P\sup_{t\le T}|M_t^c|^p
\asymp_p
\E^P\left(\int_0^T|Z_s|_{a_s^P}^2ds\right)^{p/2}.
\]
For the purely discontinuous part, the Bichteler--Jacod inequality (see, e.g., \cite{Dirksen2014}) gives
\begin{align*}
\E^P\sup_{t\le T}|M_t^d|^p
\asymp_p{}&
\E^P\left(\int_0^T\|U_s\|_\Pi^2ds\right)^{p/2}
+\E^P\int_0^T\|U_s\|_{L_\Pi^p}^pds;
\end{align*}
see \cite[Theorem~1.1]{Dirksen2014}. The constants depend only on $p$
because the compensator $\Pi(de)dt$ is identical under all models.

The upper bound for $M=M^c+M^d$ follows from the triangle inequality. For the
lower bound, BDG applied to $M$ and the pathwise inequality
$[M^c]_T\le[M]_T$ give
\[
\E^P\sup_{t\le T}|M_t^c|^p
\le C_p\E^P[M^c]_T^{p/2}
\le C_p\E^P[M]_T^{p/2}
\le C_p\E^P\sup_{t\le T}|M_t|^p.
\]
Since $M^d=M-M^c$, we also obtain
$\E^P\sup|M^d|^p\le C_p\E^P\sup|M|^p$. The two preceding equivalences then
yield \eqref{eq:martplower}. The uniform ellipticity bound
$\underline a\le a^P\le\overline a$ ensures that passage between $|Z|$ and
$|Z|_{a^P}$ is uniform in $P$. Finally, approximation by simple integrands,
Fatou's lemma for lower bounds, and completeness for upper bounds give the
claimed extension.
\end{proof}

\subsection{Wasserstein distance and the law-dependent operator}

For $\mu,\nu\in\mathcal P_2(\R^n)$, let $\operatorname{Cpl}(\mu,\nu)$
denote the set of probability measures on $\R^n\times\R^n$ with marginals
$\mu,\nu$. We use the distance
\[
W_2(\mu,\nu)^2=\inf_{\pi\in\operatorname{Cpl}(\mu,\nu)}
\int|x-y|^2\pi(dx,dy).
\]

For $\mu\in\mathcal P_2(\R^q)$, write
$M_2(\mu)^2=\int|x|^2\mu(dx)$. If $V,V'\in L^2(P)$, then
\begin{equation}
\Wtwo^2(\Lcal^P(V),\Lcal^P(V'))\le\E^P|V-V'|^2.
\label{eq:coupling}
\end{equation}

\begin{lemma}[Law flows and representatives]\label{lem:lawmeas}
For a jointly measurable $V\in L^2(dt\otimes P;\R^n)$, there exists a Borel
flow $m^V:[0,T]\to\mathcal P_2(\R^n)$ equal to $\Lcal^P(V_t)$ for almost
every $t$. Its equivalence class does not depend on the representative of
$V$, and
\[
\int_0^T W_2^2(m_t^V,m_t^{V'})dt
\le\E^P\int_0^T|V_t-V_t'|^2dt.
\]
\end{lemma}
\begin{proof}
Tonelli's theorem gives measurability of $q(t)=\E^P|V_t|^2$ and
$\int_0^Tq(t)dt<\infty$. On the null set $\{q=\infty\}$, set
$m_t^V=\delta_0$; elsewhere take the law of $V_t$. For every bounded
continuous function $h$, Fubini's theorem gives measurability of
$t\mapsto\E^Ph(V_t)$. A countable determining family for probability
measures on $\R^n$ therefore yields measurability for the weak topology. The
second moment is also measurable, hence the flow is Borel for $W_2$. For
almost every $t$, the joint law of $(V_t,V_t')$ is a coupling, so
$W_2^2(m_t^V,m_t^{V'})\le\E^P|V_t-V_t'|^2$. Tonelli's theorem yields the
integrated inequality. If $V=V'$ in $L^2(dt\otimes P)$, the right-hand side
vanishes, and the two flows coincide almost everywhere. In particular, the
backward fixed-point map involving $\Lcal^P(y_t)$ is well defined on
$L^2$-equivalence classes.
\end{proof}

Let $\phi:\R^d\to\R$ be convex, finite, lower semicontinuous, and
$\alpha_A$-strongly convex. Let $K:\R^d\times\R^d\to\R$ be such that
$K(\cdot,y)$ is convex and $C^1$, with
\begin{align}
|\nabla_xK(x,y)-\nabla_xK(x',y')|
&\le L_K(|x-x'|+|y-y'|),\label{eq:Klip}\\
|\nabla_xK(x,y)|&\le C_K(1+|x|+|y|).\label{eq:Kgrowth}
\end{align}
Set
\begin{equation}
F_\mu(x):=\int\nabla_xK(x,y)\mu(dy),\qquad
A(x,\mu):=\partial\phi(x)+F_\mu(x).
\label{eq:A}
\end{equation}

\paragraph{Associated potential and minimal-norm selection.}
To make the potential explicit without imposing an additional bound on
$K(0,y)$, define
\[
\widehat K(x,y)=K(x,y)-K(0,y),\qquad
\Phi(x,\mu)=\phi(x)+\int\widehat K(x,y)\mu(dy).
\]
By the fundamental theorem of calculus,
\[
\widehat K(x,y)=\int_0^1\langle\nabla_xK(tx,y),x\rangle dt,
\quad
|\widehat K(x,y)|\le C_K|x|(1+|x|+|y|).
\]
Hence the integral is finite for $\mu\in\mathcal P_2$. On each ball
$|x|\le R$, the derivatives are dominated by $C_K(1+R+|y|)$, which is
integrable. Passing to the limit in the difference quotients by dominated
convergence gives
\[
\nabla_x\int\widehat K(x,y)\mu(dy)=F_\mu(x),\qquad
\partial_x\Phi(x,\mu)=\partial\phi(x)+F_\mu(x)=A(x,\mu).
\]
The finite convex function $\phi$ has a subgradient at every interior point
of its domain, here all of $\R^d$. Thus each $A(x,\mu)$ is nonempty, closed,
and convex. If $(v_n)$ minimizes the norm on this set, it is bounded; a
subsequence converges in $\R^d$, and its limit belongs to $A(x,\mu)$ and
attains the minimum. Two distinct minimizers $v,w$ are impossible because
\[
\left|\frac{v+w}{2}\right|^2
=\frac{|v|^2+|w|^2}{2}-\frac{|v-w|^2}{4}.
\]
This justifies the definition of $A^0$. Adding a finite function
$\mathcal R(\mu)$ to $\Phi$ does not change $\partial_x\Phi$.

\begin{proposition}[Resolvent and Yosida approximation]\label{prop:Yosida}
For each $\mu$, $A(\cdot,\mu)$ is maximal and
$\alpha_A$-strongly monotone. If
\[
J_\lambda(x,\mu)=(I+\lambda A(\cdot,\mu))^{-1}(x),\qquad
A^\lambda(x,\mu)=\lambda^{-1}(x-J_\lambda(x,\mu)),
\]
then
\begin{align}
|J_\lambda(x,\mu)-J_\lambda(x',\mu')|
&\le |x-x'|+\lambda L_K\Wtwo(\mu,\mu'),\label{eq:Jlip}\\
|A^\lambda(x,\mu)-A^\lambda(x',\mu')|
&\le\lambda^{-1}|x-x'|+L_K\Wtwo(\mu,\mu'),\label{eq:AYlip}\\
|A^\lambda(x,\mu)|&\le|A^0(x,\mu)|,\qquad
A^\lambda(x,\mu)\to A^0(x,\mu).\label{eq:Aconv}
\end{align}
\end{proposition}

\begin{proof}
The potential $\Phi(\cdot,\mu)$ defined above is convex, finite, and
$\alpha_A$-strongly convex. For $x\in\R^d$, $\lambda>0$, consider
\[
Q_{x,\lambda}(y)=\Phi(y,\mu)+\frac{|x-y|^2}{2\lambda}.
\]
If $v_0\in\partial\Phi(0,\mu)$, the inequality
$\Phi(y,\mu)\ge\Phi(0,\mu)+\langle v_0,y\rangle$ shows that
$Q_{x,\lambda}(y)\to\infty$ as $|y|\to\infty$. Continuity and strict
convexity give a unique minimizer $y$. The minimum condition for a convex
function reads
\[
0\in\partial\Phi(y,\mu)+\lambda^{-1}(y-x),\qquad
x-y\in\lambda A(y,\mu).
\]
Thus the resolvent is defined everywhere. To verify maximality, let $(x,v)$
be monotone with respect to the whole graph of $A(\cdot,\mu)$. Take
$y=J_\lambda(x+\lambda v,\mu)$ and
$w=(x+\lambda v-y)/\lambda\in A(y,\mu)$. Then
\[
0\le\langle v-w,x-y\rangle=-\lambda^{-1}|x-y|^2.
\]
Hence $y=x$ and $v=w\in A(x,\mu)$, so the graph admits no proper monotone
extension. Strong monotonicity follows by adding the two subgradient
inequalities for the strongly convex function $\phi$, together with the
monotonicity of $F_\mu$:
\[
\langle v-v',x-x'\rangle\ge\alpha_A|x-x'|^2,
\qquad v\in A(x,\mu),\quad v'\in A(x',\mu).
\]
This is also the classical framework of \cite{Brezis1973}. For a coupling
$\pi$ of $\mu,\mu'$,
\[
|F_\mu(x)-F_{\mu'}(x')|
\le L_K|x-x'|+L_K\int|y-y'|\pi(dy,dy'),
\]
hence
\[
|F_\mu(x)-F_{\mu'}(x')|
\le L_K(|x-x'|+\Wtwo(\mu,\mu')).
\]
Write $y=J_\lambda(x,\mu)$ and $y'=J_\lambda(x',\mu')$. Subtracting the two
resolvent identities and taking the scalar product with $y-y'$, the monotone
terms are nonnegative and yield
\[
|y-y'|^2\le
\bigl(|x-x'|+\lambda L_K\Wtwo(\mu,\mu')\bigr)|y-y'|.
\]
This proves \eqref{eq:Jlip}. At fixed law, monotonicity gives
\[
\lambda|A^\lambda(x,\mu)-A^\lambda(x',\mu)|^2
\le\langle A^\lambda(x,\mu)-A^\lambda(x',\mu),x-x'\rangle.
\]
Thus the Lipschitz constant in $x$ is $\lambda^{-1}$. At fixed state,
\eqref{eq:Jlip} gives
\[
|A^\lambda(x,\mu)-A^\lambda(x,\mu')|
=\lambda^{-1}|J_\lambda(x,\mu)-J_\lambda(x,\mu')|
\le L_K\Wtwo(\mu,\mu').
\]
The triangle inequality between $(x,\mu)$, $(x',\mu)$, and $(x',\mu')$ then
proves \eqref{eq:AYlip}.

Finally, if $a^0=A^0(x,\mu)$, monotonicity between
$(J_\lambda x,A^\lambda x)$ and $(x,a^0)$ gives
$|A^\lambda x|^2\le\langle A^\lambda x,a^0\rangle$. Hence
$|A^\lambda x|\le|a^0|$ and $J_\lambda x\to x$. Every cluster point of
$A^\lambda x$ belongs to the closed graph of $A$ and has norm no larger than
that of $a^0$. Uniqueness of the minimal-norm selection forces convergence to
$a^0$.
\end{proof}

For every $\lambda>0$, estimates \eqref{eq:Jlip} and \eqref{eq:AYlip} show
that $(x,\mu)\mapsto J_\lambda(x,\mu)$ and
$(x,\mu)\mapsto A^\lambda(x,\mu)$ are continuous on
$\R^d\times(\mathcal P_2(\R^d),\Wtwo)$, hence Borel measurable. Therefore
$A^\lambda(X_{t-},\Lcal(X_{t-}))$ is predictable whenever $X$ is adapted and
c\`adl\`ag. The pointwise convergence $A^{1/n}\to A^0$ also shows that $A^0$
is Borel measurable. Thus the measurability clause in \ref{HD} follows from
the operator construction; linear growth remains an additional assumption in
the general framework. One may also verify directly that the joint graph is
closed. If $(x_n,\mu_n,v_n)\to(x,\mu,v)$ in
$\R^d\times(\mathcal P_2(\R^d),W_2)\times\R^d$ and
$v_n\in A(x_n,\mu_n)$, then for every $z\in\R^d$,
\[
\phi(z)\ge\phi(x_n)+\langle v_n-F_{\mu_n}(x_n),z-x_n\rangle.
\]
The finite convex function $\phi$ is continuous; \eqref{eq:Klip} allows us
to pass to the limit and obtain
$\phi(z)\ge\phi(x)+\langle v-F_\mu(x),z-x\rangle$. Hence
$v\in A(x,\mu)$ and the joint graph is Borel. Measurability of the
minimal-norm selection follows instead from the pointwise limit above, without
assuming continuity across the interfaces.

\paragraph{ALA prototype.}
For $r\in\R$, set $r_+=\max\{r,0\}$ and $r_-=\max\{-r,0\}$. We use
\begin{align}
\phi_{\rm ALA}(x)
&=\sum_{k=1}^d\{\kappa_+(x_k)_++\kappa_-(x_k)_-\}
+\frac{\alpha_A}{2}|x|^2,\qquad
\kappa_->\kappa_+>0,\quad\alpha_A>0,\label{eq:ALAphi}\\
K(x,y)&=\eta\{\sqrt{1+|x-y|^2}-1\},\qquad \eta>0.\label{eq:ALAK}
\end{align}
For this prototype, we write
\[
A_{\rm ALA}(x,\mu):=\partial\phi_{\rm ALA}(x)+F_\mu(x),
\qquad
A_{\rm ALA}^0(x,\mu)
:=\operatorname*{argmin}_{\gamma\in A_{\rm ALA}(x,\mu)}|\gamma|.
\]
To lighten notation, the generic symbol $A^0$ is retained in the theorems and
means $A_{\rm ALA}^0$ when they are applied to this prototype. Here,
\[
F_\mu(x)=\eta\int_{\R^d}
\frac{x-y}{\sqrt{1+|x-y|^2}}\,\mu(dy).
\]
For $1\le k\le d$, the minimal-norm selection is explicitly
\begin{equation}
[A^0(x,\mu)]_k=\alpha_Ax_k+[F_\mu(x)]_k+
\begin{cases}
\kappa_+,&x_k>0,\\
-\kappa_-,&x_k<0,\\
\operatorname{proj}_{[-\kappa_-,\kappa_+]}
\bigl(-[F_\mu(x)]_k\bigr),&x_k=0.
\end{cases}
\label{eq:A0explicit}
\end{equation}
When $x_k=0$, the full coordinate satisfies
\[
[A^0(x,\mu)]_k
=\operatorname{proj}_{[F_\mu(x)]_k+[-\kappa_-,\kappa_+]}(0).
\]
The singular set is
\begin{equation}
\Sigma=\bigcup_{k=1}^d\{x:x_k=0\},
\label{eq:Sigma}
\end{equation}
and
\begin{equation}
|A^0(x,\mu)|\le C_A(1+|x|+M_2(\mu)).
\label{eq:A0growth}
\end{equation}
For this prototype, the bound follows directly from $|F_\mu(x)|\le\eta$ and
\eqref{eq:A0explicit}; the term $M_2(\mu)$ may even be omitted from the
right-hand side.

\subsection{Coefficients, controls, and system}

The control $\alpha$ takes values in a closed set $D\subset\R^k$ containing
$0$. To avoid ambiguity arising from different completions of the filtrations,
we require $\alpha$ to admit a version predictable with respect to the raw
canonical filtration generated by $(\xi,B,N)$. The same functional then
defines, by completion, an $\F^P$-predictable control under each model. For
$p\ge2$, its integrability is defined by
\begin{equation}
\sup_{P\in\Pcal}\E^P\int_0^T|\alpha_t|^pdt<\infty.
\label{eq:controlq}
\end{equation}
The set of controls satisfying these conditions is denoted by
$\mathcal A^p$. Well-posedness uses $\alpha\in\mathcal A^2$. Controls used in
the product construction are additionally subject to
Assumption~\ref{Hprod}. This restriction is a cross-model consistency
assumption; it does not construct a control by quasi-sure aggregation.

The coefficients have dimensions
\begin{align*}
b&:[0,T]\times\R^d\times\mathcal P_2(\R^d)\times D\to\R^d,\\
\sigma&:[0,T]\times\R^d\times\mathcal P_2(\R^d)\times D
\to\R^{d\times m},\\
\beta&:[0,T]\times\R^d\times\mathcal P_2(\R^d)\times D\times E
\to\R^d,\\
f&:[0,T]\times\R^d\times\mathcal P_2(\R^d)\times\R
\times\R^{1\times m}\times L_\Pi^2(E;\R)
\times\mathcal P_2(\R)\times D\to\R,\\
g&:\R^d\times\mathcal P_2(\R^d)\to\R.
\end{align*}

\begin{hypothesis}[Forward coefficients]\label{HF}
The functions $b,\sigma,\beta$ are measurable and there exist
$L_F,C_F<\infty$ such that
\begin{align}
&|b(t,x,\mu,\alpha)-b(t,x',\mu',\alpha')|
+\|\sigma(t,x,\mu,\alpha)-\sigma(t,x',\mu',\alpha')\|\nonumber\\
&\quad+\left(\int_E|\beta(t,x,\mu,\alpha,e)
-\beta(t,x',\mu',\alpha',e)|^2\Pi(de)\right)^{1/2}\nonumber\\
&\le L_F\bigl(|x-x'|+\Wtwo(\mu,\mu')+|\alpha-\alpha'|\bigr),
\label{eq:FLip}\\
&|b(t,0,\delta_0,0)|+\|\sigma(t,0,\delta_0,0)\|
+\|\beta(t,0,\delta_0,0,\cdot)\|_{L_\Pi^2}\le C_F.
\label{eq:Fzero}
\end{align}
\end{hypothesis}

\begin{hypothesis}[Backward coefficients]\label{HB}
The functions $f$ and $g$ are Borel measurable. There exist $L_B,C_B<\infty$
such that
\begin{align}
&|f(t,x,\mu,y,z,u,\eta,\alpha)
-f(t,x',\mu',y',z',u',\eta',\alpha')|\nonumber\\
&\le L_B\bigl(
|x-x'|+\Wtwo(\mu,\mu')+|y-y'|+|z-z'|
+\|u-u'\|_\Pi+\Wtwo(\eta,\eta')+|\alpha-\alpha'|
\bigr),\label{eq:fLip}\\
&|f(t,0,\delta_0,0,0,0,\delta_0,0)|\le C_B,\label{eq:fzero}\\
&|g(x,\mu)-g(x',\mu')|
\le L_B\bigl(|x-x'|+\Wtwo(\mu,\mu')\bigr),
\qquad |g(0,\delta_0)|\le C_B.
\label{eq:gLip}
\end{align}
\end{hypothesis}

\begin{hypothesis}[Data and minimal-norm selection]\label{HD}
The initial variable $\xi$ is $\mathcal F_0^P$-measurable under each $P$,
its law does not depend on the model, and
\[
\sup_{P\in\Pcal}\E^P|\xi|^2<\infty.
\]
The minimal-norm selection $(x,\mu)\mapsto A^0(x,\mu)$ is Borel measurable
and satisfies
\[
|A^0(x,\mu)|\le C_A(1+|x|+M_2(\mu)).
\]
Finally, $\varrho\in\R^d$ is fixed.
\end{hypothesis}

\begin{hypothesis}[Higher moments for rates]\label{HqJ}
For a fixed $q>4$, assume $\alpha\in\mathcal A^q$ and
$\sup_P\E^P|\xi|^q<\infty$. For $\mu\in\mathcal P_q(\R^d)$,
\begin{equation}
\int_E|\beta(t,x,\mu,\alpha,e)|^q\Pi(de)
\le C_F^q(1+|x|^q+M_q(\mu)^q+|\alpha|^q).
\label{eq:betaq}
\end{equation}
This assumption is added to the quadratic framework only for higher moments
and particle rates. It is not required for well-posedness or quadratic
stability.
\end{hypothesis}

For $P\in\Pcal$, set
\[
\mu_t^P=\Lcal^P(X_t^P),\quad
\mu_{t-}^P=\Lcal^P(X_{t-}^P),\quad
\eta_t^P=\Lcal^P(Y_t^P).
\]
Under the preceding assumptions, the rigorous formulation of the system
studied below is
\begin{align}
X_t^P
={}&\xi+\int_0^t b(s,X_{s-}^P,\mu_{s-}^P,\alpha_s)ds
+\int_0^t\sigma(s,X_{s-}^P,\mu_{s-}^P,\alpha_s)dB_s\nonumber\\
&+\int_0^t\int_E\beta(s,X_{s-}^P,\mu_{s-}^P,\alpha_s,e)
\widetilde N^P(ds,de),\label{eq:FWD}\\
Y_t^P
={}&g(X_T^P,\mu_T^P)
+\int_t^T f(s,X_{s-}^P,\mu_{s-}^P,Y_s^P,Z_s^P,U_s^P,
\eta_s^P,\alpha_s)ds\nonumber\\
&+\mathsf K_T^P-\mathsf K_t^P
-\int_t^TZ_s^P\,dB_s
-\int_t^T\int_EU_s^P(e)\widetilde N^P(ds,de),\label{eq:BWD}\\
\mathsf K_t^P
={}&\int_0^t\langle\varrho,
A^0(X_{s-}^P,\mu_{s-}^P)\rangle ds.
\label{eq:Kdef}
\end{align}
The use of left limits makes the integrands predictable. Since a c\`adl\`ag
process has at most countably many jump times,
$X_{s-}^P=X_s^P$ for $dt\otimes P$-a.e. $(s,\omega)$.

\begin{definition}[Modelwise solution]
\label{def:modelsolution}
Fix $P\in\Pcal$. A solution of the selected system is a sextuple
\[
(X^P,Y^P,Z^P,U^P,\Gamma^P,\mathsf K^P)
\]
such that
\begin{align*}
&X^P\in\mathbb S^2(P;\R^d),\qquad
Y^P,\mathsf K^P\in\mathbb S^2(P),\\
&Z^P\in\mathbb H_B^2(P),\qquad
U^P\in\mathbb H_\Pi^2(P),\qquad
\Gamma^P\in L^2(dt\otimes P;\R^d).
\end{align*}
The processes $X^P,Y^P$ are adapted and c\`adl\`ag, $Z^P,U^P$ are
predictable, $\Gamma^P$ is predictable, and $\mathsf K^P$ is adapted,
continuous, and absolutely continuous in time. They satisfy
\begin{align*}
\mu_t^P&=\Lcal^P(X_t^P),&
\eta_t^P&=\Lcal^P(Y_t^P),\\
\Gamma_t^P&=A^0(X_{t-}^P,\mu_{t-}^P)
&&dt\otimes P\text{-a.e.},
\end{align*}
and identities \eqref{eq:FWD}--\eqref{eq:Kdef} for every
$t\in[0,T]$, up to indistinguishability under $P$. All stochastic integrals
are taken in their respective square-integrable spaces. A regularized
solution with parameter $\lambda>0$ is defined analogously by replacing
$A^0$ by $A^\lambda$ in $\Gamma^P$ and in \eqref{eq:Kdef}.
\end{definition}

\paragraph{Dimension of the backward variable.}
The choice $Y\in\R$ corresponds to a scalar evaluation of the vector state
$X\in\R^d$. The conditional identity is
\[
Y_t^P=\E^P\left[g(X_T^P,\mu_T^P)+\int_t^T f_s^Pds
+\mathsf K_T^P-\mathsf K_t^P\mid\mathcal F_t^P\right],
\]
where $f_s^P$ denotes the generator evaluated along the solution in
\eqref{eq:BWD}. This identity follows because the increments of the two
martingales have zero conditional expectation. When $f$ depends on
$(Y,Z,U)$, the identity is implicit. The projection
$\langle\varrho,A^0\rangle$ ensures scalar dimensional compatibility. It does
not transform monotonicity in $x$,
\[
\langle v-v',x-x'\rangle\ge0,
\qquad v\in A(x,\mu),\quad v'\in A(x',\mu),
\]
into monotonicity of the generator in $Y$. For example, when $d=1$,
$A(x)=x$ and $\varrho=-1$, the projection is $-x$. The source is exogenous
to the backward fixed point because the system is triangular. A formulation
with $Y\in\R^k$ would require $g,f\in\R^k$,
$Z\in\R^{k\times m}$, $U\in L_\Pi^2(E;\R^k)$, and a linear map
$R:\R^d\to\R^k$ in place of $\varrho^\top$. The backward law would then
belong to $\mathcal P_2(\R^k)$; the empirical estimates currently used in
one dimension would have to be replaced by their $k$-dimensional versions.
Under the corresponding Lipschitz assumptions, the quadratic estimate uses
$|Y|^2$, $\operatorname{Tr}(ZaZ^\top)$, and
$\int_E|U(e)|^2\Pi(de)$; the contraction construction encounters no spectral
monotonicity obstruction. This vector-valued formulation is not a result of
the present paper.

\begin{remark}[Nature of $\mathsf K$]
The operator acts on the forward state. The process $\mathsf K^P$ is
absolutely continuous and its total variation satisfies
\[
\operatorname{Var}(\mathsf K^P)_T
\le |\varrho|\int_0^T|A^0(X_{s-}^P,\mu_{s-}^P)|ds.
\]
It need not be increasing because the scalar product with $\varrho$ may
change sign. It does not carry the jumps of $Y^P$; those come from the Poisson
integral. Hence \eqref{eq:BWD} is neither a reflected equation nor a BSVI in
the variable $Y$.
\end{remark}

The dependencies among the results are summarized as follows:
\begin{center}
\small
\begin{tabular}{p{0.28\textwidth}p{0.63\textwidth}}
\toprule
Result & Assumptions and results used \\
\midrule
Well-posedness & \ref{HF}, \ref{HB}, \ref{HD}, control in
$\mathcal A^2$, and representation \eqref{eq:PRP}. \\
Stability & Quadratic solutions already constructed, with the same integrators. \\
Yosida approximation at fixed $P$ & Domination $|A^\lambda|\le|A^0|$;
for uniformity, \ref{HY} or a criterion implying it. \\
Particle rates & \ref{HqJ}, \ref{Hprod}, moments of order $q$, followed by
the empirical Wasserstein estimate. \\
Joint ALA bound & Under \ref{HoccD}, baseline rate; under \ref{Hocc2},
enhanced cumulative control and exponent $\theta/(\theta+1)$. \\
Occupation & \ref{HellJ}, \ref{HtruncJ}, \ref{HmultJ}, or \ref{HconjJ},
depending on the mechanism considered. \\
\bottomrule
\end{tabular}
\end{center}
The construction of solutions does not require occupation assumptions. The
criteria introduced later verify these assumptions for the rates, without
entering the well-posedness fixed-point arguments.
\section{Modelwise well-posedness}

\subsection{Forward equation with jumps}

Existence and uniqueness for Lipschitz Brownian--Poisson SDEs belong to the
theory developed in \cite[Chapter~6]{Applebaum2009}; well-posedness and
propagation of chaos for McKean--Vlasov equations with jumps are studied in
\cite{Erny2022}. These results are not claimed as new here. The estimates and
the fixed-point argument on the law flow below specify the adaptation to predictable
controls and to constants uniform over $\Pcal$, which cannot be inferred from
a citation without checking the assumptions. The backward result is
constructed explicitly in Theorem~\ref{thm:Bwell}, before its stability
estimate is established.

\begin{lemma}[Integral stability of the forward equation]\label{lem:Fstab}
Under \ref{HF}, fix $P\in\Pcal$. Let
$m,\bar m\in C([0,T];\mathcal P_2(\R^d))$, let
$\xi,\bar\xi\in L^2(P;\R^d)$ be two $\mathcal F_0^P$-measurable initial
data, and let $\alpha,\bar\alpha\in L^2(dt\otimes P)$ be two predictable
$D$-valued controls. Let $X,\bar X\in\mathbb S^2(P;\R^d)$ be the solutions
with prescribed flows $m,\bar m$, these data, the same $P$, and the same
noises. Then, for every $t\in[0,T]$,
\begin{align}
\E^P\sup_{r\le t}|X_r-\bar X_r|^2
\le C\Bigl(&\E^P|\xi-\bar\xi|^2
+\E^P\int_0^t|\alpha_s-\bar\alpha_s|^2ds\nonumber\\
&+\int_0^t\Wtwo^2(m_s,\bar m_s)ds\Bigr).
\label{eq:Fstab}
\end{align}
One may take $C=\max\{4,k_T\}e^{k_TT}$, where
\[
k_T=3L_F^2\{4T+16\lambda_{\max}(\overline a)+16\}.
\]
This constant is independent of $P$, the flows, and the data.
\end{lemma}

\begin{proof}
\emph{Step 1: predictable differences and integrability.}
Set $\Delta\xi=\xi-\bar\xi$, $\Delta\alpha=\alpha-\bar\alpha$,
$\Delta X=X-\bar X$, and
\begin{align*}
\Delta b_s&=b(s,X_{s-},m_s,\alpha_s)
-b(s,\bar X_{s-},\bar m_s,\bar\alpha_s),\\
\Delta\sigma_s&=\sigma(s,X_{s-},m_s,\alpha_s)
-\sigma(s,\bar X_{s-},\bar m_s,\bar\alpha_s),\\
\Delta\beta_s(e)&=\beta(s,X_{s-},m_s,\alpha_s,e)
-\beta(s,\bar X_{s-},\bar m_s,\bar\alpha_s,e).
\end{align*}
Left limits are used for predictable integrands. The deterministic continuous
flows have uniformly bounded second moments on $[0,T]$. By \eqref{eq:FLip},
each of the three quantities $|\Delta b_s|^2$,
$\|\Delta\sigma_s\|^2$, and $\|\Delta\beta_s\|_\Pi^2$ is bounded by
\begin{equation}
3L_F^2\{|\Delta X_{s-}|^2+W_2^2(m_s,\bar m_s)
+|\Delta\alpha_s|^2\}.
\label{eq:Fdiffpoint}
\end{equation}
The right-hand side is integrable under $ds\otimes P$, because both solutions
belong to $\mathbb S^2$. The following martingale integrals are therefore
square integrable, and the difference satisfies
\[
\Delta X_t=\Delta\xi+D_t+M_t^B+M_t^J,
\quad D_t=\int_0^t\Delta b_sds,
\]
\[
M_t^B=\int_0^t\Delta\sigma_s dB_s,
\qquad M_t^J=\int_0^t\int_E\Delta\beta_s(e)\widetilde N^P(ds,de).
\]
Here $\widetilde N^P=N-\Pi(de)ds$ is the compensated measure. The argument
would not apply verbatim to an uncompensated integral against $N$.

\emph{Step 2: estimate of each integral.}
Pathwise Cauchy--Schwarz yields
\[
\sup_{r\le t}|D_r|^2
\le\sup_{r\le t}\left(r\int_0^r|\Delta b_s|^2ds\right)
\le t\int_0^t|\Delta b_s|^2ds.
\]
The norm of a vector-valued martingale is a submartingale; Doob's $L^2$
inequality and the continuous isometry give
\begin{align*}
\E^P\sup_{r\le t}|M_r^B|^2
&\le4\E^P|M_t^B|^2\\
&=4\E^P\int_0^t\operatorname{Tr}
(\Delta\sigma_s a_s^P\Delta\sigma_s^\top)ds\\
&\le4\lambda_{\max}(\overline a)\E^P\int_0^t\|\Delta\sigma_s\|^2ds.
\end{align*}
The last inequality uses $a_s^P\le\overline a$ and the Frobenius matrix
norm. For the discontinuous part, Doob's inequality and the Poisson isometry
give
\begin{align*}
\E^P\sup_{r\le t}|M_r^J|^2
&\le4\E^P|M_t^J|^2
=4\E^P\int_0^t\int_E|\Delta\beta_s(e)|^2\Pi(de)ds.
\end{align*}
These identities remain valid when $\Pi(E)=\infty$: the integrals are
constructed by completion of simple integrands in $L^2$, and both the
isometry and Doob's inequality pass to the limit. The constant $4$ in Doob's
inequality is common to all $P$; no independence between the two martingales
is required to estimate them separately.

\emph{Step 3: closed integral inequality.}
Set
\[
u_P(t)=\E^P\sup_{r\le t}|\Delta X_r|^2,
\qquad q_P(s)=W_2^2(m_s,\bar m_s)+\E^P|\Delta\alpha_s|^2.
\]
The function $u_P$ is measurable, nonnegative, and bounded on $[0,T]$. The
vector inequality $|v_1+\cdots+v_4|^2\le4\sum_{j=1}^4|v_j|^2$ and Step 2
imply
\begin{align*}
u_P(t)\le{}&4\E^P|\Delta\xi|^2+4T\E^P\int_0^t|\Delta b_s|^2ds\\
&+16\lambda_{\max}(\overline a)\E^P\int_0^t\|\Delta\sigma_s\|^2ds
+16\E^P\int_0^t\|\Delta\beta_s\|_\Pi^2ds.
\end{align*}
For $s>0$, $|\Delta X_{s-}|^2\le\sup_{r\le s}|\Delta X_r|^2$.
Applying \eqref{eq:Fdiffpoint} and Tonelli's theorem gives
\begin{equation}
u_P(t)\le h_P(t)+k_T\int_0^tu_P(s)ds,
\quad h_P(t)=4\E^P|\Delta\xi|^2+k_T\int_0^tq_P(s)ds.
\label{eq:FbeforeGronwall}
\end{equation}
The function $h_P$ is finite and nondecreasing.

\emph{Step 4: explicit solution of the Gronwall inequality.}
We record the argument in a form that can be reused below. If $u\ge0$ is
integrable, $h\ge0$ is nondecreasing, and
$u(t)\le h(t)+k\int_0^tu(s)ds$, $k\ge0$, define
$V(t)=\int_0^tu(s)ds$. Then $V$ is absolutely continuous and, for a.e. $t$,
\[
\frac{d}{dt}(e^{-kt}V(t))
=e^{-kt}(V'(t)-kV(t))\le e^{-kt}h(t).
\]
Integrating from $0$ to $t$ and substituting into the inequality for $u$ gives
\begin{equation}
V(t)\le\int_0^te^{k(t-s)}h(s)ds,
\qquad u(t)\le h(t)+k\int_0^te^{k(t-s)}h(s)ds\le h(t)e^{kt}.
\label{eq:GronwallExplicit}
\end{equation}
The last inequality uses $h(s)\le h(t)$ for $s\le t$; if $k=0$, it simply
reads $u(t)\le h(t)$. No differentiability of $u$ or $h$ is assumed.
Applied to \eqref{eq:FbeforeGronwall}, this yields
\[
u_P(t)\le e^{k_Tt}\left(4\E^P|\Delta\xi|^2
+k_T\int_0^t\{W_2^2(m_s,\bar m_s)+\E^P|\Delta\alpha_s|^2\}ds\right).
\]
Since $t\le T$, the stated constant gives \eqref{eq:Fstab}. Its expression
involves only $T,L_F,\overline a$, which are common to all models. The
supremum over $P$, when needed, is taken only after this estimate; no
dominating measure is used.
\end{proof}

\begin{corollary}[Stability with random flows]\label{cor:Fstabrandom}
The conclusion of Lemma~\ref{lem:Fstab} remains valid if $m_t$ and $\bar m_t$
are progressively measurable $\mathcal P_2(\R^d)$-valued random variables
satisfying
$\E^P\int_0^T[M_2(m_s)^2+M_2(\bar m_s)^2]ds<\infty$, provided the last term
in \eqref{eq:Fstab} is replaced by
\[
\E^P\int_0^t\Wtwo^2(m_s,\bar m_s)ds.
\]
\end{corollary}

\begin{proof}
The difference has exactly the same integral form as in the preceding proof.
The Lipschitz bounds apply pathwise to $m_s(\omega)$ and $\bar m_s(\omega)$.
Thus, if $D(t)=\E^P\sup_{r\le t}|X_r-\bar X_r|^2$, Cauchy--Schwarz and
Lemma~\ref{lem:mart} give
\[
D(t)\le C\E^P|\Delta\xi|^2
+C\E^P\int_0^t|\Delta\alpha_s|^2ds
+C\int_0^tD(s)ds
+C\E^P\int_0^t\Wtwo^2(m_s,\bar m_s)ds.
\]
Set $h(t)=C\E^P|\Delta\xi|^2+C\E^P\int_0^t
(|\Delta\alpha_s|^2+W_2^2(m_s,\bar m_s))ds$. This function is
nondecreasing; \eqref{eq:GronwallExplicit} yields $D(t)\le e^{Ct}h(t)$.
The stochastic integrands are taken in predictable versions equal
$dt\otimes P$-a.e. to their progressively measurable versions. The
deterministic nature of the flows is not used in this argument.
\end{proof}

\begin{lemma}[Construction of the forward equation with frozen law]\label{lem:frozenfull}
Under \ref{HF}, for $m\in C([0,T];\mathcal P_2(\R^d))$,
$\xi\in L^2(\mathcal F_0^P,P)$, and a predictable $D$-valued control
$\alpha$ satisfying $\E^P\int_0^T|\alpha_t|^2dt<\infty$, equation
\eqref{eq:FWD} with the law replaced by $m$ admits a unique solution
$X^m\in\mathbb S^2(P;\R^d)$. Moreover,
\[
\E^P\sup_{t\le T}|X_t^m|^2\le C\left[1+\E^P|\xi|^2+
\E^P\int_0^T|\alpha_t|^2dt+\int_0^TM_2(m_t)^2dt\right].
\]
\end{lemma}
\begin{proof}
Set $X^0_t=\xi$ and construct $X^{n+1}$ by evaluating the three coefficients
in the integral equation at $(X^n_{t-},m_t,\alpha_t)$. Linear growth shows
inductively that every integral is well defined in $L^2$ and
$X^n\in\mathbb S^2$. For
$D_n(t)=\E^P\sup_{s\le t}|X^{n+1}_s-X^n_s|^2$, the two martingale
inequalities in Lemma~\ref{lem:mart} give
\[
D_n(t)\le C\int_0^t D_{n-1}(s)ds,
\quad
D_n(T)\le D_0(T)\frac{(CT)^n}{n!}.
\]
The series $\sum_n\sqrt{(CT)^n/n!}$ converges by the ratio test. By the
triangle inequality in $\mathbb S^2$, $(X^n)$ is therefore Cauchy. Its limit
$X^m$ has an adapted c\`adl\`ag version: a subsequence converges uniformly
almost surely, and the uniform limit of c\`adl\`ag functions is c\`adl\`ag.
Lipschitz continuity gives
\[
\E^P\int_0^T\left(|b(X^n)-b(X^m)|^2+
\|\sigma(X^n)-\sigma(X^m)\|^2+
\|\beta(X^n)-\beta(X^m)\|_\Pi^2\right)dt\to0,
\]
with all other arguments frozen. Cauchy--Schwarz and the isometries allow
passage to the limit in every integral. For the stated bound, the integral
form yields
\[
\E^P\sup_{s\le t}|X_s^m|^2\le C\left[\E^P|\xi|^2+
\int_0^t(1+M_2(m_s)^2+\E^P|\alpha_s|^2)ds+
\int_0^t\E^P\sup_{u\le s}|X_u^m|^2ds\right].
\]
With $u(t)=\E^P\sup_{s\le t}|X_s^m|^2$ and
$h(t)=C[\E^P|\xi|^2+\int_0^t(1+M_2(m_s)^2+\E^P|\alpha_s|^2)ds]$,
\eqref{eq:GronwallExplicit} gives $u(T)\le h(T)e^{CT}$. Applied to two
solutions with the same data, Lemma~\ref{lem:Fstab} yields
indistinguishability.
\end{proof}

\begin{theorem}[Well-posedness and moments of the forward equation]\label{thm:Fwell}
Let $\alpha\in\mathcal A^2$. Under Assumptions~\ref{HF} and~\ref{HD},
for every $P\in\Pcal$, \eqref{eq:FWD} admits a unique strong c\`adl\`ag
solution. Under Assumption~\ref{HqJ}, moreover,
\begin{equation}
\sup_{P\in\Pcal}\E^P\sup_{t\le T}|X_t^P|^q
\le C\left(1+\sup_{P\in\Pcal}\E^P|\xi|^q
+\sup_{P\in\Pcal}\E^P\int_0^T|\alpha_t|^qdt\right).
\label{eq:Fq}
\end{equation}
\end{theorem}

\begin{proof}
Fix $P\in\Pcal$.  For a deterministic flow
$m\in C([0,T];\mathcal P_2(\R^d))$, freeze the law variable in
\eqref{eq:FWD}.  Picard iteration in $\mathbb S^2(P)$ is well defined by
\ref{HF}--\ref{HD}.  Lemma~\ref{lem:Fstab}, applied to successive iterates,
gives
\[
D_n(t):=\E^P\sup_{r\le t}|X_r^{n+1}-X_r^n|^2
\le C\int_0^tD_{n-1}(s)\,ds,
\]
hence $D_n(T)\le C_0(CT)^n/n!$.  The iterates converge to the unique
frozen-flow solution.

Define
\[
\Phi_P(m)_t=\Lcal^P(X_t^m),\qquad
d_\rho(m,\bar m)^2=
\sup_{t\le T}e^{-\rho t}\Wtwo^2(m_t,\bar m_t).
\]
The natural coupling and Lemma~\ref{lem:Fstab} imply
\[
e^{-\rho t}\Wtwo^2(\Phi_P(m)_t,\Phi_P(\bar m)_t)
\le \frac C\rho d_\rho(m,\bar m)^2.
\]
For $\rho>C$, $\Phi_P$ is a contraction on the complete space of continuous
$\mathcal P_2$-valued flows with the prescribed initial law.  Its fixed point
is the McKean--Vlasov solution; uniqueness follows from the same stability
estimate and Gronwall.

The same estimates with a zero reference process give
\begin{equation}
\sup_P\E^P\sup_{t\le T}|X_t^P|^2
\le C\left(1+\sup_P\E^P|\xi|^2+
\sup_P\E^P\int_0^T|\alpha_t|^2dt\right).
\label{eq:Fsecond}
\end{equation}
Under \ref{HqJ}, the compensated-Poisson BDG estimate used below is
\begin{align}
\E^P\sup_{r\le t}\left|
\int_0^r\int_E\beta_s(e)\widetilde N^P(ds,de)\right|^q
\le C_q\E^P\left[
\left(\int_0^t\int_E|\beta_s(e)|^2\Pi(de)ds\right)^{q/2}
+\int_0^t\int_E|\beta_s(e)|^q\Pi(de)ds\right].
\label{eq:BJ}
\end{align}
Together with the continuous BDG inequality, linear growth, Young and
Gronwall, this yields \eqref{eq:Fq}.  Uniformity in $P$ follows from the
uniform ellipticity bounds and the common compensator.
\end{proof}

\subsection{Existence and uniqueness of the backward equation with jumps}

Well-posedness under discontinuous martingales is studied in the preprint of
Papapantoleon, Saplaouras and Theodorakopoulos (arXiv:2408.13758,
Section~4). For fixed $P$ and $X^P$, set
\begin{align}
F^P(t,\omega,y,z,u,\eta)
&=f(t,X_{t-}^P(\omega),\mu_{t-}^P,y,z,u,\eta,\alpha_t(\omega))
+h_t^P(\omega),\nonumber\\
h_t^P&=\langle\varrho,A^0(X_{t-}^P,\mu_{t-}^P)\rangle.
\label{eq:exogenousgenerator}
\end{align}
For $v=(y,z,u,\eta)$ and $v'=(y',z',u',\eta')$, \ref{HB} gives
\[
|F^P(t,v)-F^P(t,v')|
\le L_B\{|y-y'|+|z-z'|+\|u-u'\|_\Pi+W_2(\eta,\eta')\}.
\]
The source $h^P$ cancels from this difference. The proof below verifies its
integrability and constructs the solution under full dependence on
$u\in L_\Pi^2$, with constants common to all models. The comparison of sources
evaluated along different forward processes is treated separately.

\begin{theorem}[Existence and uniqueness of the jump BSDE along the forward process]
\label{thm:Bwell}
Let $\alpha\in\mathcal A^2$. Under Assumptions~\ref{HF}, \ref{HB},
\ref{HD}, and the representation property \eqref{eq:PRP}, fix
$P\in\Pcal$ and the forward process $X^P$ from Theorem~\ref{thm:Fwell},
with $\mu_t^P=\Lcal^P(X_t^P)$. Define
\[
h_t^P=\langle\varrho,A^0(X_{t-}^P,\mu_{t-}^P)\rangle,
\qquad \mathsf K_t^P=\int_0^th_s^Pds.
\]
There exists a unique triple
\[
(Y^P,Z^P,U^P)\in\mathbb S^2(P)\times\mathbb H_B^2(P)
\times\mathbb H_\Pi^2(P)
\]
satisfying, simultaneously for all $t\in[0,T]$, $P$-a.s.,
\begin{align}
Y_t^P={}&g(X_T^P,\mu_T^P)+\int_t^T
f(s,X_{s-}^P,\mu_{s-}^P,Y_s^P,Z_s^P,U_s^P,
\Lcal^P(Y_s^P),\alpha_s)ds\nonumber\\
&+\mathsf K_T^P-\mathsf K_t^P-
\int_t^TZ_s^PdB_s-\int_t^T\int_EU_s^P(e)\widetilde N^P(ds,de).
\label{eq:Bexist}
\end{align}
Uniqueness of $Y$ holds up to indistinguishability, while uniqueness of
$Z,U$ holds in their energy norms. Moreover,
\begin{align}
\sup_P\E^P\Bigl[\sup_{t\le T}|Y_t^P|^2+
\sup_{t\le T}|\mathsf K_t^P|^2+
\int_0^T\{|Z_t^P|_{a_t^P}^2+\|U_t^P\|_\Pi^2\}dt\Bigr]<\infty.
\label{eq:Bexistbound}
\end{align}
No occupation condition and no moment assumption $q>4$ is required.
\end{theorem}

\begin{proof}
\emph{Step 1: BSDE data after construction of the forward process.}
Fix $P$. Theorem~\ref{thm:Fwell} provides $X=X^P$ and $\mu=\mu^P$. Set
\[
\zeta=g(X_T,\mu_T),\qquad
h_t=\langle\varrho,A^0(X_{t-},\mu_{t-})\rangle,
\qquad
\ell_t=f(t,X_{t-},\mu_{t-},0,0,0,\delta_0,\alpha_t)+h_t.
\]
Measurability of the law flow and of $A^0$, together with predictability of
$X_-$, makes $h$ predictable. By \ref{HB}, \ref{HD}, and
$M_2(\mu_t)^2=\E^P|X_t|^2$,
\begin{align*}
\E^P|\zeta|^2+\E^P\int_0^T(|\ell_t|^2+|h_t|^2)dt
&\le C\left(1+\E^P\sup_{t\le T}|X_t|^2+
\E^P\int_0^T|\alpha_t|^2dt\right)\le C.
\end{align*}
The last constant is uniform in $P$ by \eqref{eq:Fsecond}. Write
$\mathfrak d_P=\E^P|\zeta|^2+
\E^P\int_0^T(|\ell_t|^2+|h_t|^2)dt$.

\emph{Step 2: complete input space.}
For $\gamma\ge2$, consider
\[
\mathcal H_\gamma(P)=L^2_{\rm pred}(dt\otimes P;\R)
\times\mathbb H_B^2(P)\times\mathbb H_\Pi^2(P),
\]
with norm
\[
\|y,z,u\|_\gamma^2=
\E^P\int_0^Te^{\gamma t}(|y_t|^2+|z_t|_{a_t^P}^2+
\|u_t\|_\Pi^2)dt.
\]
The three factors are Hilbert spaces of equivalence classes of predictable
processes. Bound \eqref{eq:qv} and
$1\le e^{\gamma t}\le e^{\gamma T}$ make this norm equivalent to the usual
product norm, so the space is complete. Lemma~\ref{lem:lawmeas} provides a
measurable representative of the flow $\Lcal^P(y_t)$, defined for a.e. $t$,
and shows that the choice of representative of $y$ does not change the
following integrals.

\emph{Step 3: explicit solution for a fixed input.}
For $v=(y,z,u)\in\mathcal H_\gamma(P)$, define
\[
G_t(v)=f(t,X_{t-},\mu_{t-},y_t,z_t,u_t,\Lcal^P(y_t),\alpha_t)+h_t.
\]
The Lipschitz condition and coupling with zero give
\[
\E^P|G_t(v)|^2\le C\E^P\left[
|\ell_t|^2+|y_t|^2+|z_t|_{a_t^P}^2+\|u_t\|_\Pi^2\right].
\]
Thus $G(v)\in L^2(dt\otimes P)$ and
$Q(v)=\zeta+\int_0^TG_s(v)ds\in L^2(P)$. The martingale
$M_t(v)=\E^P[Q(v)\mid\mathcal F_t^P]$ admits a c\`adl\`ag version and, by
\eqref{eq:PRP}, the unique representation
\[
M_t(v)=M_0(v)+\int_0^tZ_s(v)dB_s+
\int_0^t\int_EU_s(v,e)\widetilde N^P(ds,de).
\]
Define $Y_t(v)=M_t(v)-\int_0^tG_s(v)ds$. Since $M_T(v)=Q(v)$, we have,
simultaneously for every $t$,
\[
Y_t(v)=\zeta+\int_t^TG_s(v)ds
-\int_t^TZ_s(v)dB_s-\int_t^T\int_EU_s(v,e)\widetilde N^P(ds,de).
\]
Doob's inequality and the isometries give
\begin{align*}
\E^P\sup_t|Y_t(v)|^2
&\le8\E^P|Q(v)|^2+2T\E^P\int_0^T|G_s(v)|^2ds,\\
\E^P\int_0^T(|Z_s(v)|_{a_s^P}^2+\|U_s(v)\|_\Pi^2)ds
&=\E^P|Q(v)-\E^P[Q(v)\mid\mathcal F_0^P]|^2
\le\E^P|Q(v)|^2.
\end{align*}
These bounds show that the output belongs to the required spaces. The
fixed-point map is $\Psi(v)=(Y_-(v),Z(v),U(v))$, with $Y_{0-}=Y_0$. Its first
component is predictable, and $Y_-=Y$ for $dt\otimes P$-a.e.
$(t,\omega)$.

\emph{Step 4: contraction, including explicit treatment of the law term.}
Compare two inputs $v,\bar v$. Their generator difference satisfies
\begin{align*}
|\Delta G_t|^2\le C L_B^2\bigl(
|\Delta y_t|^2+|\Delta z_t|_{a_t^P}^2+
\|\Delta u_t\|_\Pi^2+
\Wtwo^2(\Lcal^P(y_t),\Lcal^P(\bar y_t))\bigr).
\end{align*}
Only after taking expectation, the coupling inequality
$\Wtwo^2(\Lcal^P(y_t),\Lcal^P(\bar y_t))\le\E^P|\Delta y_t|^2$ yields
\[
\E^P\int_0^Te^{\gamma t}|\Delta G_t|^2dt
\le C L_B^2\|v-\bar v\|_\gamma^2.
\]
The terminal difference is zero. We now establish the energy identity without
appealing to a separate stability theorem. By Step 3, both outputs belong to
$\mathbb S^2\times\mathbb H_B^2\times\mathbb H_\Pi^2$. The product
martingales arising in Itô's formula have integrable suprema. For the
continuous term, BDG and Cauchy--Schwarz give
\[
\E^P\sup_{t\le T}\left|\int_0^te^{\gamma s}
\Delta Y_{s-}\Delta Z_s dB_s\right|
\le C_{\gamma,T}(\E^P\sup_s|\Delta Y_s|^2)^{1/2}
\left(\E^P\int_0^T|\Delta Z_s|_{a_s^P}^2ds\right)^{1/2}<\infty.
\]
For the discontinuous term, optional quadratic variation and compensation give
similarly
\begin{align*}
&\E^P\sup_{t\le T}\left|\int_0^t\int_Ee^{\gamma s}
\Delta Y_{s-}\Delta U_s(e)\widetilde N^P(ds,de)\right|\\
&\quad\le C_{\gamma,T}(\E^P\sup_s|\Delta Y_s|^2)^{1/2}
\left(\E^P\int_0^T\int_E|\Delta U_s(e)|^2N(ds,de)\right)^{1/2}\\
&\quad=C_{\gamma,T}(\E^P\sup_s|\Delta Y_s|^2)^{1/2}
\left(\E^P\int_0^T\|\Delta U_s\|_\Pi^2ds\right)^{1/2}<\infty.
\end{align*}
These bounds are obtained first after localization, then by Fatou's lemma.
They show that the product martingales belong to $H^1$; the stopped
martingales converge in $L^1$ and retain zero expectation. We keep the jump
quadratic term in the positive form
$\int e^{\gamma s}|\Delta U_s(e)|^2N(ds,de)$, whose expectation equals the
compensated integral. Positive terms pass to the limit by monotone
convergence, while the drift term passes by Cauchy--Schwarz and domination.
The terminal datum passes by domination with
$e^{\gamma T}\sup_s|\Delta Y_s|^2$ by choosing threshold stopping times that
are eventually equal to $T$ almost surely. Thus no fourth moment of $U$ is
needed. The resulting identity is
\begin{align*}
\E^P|\Delta Y_0|^2+
\E^P\int_0^Te^{\gamma t}
\{\gamma|\Delta Y_t|^2+|\Delta Z_t|_{a_t^P}^2+
\|\Delta U_t\|_\Pi^2\}dt
=2\E^P\int_0^Te^{\gamma t}\Delta Y_t\Delta G_tdt.
\end{align*}
Using $2ab\le\gamma a^2/2+2b^2/\gamma$ and $\gamma\ge2$,
\[
\|\Psi(v)-\Psi(\bar v)\|_\gamma^2
\le\frac{C L_B^2}{\gamma}\|v-\bar v\|_\gamma^2.
\]
The choice can be made explicit. Let
$a_*=\lambda_{\min}(\underline a)>0$ and
$c_*=8L_B^2\max\{1,a_*^{-1}\}$. Indeed,
$(r_1+\cdots+r_4)^2\le4\sum_jr_j^2$, followed by the law coupling, gives
\[
\E^P|\Delta G_t|^2\le c_*\E^P
(|\Delta y_t|^2+|\Delta z_t|_{a_t^P}^2+\|\Delta u_t\|_\Pi^2).
\]
The preceding energy estimate therefore yields the contraction factor
$\chi=2c_*/\gamma$. Choosing $\gamma=2+8c_*$ ensures $\gamma\ge2$ and
$0\le\chi<1/2$, independently of $P$. The norm contracts with factor
$\sqrt\chi$ and its square with factor $\chi$.

\emph{Step 5: construction of the fixed point and passage to the integrals.}
Start from $v^0=0$ and define $v^{n+1}=\Psi(v^n)$. Step 4 gives
\[
\|v^{n+1}-v^n\|_\gamma\le\chi^{n/2}\|v^1\|_\gamma,
\qquad
\sum_{n=0}^\infty\|v^{n+1}-v^n\|_\gamma<\infty.
\]
Completeness gives a limit $v=(y,z,u)$ and continuity of $\Psi$ yields
$\Psi(v)=v$. We spell out its identification with a c\`adl\`ag solution.
The generator estimate of Step 4 shows that
$G(v^n)\to G(v)$ strongly in $L^2(dt\otimes P)$. For the outputs of Step 3,
Doob's inequality and Cauchy--Schwarz yield
\begin{align*}
\E^P\sup_t|M_t(v^n)-M_t(v)|^2
&\le4T\E^P\int_0^T|G_s(v^n)-G_s(v)|^2ds,\\
\E^P\sup_t|Y_t(v^n)-Y_t(v)|^2
&\le10T\E^P\int_0^T|G_s(v^n)-G_s(v)|^2ds\longrightarrow0.
\end{align*}
The outputs $Z(v^n),U(v^n)$ converge to $z,u$ in their energy norms by
convergence of $v^{n+1}$. More precisely, writing $Z_n=Z(v^n)$ and
$U_n=U(v^n)$,
\begin{align*}
\E^P\sup_{t\le T}\left|\int_0^t(Z_n-z)dB_s\right|^2
&\le4\E^P\int_0^T|Z_n-z|_{a_s^P}^2ds\longrightarrow0,\\
\E^P\sup_{t\le T}\left|\int_0^t\int_E(U_n-u)\widetilde N^P(ds,de)\right|^2
&\le4\E^P\int_0^T\|U_n-u\|_\Pi^2ds\longrightarrow0,\\
\E^P\sup_{t\le T}\left|\int_0^t(G_s(v^n)-G_s(v))ds\right|^2
&\le T\E^P\int_0^T|G_s(v^n)-G_s(v)|^2ds\longrightarrow0.
\end{align*}
The limit $Y=Y(v)$ is c\`adl\`ag and satisfies $y=Y_-$ in
$L^2(dt\otimes P)$. By Fubini's theorem,
$\Lcal^P(y_t)=\Lcal^P(Y_t)$ for almost every $t$. Hence the formula in Step 3
is exactly \eqref{eq:BWD} with the backward law of $Y$. The integral
identities hold for all $t$ outside a single $P$-null set, after choosing their
c\`adl\`ag versions.

\emph{Step 6: uniform estimate and uniqueness.}
For zero input, $G(0)=\ell$. Step 3 gives
$\|\Psi(0)\|_\gamma^2\le C_{T,\gamma}\mathfrak d_P$. The contraction
implies
\[
\|v\|_\gamma\le\|\Psi(v)-\Psi(0)\|_\gamma+
\|\Psi(0)\|_\gamma
\le\sqrt\chi\|v\|_\gamma+\|\Psi(0)\|_\gamma,
\]
hence
\[
\|v\|_\gamma^2\le(1-\sqrt\chi)^{-2}\|\Psi(0)\|_\gamma^2
\le (1-\sqrt\chi)^{-2}C_{T,\gamma}\mathfrak d_P.
\]
The factor $(1-\sqrt\chi)^{-2}$ is finite and common to all models. The bound
on $G(v)$ and then the bound on $Y(v)$ from Step 3 give
\[
\E^P\sup_t|Y_t|^2+
\E^P\int_0^T(|Z_t|_{a_t^P}^2+\|U_t\|_\Pi^2)dt
\le C\mathfrak d_P.
\]
Finally, $\mathsf K_t=\int_0^th_sds$ is absolutely continuous and
$\E^P\sup_t|\mathsf K_t|^2\le T\E^P\int_0^T|h_t|^2dt$. Together with
Step 1 and the forward bound, this proves \eqref{eq:Bexistbound}.

Any other quadratic solution of the BSDE along this fixed forward process
defines a fixed point of the same map $\Psi$. Contraction forces equality of
the equivalence classes $(Y_-,Z,U)$. The integral equation and the isometries
then imply indistinguishability of the c\`adl\`ag versions of $Y$ and equality
of $\mathsf K$. Uniqueness of $Z,U$ is understood in their energy norms.
\end{proof}

\begin{theorem}[Well-posedness of the selected system]\label{thm:well}
Let $\alpha\in\mathcal A^2$. Under Assumptions~\ref{HF}, \ref{HB},
and~\ref{HD}, in the probabilistic framework with representation
\eqref{eq:PRP}, the system
\eqref{eq:FWD}--\eqref{eq:Kdef} admits, for every $P\in\Pcal$, a unique
solution
\[
(X^P,Y^P,Z^P,U^P,\mathsf K^P)
\in\mathbb S^2(P;\R^d)\times\mathbb S^2(P)\times\mathbb H_B^2(P)
\times\mathbb H_\Pi^2(P)\times\mathbb S^2(P).
\]
It satisfies
\begin{align}
\sup_{P\in\Pcal}\E^P\Bigl[
&\sup_{t\le T}|X_t^P|^2+\sup_{t\le T}|Y_t^P|^2
+\sup_{t\le T}|\mathsf K_t^P|^2\nonumber\\
&+\int_0^T\{|Z_t^P|_{a_t^P}^2+\|U_t^P\|_\Pi^2\}dt
\Bigr]<\infty.
\label{eq:uniformwell}
\end{align}
\end{theorem}

\begin{proof}
Theorem~\ref{thm:Fwell} constructs the unique forward process $X^P$ and its
law flow. Theorem~\ref{thm:Bwell} then constructs $(Y^P,Z^P,U^P)$ and
$\mathsf K^P$ along this forward process. The triangular structure ensures
that this construction solves the full system. If two system solutions are
given, forward uniqueness identifies their first components, and backward
uniqueness then identifies their triples in the senses stated in
Theorem~\ref{thm:Bwell}. The integral definition finally identifies their
cumulative processes. Adding \eqref{eq:Fsecond} and \eqref{eq:Bexistbound}
gives \eqref{eq:uniformwell}.
\end{proof}

\subsection{Stability of the backward equation with jumps}

Having established existence and uniqueness, we now compare solutions
associated with different data. In the following lemma, the exogenous sources
are predictable and square integrable, while the other data satisfy the
preceding quadratic and Lipschitz assumptions.

\begin{lemma}[Stability of a McKean--Vlasov BSDE with jumps]
\label{lem:Bstab}
Let $(Y,Z,U)$ and $(\bar Y,\bar Z,\bar U)$ be two solutions with terminal
data $\zeta,\bar\zeta$ and exogenous terms $h,\bar h$. The other arguments of
the generator are respectively $(X,\mu,\alpha)$ and
$(\bar X,\bar\mu,\bar\alpha)$. Then
\begin{align}
&\E^P\sup_{t\le T}|Y_t-\bar Y_t|^2
+\E^P\int_0^T|Z_t-\bar Z_t|_{a_t^P}^2dt
+\E^P\int_0^T\|U_t-\bar U_t\|_\Pi^2dt\nonumber\\
&\le C\E^P\left[
|\zeta-\bar\zeta|^2+\int_0^T\{
|X_t-\bar X_t|^2+\Wtwo^2(\mu_t,\bar\mu_t)
+|\alpha_t-\bar\alpha_t|^2+|h_t-\bar h_t|^2\}dt\right],
\label{eq:Bstab}
\end{align}
where $C$ is independent of $P$.
\end{lemma}

\begin{proof}
\emph{Step 1: integral differences and error data.}
Fix $P$; the two solutions use the same integrators. Set
$\Delta Y=Y-\bar Y$, $\Delta Z=Z-\bar Z$, $\Delta U=U-\bar U$, and
$\Delta\zeta=\zeta-\bar\zeta$. With
\begin{align*}
\Delta f_s={}&f(s,X_{s-},\mu_{s-},Y_s,Z_s,U_s,\eta_s,\alpha_s)\\
&-f(s,\bar X_{s-},\bar\mu_{s-},\bar Y_s,\bar Z_s,\bar U_s,
\bar\eta_s,\bar\alpha_s),
\end{align*}
we have, for $G_s=\Delta f_s+\Delta h_s$,
\[
\Delta Y_t=\Delta\zeta+\int_t^T G_sds-(M_T-M_t),
\quad M_t=\int_0^t\Delta Z_sdB_s+
\int_0^t\int_E\Delta U_s(e)\widetilde N^P(ds,de).
\]
Define
\[
R_s=|\Delta X_{s-}|+W_2(\mu_{s-},\bar\mu_{s-})
+|\Delta\alpha_s|+|\Delta h_s|,
\quad w_s=W_2(\eta_s,\bar\eta_s),
\]
\[
z_s^2=|\Delta Z_s|_{a_s^P}^2,\qquad
v_s^2=\|\Delta U_s\|_\Pi^2,\qquad
\mathfrak e_P=\E^P|\Delta\zeta|^2+\E^P\int_0^T R_s^2ds.
\]
If $\mathfrak e_P=\infty$, the claimed inequality is trivial, so assume it
is finite. Let $L=\max\{1,L_B\}$ and
$a_*=\lambda_{\min}(\underline a)$. Then
\begin{equation}
|G_s|\le L\{|\Delta Y_s|+a_*^{-1/2}z_s+v_s+w_s+R_s\},
\qquad w_s^2\le\E^P|\Delta Y_s|^2.
\label{eq:Bstablipdetail}
\end{equation}
In particular, $G\in L^2(dt\otimes P)$ by the solution spaces.

\emph{Step 2: Itô formula, compensation, and justification of expectations.}
For $\gamma>0$, Itô's formula with jumps between $t$ and $T$ gives
\begin{align*}
&e^{\gamma t}|\Delta Y_t|^2+
\int_t^Te^{\gamma s}(\gamma|\Delta Y_{s-}|^2+z_s^2)ds
+\int_t^T\int_E e^{\gamma s}|\Delta U_s(e)|^2N(ds,de)\\
&=e^{\gamma T}|\Delta\zeta|^2+
2\int_t^Te^{\gamma s}\Delta Y_{s-}G_sds
-2\int_t^Te^{\gamma s}\Delta Y_{s-}\Delta Z_sdB_s\\
&\hspace{18mm}-2\int_t^T\int_Ee^{\gamma s}
\Delta Y_{s-}\Delta U_s(e)\widetilde N^P(ds,de).
\end{align*}
The two product integrals are initially only local martingales. For their
stopped versions, BDG and Cauchy--Schwarz respectively yield
\[
C_{\gamma,T}(\E^P\sup_s|\Delta Y_s|^2)^{1/2}
(\E^P\int_0^Tz_s^2ds)^{1/2},
\]
\[
C_{\gamma,T}(\E^P\sup_s|\Delta Y_s|^2)^{1/2}
\left(\E^P\int_0^T\int_E|\Delta U_s(e)|^2N(ds,de)\right)^{1/2}.
\]
Compensation identifies the last expectation with
$\E^P\int_0^Tv_s^2ds$. These bounds are finite. Fatou's lemma after
localization shows that the suprema of the product martingales are integrable.
Hence they are uniformly integrable, and their increments have zero
expectation. The drift is integrable because
\[
\E^P\int_0^T|\Delta Y_sG_s|ds
\le(\E^P\int_0^T|\Delta Y_s|^2ds)^{1/2}
(\E^P\int_0^T|G_s|^2ds)^{1/2}<\infty.
\]
The jump quadratic term is positive and has finite expectation, so Tonelli's
theorem and compensation apply directly. We do not treat it as a quadratic
integral against $\widetilde N^P$, which would require an additional moment.
The values at $s$ and $s-$ coincide $ds\otimes P$-a.e. Consequently,
\begin{align}
&e^{\gamma t}\E^P|\Delta Y_t|^2+
\E^P\int_t^Te^{\gamma s}(\gamma|\Delta Y_s|^2+z_s^2+v_s^2)ds
\nonumber\\
&=e^{\gamma T}\E^P|\Delta\zeta|^2+
2\E^P\int_t^Te^{\gamma s}\Delta Y_sG_sds.
\label{eq:Bstabexact}
\end{align}
This justification avoids assuming that an arbitrary sequence of stopping
times converging to $T$ makes $Y_{\tau_n}$ converge to $Y_T$ in the presence
of a possible terminal jump.

\emph{Step 3: explicit absorption, including the backward law.}
Each term in \eqref{eq:Bstablipdetail} is treated by Young's inequality:
\begin{align*}
2L|\Delta Y_s|a_*^{-1/2}z_s
&\le\tfrac14z_s^2+4L^2a_*^{-1}|\Delta Y_s|^2,\\
2L|\Delta Y_s|v_s&\le\tfrac14v_s^2+4L^2|\Delta Y_s|^2,\\
2L|\Delta Y_s|w_s&\le L|\Delta Y_s|^2+Lw_s^2,\\
2L|\Delta Y_s|R_s&\le L|\Delta Y_s|^2+LR_s^2.
\end{align*}
Therefore, with $c_Y=4L+4L^2(1+a_*^{-1})$,
\[
2\Delta Y_sG_s\le c_Y|\Delta Y_s|^2
+\tfrac14(z_s^2+v_s^2)+Lw_s^2+LR_s^2.
\]
The quantity $w_s$ is deterministic under $P$; after expectation,
$Lw_s^2\le L\E^P|\Delta Y_s|^2$. Choose $\gamma=c_Y+L+1$.
Substitution in \eqref{eq:Bstabexact} yields
\begin{align*}
&e^{\gamma t}\E^P|\Delta Y_t|^2+
\E^P\int_t^Te^{\gamma s}|\Delta Y_s|^2ds
+\tfrac34\E^P\int_t^Te^{\gamma s}(z_s^2+v_s^2)ds\\
&\le e^{\gamma T}\E^P|\Delta\zeta|^2
+L\E^P\int_t^Te^{\gamma s}R_s^2ds
\le Le^{\gamma T}\mathfrak e_P.
\end{align*}
Taking the supremum in $t$ for the first term and $t=0$ for the integrals,
we obtain
\begin{equation}
\sup_{t\le T}\E^P|\Delta Y_t|^2+
\E^P\int_0^T(|\Delta Y_s|^2+z_s^2+v_s^2)ds
\le C\mathfrak e_P.
\label{eq:Benergy}
\end{equation}
The choice of $\gamma$ depends only on $L_B,\underline a$.

\emph{Step 4: time supremum and norm of the cumulative process.}
By \eqref{eq:Bstablipdetail},
\[
\E^P\int_0^T|G_s|^2ds
\le5L^2\E^P\int_0^T
(|\Delta Y_s|^2+a_*^{-1}z_s^2+v_s^2+w_s^2+R_s^2)ds
\le C\mathfrak e_P.
\]
Orthogonality of the continuous and discontinuous parts, the isometry, and
Doob's inequality imply
\[
\E^P\sup_{t\le T}|M_t|^2
\le4\E^P|M_T|^2=4\E^P\int_0^T(z_s^2+v_s^2)ds.
\]
The integral form from Step 1 then yields
\begin{align*}
\E^P\sup_{t\le T}|\Delta Y_t|^2
&\le3\E^P|\Delta\zeta|^2+3T\E^P\int_0^T|G_s|^2ds
+12\E^P\sup_{t\le T}|M_t|^2\le C\mathfrak e_P.
\end{align*}
Finally,
$R_s^2\le4(|\Delta X_{s-}|^2+W_2^2(\mu_{s-},\bar\mu_{s-})
+|\Delta\alpha_s|^2+|\Delta h_s|^2)$ gives \eqref{eq:Bstab}. When
$\mathsf K_t=\int_0^th_sds$ and
$\bar{\mathsf K}_t=\int_0^t\bar h_sds$, we also have
\[
\E^P\sup_{t\le T}|\mathsf K_t-\bar{\mathsf K}_t|^2
\le T\E^P\int_0^T|\Delta h_s|^2ds.
\]
All estimates are obtained before taking the supremum over $P$ and use common
constants; no comparison is made between integrals defined under different
probability measures.
\end{proof}

\begin{lemma}[Continuity of the laws despite jumps]\label{lem:lawcontinuous}
For square-integrable solutions, the maps $t\mapsto\mu_t^P$ and
$t\mapsto\eta_t^P$ are continuous in $W_2$ under each $P$. For every
deterministic time $t>0$,
$X_t^P=X_{t-}^P$ and $Y_t^P=Y_{t-}^P$, $P$-a.s.
\end{lemma}
\begin{proof}
For the forward process, the isometries and Cauchy--Schwarz give, for $s<t$,
\[
\E^P|X_t-X_s|^2\le3(t-s)\E^P\int_s^t|b_u|^2du
+3\E^P\int_s^t\operatorname{Tr}(\sigma_u a_u^P\sigma_u^\top)du
+3\E^P\int_s^t\|\beta_u\|_\Pi^2du.
\]
All integrands are integrable on $[0,T]$. The same estimate for $Y$ uses
$f+h$, $Z$, and $U$, where $h=\langle\varrho,\Gamma\rangle$; their quadratic
integrability follows from Lipschitz growth and the solution spaces. Absolute
continuity of time integrals yields $L^2$ continuity of both processes at
deterministic times. The canonical coupling then gives $W_2$ continuity of
the laws. For $s_n\uparrow t$, the c\`adl\`ag property gives
$X_{s_n}\to X_{t-}$ almost surely, while $L^2$ continuity gives
$X_{s_n}\to X_t$ in probability. Uniqueness of the limit in probability
implies $X_t=X_{t-}$ almost surely; the same argument applies to $Y$. This
statement at deterministic $t$ does not mean that paths are continuous at
their random jump times.
\end{proof}
\section{Stability under perturbations of the data}

The results of this section use the quadratic framework
$\alpha\in\mathcal A^2$ and Assumptions~\ref{HF}, \ref{HB}, \ref{HD}.
The higher moments in \ref{HqJ} are not needed here.

The comparison keeps $P$, the filtration, and the integrators $(B,N)$ fixed.
Stability with varying integrators and filtrations is studied in the preprint
of Papapantoleon, Saplaouras and Theodorakopoulos (arXiv:2506.03562).

Consider two sets of data
\[
\mathfrak D=(\xi,\alpha,b,\sigma,\beta,f,g,A),\qquad
\bar{\mathfrak D}=(\bar\xi,\bar\alpha,\bar b,\bar\sigma,
\bar\beta,\bar f,\bar g,\bar A),
\]
on the same stochastic basis under $P$. We assume that their coefficients
satisfy Assumptions~\ref{HF}--\ref{HD} with the same structural constants.
Objects associated with the second set of data carry a bar. For $\lambda>0$,
the two operators are replaced by their Yosida approximations
$A^\lambda$ and $\bar A^\lambda$.

The forward residual, evaluated along the second solution, is
\begin{align}
\mathfrak R_F(P):={}&\E^P|\xi-\bar\xi|^2
+\E^P\int_0^T|\alpha_t-\bar\alpha_t|^2dt\nonumber\\
&+\E^P\int_0^T\Bigl[
|b(t,\bar X_{t-},\bar\mu_{t-},\bar\alpha_t)
-\bar b(t,\bar X_{t-},\bar\mu_{t-},\bar\alpha_t)|^2\nonumber\\
&\qquad+\|\sigma(t,\bar X_{t-},\bar\mu_{t-},\bar\alpha_t)
-\bar\sigma(t,\bar X_{t-},\bar\mu_{t-},\bar\alpha_t)\|^2\nonumber\\
&\qquad+\|\beta(t,\bar X_{t-},\bar\mu_{t-},\bar\alpha_t,\cdot)
-\bar\beta(t,\bar X_{t-},\bar\mu_{t-},\bar\alpha_t,\cdot)
\|_{L_\Pi^2}^2\Bigr]dt.
\label{eq:dataRF}
\end{align}
For the regularized backward part, set
\begin{align}
\mathfrak R_B^\lambda(P):={}&
\E^P|g(\bar X_T,\bar\mu_T)-\bar g(\bar X_T,\bar\mu_T)|^2
\nonumber\\
&+\E^P\int_0^T\bigl|f(t,\bar X_{t-},\bar\mu_{t-},
\bar Y_t,\bar Z_t,\bar U_t,\bar\eta_t,\bar\alpha_t)\nonumber\\
&\hspace{34mm}-\bar f(t,\bar X_{t-},\bar\mu_{t-},
\bar Y_t,\bar Z_t,\bar U_t,\bar\eta_t,\bar\alpha_t)
\bigr|^2dt\nonumber\\
&+\E^P\int_0^T|A^\lambda(\bar X_{t-},\bar\mu_{t-})
-\bar A^\lambda(\bar X_{t-},\bar\mu_{t-})|^2dt.
\label{eq:dataRB}
\end{align}
In \eqref{eq:dataRB}, $(\bar Y,\bar Z,\bar U)$ denotes the regularized
solution of the second system with the same parameter $\lambda$.

\begin{theorem}[Strong stability with fixed integrators]
\label{thm:datastability}
Let $\lambda>0$. Under Assumptions~\ref{HF}, \ref{HB}, and~\ref{HD},
with both data sets satisfying the same structural constants on the same
stochastic basis, filtration, and integrators, and with
$\alpha,\bar\alpha\in\mathcal A^2$, the regularized solutions satisfy
\begin{align}
\E^P\sup_{t\le T}|X_t-\bar X_t|^2
&\le C\mathfrak R_F(P),
\label{eq:dataXstab}\\
\E^P\Biggl[\sup_{t\le T}|Y_t-\bar Y_t|^2
&+\sup_{t\le T}|\mathsf K_t^\lambda-
\bar{\mathsf K}_t^\lambda|^2
+\int_0^T|Z_t-\bar Z_t|_{a_t^P}^2dt\nonumber\\
&+\int_0^T\|U_t-\bar U_t\|_\Pi^2dt\Biggr]
\le C\Bigl\{(1+\lambda^{-2})\mathfrak R_F(P)
+\mathfrak R_B^\lambda(P)\Bigr\}.
\label{eq:dataYstab}
\end{align}
The constant $C$ is independent of $P$, $\lambda$, and the two data sets;
all dependence on $\lambda$ is displayed explicitly in
\eqref{eq:dataYstab}.
\end{theorem}

\begin{proof}
Set $\Delta X=X-\bar X$ and evaluate coefficient residuals along the barred
solution as in \eqref{eq:dataRF}.  The forward difference equation, the
Lipschitz assumptions, Lemma~\ref{lem:mart}, the natural coupling
$\Wtwo^2(\mu_t,\bar\mu_t)\le\E^P|\Delta X_t|^2$, and Gronwall give
\[
\E^P\sup_{t\le T}|\Delta X_t|^2\le C\mathfrak R_F(P),
\]
which is \eqref{eq:dataXstab}.

For the backward equation, decompose the terminal value and generator into
their Lipschitz state differences and the residuals entering
$\mathfrak R_B^\lambda(P)$.  The Yosida estimate \eqref{eq:AYlip} gives
\[
|A^\lambda(X_{t-},\mu_{t-})
-\bar A^\lambda(\bar X_{t-},\bar\mu_{t-})|^2
\le C\lambda^{-2}|\Delta X_{t-}|^2
+C\Wtwo^2(\mu_{t-},\bar\mu_{t-})
+C|r_t^{A,\lambda}|^2.
\]
Applying the jump Itô estimate of Lemma~\ref{lem:Bstab} to
$e^{\gamma t}|\Delta Y_t|^2$, choosing $\gamma$ large enough and absorbing
the $Z$- and $U$-terms yields
\[
\sup_{t\le T}\E^P|\Delta Y_t|^2+
\E^P\int_0^T\!\!
\{|\Delta Z_t|_{a_t^P}^2+\|\Delta U_t\|_\Pi^2\}dt
\le
C\{(1+\lambda^{-2})\mathfrak R_F(P)
+\mathfrak R_B^\lambda(P)\}.
\]
BDG applied to the backward difference upgrades the first term to
$\E^P\sup_{t\le T}|\Delta Y_t|^2$ with the same right-hand side.  Finally,
\[
\sup_{t\le T}|\mathsf K_t^\lambda-\bar{\mathsf K}_t^\lambda|^2
\le T|\varrho|^2\int_0^T
|A^\lambda(X_{s-},\mu_{s-})
-\bar A^\lambda(\bar X_{s-},\bar\mu_{s-})|^2ds,
\]
and \eqref{eq:dataXstab} completes \eqref{eq:dataYstab}.  All constants are
uniform in $P$ because the ellipticity, Lipschitz and martingale constants are
uniform over $\Pcal$.
\end{proof}

\begin{corollary}[Passage to the limit uniformly in the model]
\label{cor:datauniform}
Let $(\mathfrak D^n)_n$ be a sequence of data satisfying the preceding
assumptions uniformly, and fix $\lambda>0$. If
\[
\sup_{P\in\Pcal}\mathfrak R_F^n(P)\longrightarrow0,
\qquad
\sup_{P\in\Pcal}\mathfrak R_B^{n,\lambda}(P)\longrightarrow0,
\]
then
\begin{align*}
&\sup_{P\in\Pcal}\E^P\Biggl[
\sup_{t\le T}|X_t^n-X_t|^2+\sup_{t\le T}|Y_t^{n,\lambda}-Y_t^\lambda|^2
+\sup_{t\le T}|\mathsf K_t^{n,\lambda}-\mathsf K_t^\lambda|^2\\
&\hspace{32mm}+\int_0^T|Z_t^{n,\lambda}-Z_t^\lambda|_{a_t^P}^2dt
+\int_0^T\|U_t^{n,\lambda}-U_t^\lambda\|_\Pi^2dt
\Biggr]\longrightarrow0.
\end{align*}
These are strong convergences, respectively in $\mathbb S^2(P)$,
$\mathbb H_B^2(P)$, and $\mathbb H_\Pi^2(P)$, uniformly in $P$.
\end{corollary}

\begin{proof}
Apply Theorem~\ref{thm:datastability} to
$(\mathfrak D^n,\mathfrak D)$. Its constant is uniform in $P$ and $n$ because
the structural constants are uniform. Take the supremum over $P\in\Pcal$ in
\eqref{eq:dataXstab}--\eqref{eq:dataYstab}, then let $n\to\infty$.
\end{proof}

\begin{corollary}[Stability of the selected systems]
\label{cor:selectedData}
Assume $\sup_P\mathfrak R_F^n(P)\to0$ and that the terminal and Lipschitz
residuals
\begin{align*}
\mathfrak R_{g}^n(P)&=
\E^P|g_n(X_T^P,\mu_T^P)-g(X_T^P,\mu_T^P)|^2,\\
\mathfrak R_{f}^n(P)&=
\E^P\int_0^T|f_n(t,X_{t-}^P,\mu_{t-}^P,Y_t^P,Z_t^P,U_t^P,
\eta_t^P,\alpha_t)\\
&\hspace{39mm}-f(t,X_{t-}^P,\mu_{t-}^P,Y_t^P,Z_t^P,U_t^P,
\eta_t^P,\alpha_t)|^2dt
\end{align*}
converge to zero uniformly in $P$. Assume finally that
\begin{equation}
\sup_{P\in\Pcal}\E^P\int_0^T
|A_n^0(X_{t-}^{n,P},\mu_{t-}^{n,P})
-A^0(X_{t-}^P,\mu_{t-}^P)|^2dt\longrightarrow0.
\label{eq:selectedDataSource}
\end{equation}
Then
\begin{align*}
\sup_{P\in\Pcal}\E^P\Biggl[&
\sup_{t\le T}|Y_t^{n,P}-Y_t^P|^2
+\sup_{t\le T}|\mathsf K_t^{n,P}-\mathsf K_t^P|^2\\
&+\int_0^T|Z_t^{n,P}-Z_t^P|_{a_t^P}^2dt
+\int_0^T\|U_t^{n,P}-U_t^P\|_\Pi^2dt\Biggr]\longrightarrow0.
\end{align*}
\end{corollary}

\begin{proof}
The difference of the selected sources is already controlled in
\eqref{eq:selectedDataSource}; hence it is neither necessary nor valid to use
the $\lambda^{-1}$ Lipschitz estimate for the Yosida approximation, which
holds only for $\lambda>0$. We return to the jump Itô stability estimate of
Lemma~\ref{lem:Bstab}, replacing $\Delta h$ by the difference of the selected
sources. No substitution involving $\lambda^{-2}$ is made. After Young's
inequality and absorption,
\begin{align*}
&\E^P\sup_{t\le T}|Y_t^{n,P}-Y_t^P|^2
+\E^P\int_0^T\{|Z_t^{n,P}-Z_t^P|_{a_t^P}^2
+\|U_t^{n,P}-U_t^P\|_\Pi^2\}dt\\
&\quad\le C\biggl\{\mathfrak R_F^n(P)+\mathfrak R_g^n(P)
+\mathfrak R_f^n(P)
+\E^P\int_0^T|A_n^0(X_{t-}^{n,P},\mu_{t-}^{n,P})
-A^0(X_{t-}^P,\mu_{t-}^P)|^2dt\biggr\}.
\end{align*}
Moreover, Cauchy--Schwarz gives
\[
\E^P\sup_{t\le T}|\mathsf K_t^{n,P}-\mathsf K_t^P|^2
\le T|\varrho|^2\E^P\int_0^T
|A_n^0(X_{t-}^{n,P},\mu_{t-}^{n,P})
-A^0(X_{t-}^P,\mu_{t-}^P)|^2dt.
\]
Taking the supremum over $P$ and then using the convergence assumptions
completes the proof.
\end{proof}

Condition \eqref{eq:selectedDataSource} does not follow from the forward
Lipschitz assumptions alone, because $A^0$ may be discontinuous on $\Sigma$.
It must be verified through a non-contact assumption, an occupation estimate,
or uniform convergence of the selections.
\section{Yosida approximation and graph closure}

For $\lambda>0$, replace \eqref{eq:Kdef} by
\begin{equation}
\mathsf K_t^{\lambda,P}
=\int_0^t\langle\varrho,
A^\lambda(X_{s-}^P,\mu_{s-}^P)\rangle ds
\label{eq:Klambda}
\end{equation}
and denote by $(Y^{\lambda,P},Z^{\lambda,P},U^{\lambda,P})$ the corresponding
backward solution, together with
\[
\eta_t^{P,\lambda}=\Lcal^P(Y_t^{\lambda,P}).
\]
The regularized solution exists and is unique by the construction of
Theorem~\ref{thm:Bwell}; we verify here that its data satisfy the required
conditions:
\[
h_t^{\lambda,P}=\langle\varrho,A^\lambda(X_{t-}^P,\mu_{t-}^P)\rangle,
\qquad |h_t^{\lambda,P}|^2\le|\varrho|^2|A^0(X_{t-}^P,\mu_{t-}^P)|^2.
\]
The resolvent is Borel measurable, hence $h^{\lambda,P}$ is predictable, and
\[
\sup_{P,\lambda>0}\E^P\int_0^T|h_t^{\lambda,P}|^2dt<\infty.
\]
For $F^{\lambda,P}=f(\cdot,X_-,\mu_-,\cdot,\alpha)+h^{\lambda,P}$,
\[
\sup_{P,\lambda>0}\E^P\int_0^T
|F^{\lambda,P}(t,0,0,0,\delta_0)|^2dt<\infty,
\qquad
F^{\lambda,P}(t,v)-F^{\lambda,P}(t,v')
=F^{P}(t,v)-F^{P}(t,v').
\]
Theorem~\ref{thm:Bwell} therefore applies with the same contraction and
energy constants for every $\lambda>0$. The forward process is unchanged.

\begin{lemma}[Quadratic convergence of the resolvents]\label{lem:Jstrong}
Under the quadratic well-posedness assumptions,
\[
\sup_P\E^P\int_0^T
|J_\lambda(X_{t-}^P,\mu_{t-}^P)-X_{t-}^P|^2dt\le C\lambda^2.
\]
\end{lemma}
\begin{proof}
The identity $J_\lambda-x=-\lambda A^\lambda$ and the bound
$|A^\lambda|\le|A^0|$ show that the left-hand side is bounded above by
\[
\lambda^2\sup_P\E^P\int_0^T|A^0(X_{t-}^P,\mu_{t-}^P)|^2dt.
\]
The growth of the source, Jensen's inequality and \eqref{eq:Fsecond} make
this quantity uniformly finite. No occupation assumption is used.
\end{proof}

\begin{lemma}[Convergence under a fixed model]\label{lem:fixedY}
Under Assumptions~\ref{HF}, \ref{HB} and~\ref{HD}, for every fixed
$P\in\Pcal$,
\[
\E^P\int_0^T
|A^\lambda(X_{t-}^P,\mu_{t-}^P)
-A^0(X_{t-}^P,\mu_{t-}^P)|^2dt\longrightarrow0.
\]
Consequently, \eqref{eq:Yconv} holds without the supremum over $P$.
\end{lemma}

\begin{proof}
Pointwise convergence follows from \eqref{eq:Aconv}. The domination
\[
|A^\lambda(X_{t-}^P,\mu_{t-}^P)-A^0(X_{t-}^P,\mu_{t-}^P)|^2
\le4|A^0(X_{t-}^P,\mu_{t-}^P)|^2
\]
is integrable by \ref{HD} and \eqref{eq:Fsecond}. The dominated
convergence theorem yields convergence of the source. Lemma~\ref{lem:Bstab}
and the Cauchy--Schwarz inequality applied to
$\mathsf K^{\lambda,P}-\mathsf K^P$ then yield convergence of
$(Y,Z,U,\mathsf K)$.
\end{proof}

\begin{hypothesis}[Uniform convergence of the source]\label{HY}
There exists $\delta>0$ such that
\begin{equation}
\sup_{P\in\Pcal}\E^P\int_0^T
|A^0(X_{t-}^P,\mu_{t-}^P)|^{2+\delta}dt<\infty,
\label{eq:AUI}
\end{equation}
and, for every $\varepsilon>0$,
\begin{equation}
\lim_{\lambda\downarrow0}\sup_{P\in\Pcal}
(dt\otimes P)\left(
|A^\lambda(X_{t-}^P,\mu_{t-}^P)
-A^0(X_{t-}^P,\mu_{t-}^P)|>\varepsilon\right)=0.
\label{eq:Ameasure}
\end{equation}
\end{hypothesis}

\begin{lemma}[Uniform Vitali argument]\label{lem:Vitali}
Under Assumption~\ref{HY},
\begin{equation}
\lim_{\lambda\downarrow0}\sup_{P\in\Pcal}
\E^P\int_0^T|A^\lambda(X_{t-}^P,\mu_{t-}^P)
-A^0(X_{t-}^P,\mu_{t-}^P)|^2dt=0.
\label{eq:Astrong}
\end{equation}
\end{lemma}

\begin{proof}
Write $D_\lambda^P=A^\lambda(X_{-}^P,\mu_{-}^P)
-A^0(X_{-}^P,\mu_{-}^P)$. By \eqref{eq:Aconv},
$|D_\lambda^P|\le2|A^0(X_{-}^P,\mu_{-}^P)|$. Fix $\varepsilon>0$.
Hölder's inequality on the set $\{|D_\lambda^P|>\varepsilon\}$ gives
\begin{align*}
\E^P\int_0^T|D_{\lambda,t}^P|^2dt
&\le T\varepsilon^2+
\left(\E^P\int_0^T|D_{\lambda,t}^P|^{2+\delta}dt
\right)^{2/(2+\delta)}\\
&\quad\times
\left((dt\otimes P)(|D_\lambda^P|>\varepsilon)
\right)^{\delta/(2+\delta)}\\
&\le T\varepsilon^2+C
\left((dt\otimes P)(|D_\lambda^P|>\varepsilon)
\right)^{\delta/(2+\delta)},
\end{align*}
where $C$ is independent of $P$ and $\lambda$ by \eqref{eq:AUI}.
Set $q_\lambda(\varepsilon)=\sup_P(dt\otimes P)
(|D_\lambda^P|>\varepsilon)$. Then
\[
0\le\limsup_{\lambda\downarrow0}\sup_P
\E^P\int_0^T|D_{\lambda,t}^P|^2dt
\le T\varepsilon^2+
C\limsup_{\lambda\downarrow0}q_\lambda(\varepsilon)^{\delta/(2+\delta)}
=T\varepsilon^2.
\]
Since this inequality holds for every $\varepsilon>0$, the left-hand side
is zero.
\end{proof}

\begin{theorem}[Strong Yosida convergence]\label{thm:Yconv}
Under Assumptions~\ref{HF}, \ref{HB}, \ref{HD} and~\ref{HY},
\begin{align}
\lim_{\lambda\downarrow0}\sup_{P\in\Pcal}\E^P\Bigl[
&\sup_{t\le T}|Y_t^{\lambda,P}-Y_t^P|^2
+\sup_{t\le T}|\mathsf K_t^{\lambda,P}-\mathsf K_t^P|^2\nonumber\\
&+\int_0^T|Z_t^{\lambda,P}-Z_t^P|_{a_t^P}^2dt
+\int_0^T\|U_t^{\lambda,P}-U_t^P\|_\Pi^2dt
\Bigr]=0.
\label{eq:Yconv}
\end{align}
\end{theorem}

\begin{proof}
The two equations have the same forward dynamics and the same terminal
condition. Lemma~\ref{lem:Bstab}, with
\[
h_t^{\lambda,P}-h_t^P
=\left\langle\varrho,
A^\lambda(X_{t-}^P,\mu_{t-}^P)
-A^0(X_{t-}^P,\mu_{t-}^P)\right\rangle,
\]
gives
\begin{align*}
&\E^P\sup_{t\le T}|Y_t^{\lambda,P}-Y_t^P|^2
+\E^P\int_0^T\{|Z_t^{\lambda,P}-Z_t^P|_{a_t^P}^2
+\|U_t^{\lambda,P}-U_t^P\|_\Pi^2\}dt\\
&\quad\le C|\varrho|^2\E^P\int_0^T
|D_{\lambda,t}^P|^2dt.
\end{align*}
Moreover,
\[
\sup_{t\le T}|\mathsf K_t^{\lambda,P}-\mathsf K_t^P|^2
\le T|\varrho|^2\int_0^T|D_{\lambda,t}^P|^2dt.
\]
Lemma~\ref{lem:Vitali} allows us to take the supremum over $P$ and then let
$\lambda$ tend to zero.
\end{proof}

\begin{proposition}[Strong--weak closure with varying law]\label{prop:closure}
Fix $P\in\Pcal$. Assume
\begin{align*}
X^n&\longrightarrow X
\quad\text{strongly in }L^2(dt\otimes P;\R^d),\\
\Gamma^n&\rightharpoonup\Gamma
\quad\text{weakly in }L^2(dt\otimes P;\R^d),\\
\Gamma_t^n&\in A(X_t^n,\Lcal^P(X_t^n))
\quad dt\otimes P\text{-a.e.}
\end{align*}
Then
\[
\Gamma_t\in A(X_t,\Lcal^P(X_t))
\quad dt\otimes P\text{-a.e.}
\]
\end{proposition}

\begin{proof}
Set $\mu_t^n=\Lcal^P(X_t^n)$ and $\mu_t=\Lcal^P(X_t)$. By
\eqref{eq:coupling},
\[
\int_0^T\Wtwo^2(\mu_t^n,\mu_t)dt
\le\E^P\int_0^T|X_t^n-X_t|^2dt\longrightarrow0.
\]
The bound \eqref{eq:Klip} yields
\begin{align*}
&\E^P\int_0^T
|F_{\mu_t^n}(X_t^n)-F_{\mu_t}(X_t)|^2dt\\
&\quad\le2L_K^2\E^P\int_0^T|X_t^n-X_t|^2dt
+2L_K^2\int_0^T\Wtwo^2(\mu_t^n,\mu_t)dt\longrightarrow0.
\end{align*}
The growth bound \eqref{eq:Kgrowth} guarantees that all preceding terms
belong to $L^2(dt\otimes P)$. Hence
\[
H^n:=\Gamma^n-F_{\mu^n}(X^n)
\rightharpoonup H:=\Gamma-F_\mu(X)
\quad\text{in }L^2(dt\otimes P),
\]
and $H_t^n\in\partial\phi(X_t^n)$.

Let $(V,W)$ be a pair in $L^2(dt\otimes P;\R^d)^2$ such that
$W_t\in\partial\phi(V_t)$ almost everywhere. Monotonicity gives
\[
\E^P\int_0^T\langle H_t^n-W_t,X_t^n-V_t\rangle dt\ge0.
\]
Writing, with $dt\otimes P$ omitted from the integrals,
\[
\int\langle H^n-W,X^n-V\rangle
=\int\langle H^n-W,X-V\rangle
+\int\langle H^n-W,X^n-X\rangle,
\]
weak convergence yields
$\int\langle H^n-W,X-V\rangle\to\int\langle H-W,X-V\rangle$.
Moreover,
\[
\left|\int\langle H^n-W,X^n-X\rangle\right|
\le\bigl(\sup_n\|H^n\|_{L^2}+\|W\|_{L^2}\bigr)
\|X^n-X\|_{L^2}\longrightarrow0.
\]
Therefore,
\[
\E^P\int_0^T\langle H_t-W_t,X_t-V_t\rangle dt\ge0.
\]
To conclude without invoking maximality of an operator on the function
space, set $J=(I+\partial\phi)^{-1}$ and choose
\[
V=J(X+H),\qquad W=X+H-V.
\]
The resolvent $J$ is defined everywhere and is $1$-Lipschitz; in
particular, $|J(x)|\le|x|+|J(0)|$. Hence
$V,W\in L^2(dt\otimes P;\R^d)$ and, by the definition of $J$,
$W\in\partial\phi(V)$ almost everywhere. This pair is therefore admissible
in the preceding inequality. Since $H-W=V-X$,
\[
0\le\E^P\int_0^T\langle H-W,X-V\rangle dt
=-\E^P\int_0^T|X-V|^2dt.
\]
Thus $X=V$, $H=W$, and hence $H\in\partial\phi(X)$ almost everywhere.
Since $\Gamma=H+F_\mu(X)$, the conclusion follows.
\end{proof}

\subsection{Regularization rate}

\begin{hypothesis}[Facewise regularity]\label{HfacesJ}
There exists a closed set $\Sigma\subset\R^d$, independent of $\mu$, such
that $\R^d\setminus\Sigma=\bigsqcup_{\ell\in\Lambda}O_\ell$, where the
$O_\ell$ are the connected components of $\R^d\setminus\Sigma$. On each
$O_\ell$, $A(\cdot,\mu)$ is single-valued and
\[
|A(x,\mu)-A(x',\mu)|\le L_{\rm face}|x-x'|,
\qquad x,x'\in O_\ell,
\]
with a constant independent of $\ell$ and $\mu$.
\end{hypothesis}

\begin{hypothesis}[Quantitative occupation]\label{HoccJ}
There exist $\theta,\delta>0$ and $C_{\rm occ}<\infty$ such that
\begin{align}
\sup_{P\in\Pcal}\E^P\int_0^T
|A^0(X_{t-}^P,\mu_{t-}^P)|^{2+\delta}dt&<\infty,
\label{eq:occmoment}\\
\sup_{P\in\Pcal}\E^P\int_0^T
\one_{\{\operatorname{dist}(X_{t-}^P,\Sigma)\le r\}}dt
&\le C_{\rm occ}r^\theta,\qquad 0<r\le1.
\label{eq:occ}
\end{align}
\end{hypothesis}

For the ALA prototype, under \ref{HqJ}, \eqref{eq:occmoment} is automatic
for every $\delta\in(0,q-2]$. Indeed, \ref{HD}, Jensen's inequality and
\eqref{eq:Fq} give
\[
\sup_{P\in\Pcal}\E^P\int_0^T
|A^0(X_{t-}^P,\mu_{t-}^P)|^{2+\delta}dt
\le C\left(1+\sup_P\E^P\sup_{t\le T}|X_t^P|^{2+\delta}\right)<\infty.
\]

In the prototype \eqref{eq:ALAphi}, Assumption~\ref{HfacesJ} holds with
the open orthants as faces. Assumption~\ref{HoccJ} does not follow from
the mere presence of jumps. Theorem~\ref{thm:occupation-jumps} verifies it
under an explicit condition on the continuous diffusion.

\begin{proposition}[Uniform qualitative non-contact]\label{prop:noncontactJ}
Under Assumptions~\ref{HF}, \ref{HB}, \ref{HD}, and~\ref{HfacesJ},
assume a uniform $(2+\delta)$-moment of the source,
$\delta>0$, and
\[
\omega(r):=\sup_P\E^P\int_0^T
\one_{\{\operatorname{dist}(X_{t-}^P,\Sigma)\le r\}}dt
\longrightarrow0\quad(r\downarrow0).
\]
Then \ref{HY} is satisfied and the convergence of
Theorem~\ref{thm:Yconv} is uniform over the model family.
\end{proposition}
\begin{proof}
Fix $r\in(0,1]$. On
$G=\{\operatorname{dist}(X_{-}^P,\Sigma)>r,
\lambda|A^0|\le r/2\}$, the resolvent identity gives
$|J_\lambda-X_{-}^P|\le r/2$. The two points lie in the same connected
ball disjoint from $\Sigma$, hence on the same face. Therefore
$|A^\lambda-A^0|\le L_{\rm face}\lambda|A^0|$.
On $G^c$, the domination $|A^\lambda-A^0|\le2|A^0|$ gives
\begin{align*}
\sup_P\E^P\int_0^T|A^\lambda-A^0|^2dt
\le C\lambda^2
&+4\sup_P\E^P\int_0^T|A^0|^2
\one_{\{\operatorname{dist}(X_{-}^P,\Sigma)\le r\}}dt\\
&+4\sup_P\E^P\int_0^T|A^0|^2
\one_{\{|A^0|>r/(2\lambda)\}}dt.
\end{align*}
Hölder's inequality with exponents $(2+\delta)/2$ and
$(2+\delta)/\delta$ bounds the first exceptional term by
$C\omega(r)^{\delta/(2+\delta)}$. The inequality
$|a|^2\one_{|a|>M}\le M^{-\delta}|a|^{2+\delta}$ bounds the second by
$C(\lambda/r)^\delta$. For fixed $r$, take the
$\limsup_{\lambda\downarrow0}$ and then let $r\downarrow0$. Uniform
$L^2$ convergence follows. Finally, Markov's inequality gives
\[
\sup_P(dt\otimes P)(|A^\lambda-A^0|>\varepsilon)
\le\varepsilon^{-2}\sup_P\E^P\int_0^T|A^\lambda-A^0|^2dt\to0,
\]
which verifies the second clause of \ref{HY}; the first is assumed.
\end{proof}

\begin{hypothesis}[Continuous non-degeneracy at the ALA interfaces]
\label{HellJ}
In the ALA prototype, set
$\sigma_t^P=\sigma(t,X_{t-}^P,\mu_{t-}^P,\alpha_t)$ and
$c_t^P=\sigma_t^Pa_t^P(\sigma_t^P)^\top$.
There exists $c_*>0$, independent of $P$, such that
$(c_t^P)_{kk}\ge c_*$ for $k=1,\ldots,d$, $dt\otimes P$-almost
everywhere. Non-degeneracy of $a^P$ alone is not sufficient when $\sigma$
is degenerate.
\end{hypothesis}

\begin{theorem}[Occupation of the interfaces in the presence of jumps]
\label{thm:occupation-jumps}
Let $\alpha\in\mathcal A^2$. Under Assumptions~\ref{HF}, \ref{HD},
and~\ref{HellJ},
\begin{equation}
\sup_P\E^P\int_0^T
\one_{\{\operatorname{dist}(X_{t-}^P,\Sigma)\le r\}}dt
\le C r,\qquad 0<r\le1.
\label{eq:occ-jumps-proved}
\end{equation}
If \ref{HqJ} holds, then \ref{HoccJ} is satisfied with $\theta=1$ and
any $\delta\in(0,q-2]$. The jump coefficients may depend on the state,
the law, and the control.
\end{theorem}
\begin{proof}
\emph{Step 1: convex test function.}
Fix $P,k,r$ and write $V_t=X_t^{P,k}$,
$b_t^k=b_k(t,X_{t-}^P,\mu_{t-}^P,\alpha_t)$ and
$v_t(e)=\beta_k(t,X_{t-}^P,\mu_{t-}^P,\alpha_t,e)$.
For $\psi_r(x)=\sqrt{x^2+r^2}$,
\[
|\psi_r'|\le1,\qquad
0\le\psi_r''(x)=\frac{r^2}{(x^2+r^2)^{3/2}}\le r^{-1},
\quad
\psi_r''(x)\ge\frac{1}{2\sqrt2\,r}\one_{\{|x|\le r\}}.
\]
The Taylor remainder is exactly
\[
R_r(x,v)=\psi_r(x+v)-\psi_r(x)-\psi_r'(x)v
=v^2\int_0^1(1-u)\psi_r''(x+uv)du.
\]
Hence $0\le R_r(x,v)\le v^2/(2r)$.

\emph{Step 2: integral formula and integrability.}
Itô's formula with jumps yields
\begin{align*}
\psi_r(V_T)-\psi_r(V_0)
={}&\int_0^T\psi_r'(V_{t-})b_t^kdt
+\frac12\int_0^T\psi_r''(V_{t-})(c_t^P)_{kk}dt\\
&+\int_0^T\int_E R_r(V_{t-},v_t(e))\Pi(de)dt+M_T,
\end{align*}
with
\begin{align*}
M_T={}&\int_0^T\psi_r'(V_{t-})(\sigma_t^P)_{k,\cdot}dB_t\\
&+\int_0^T\int_E
[\psi_r(V_{t-}+v_t(e))-\psi_r(V_{t-})]\widetilde N^P(dt,de).
\end{align*}
The quadratic growth of the coefficients and \eqref{eq:Fsecond} give
\[
\sup_P\E^P\int_0^T
\bigl(|b_t|^2+\|\sigma_t^P\|^2+\|\beta_t\|_\Pi^2\bigr)dt<\infty.
\]
The integrands of $M$ are therefore square integrable, since
$|\psi_r(x+v)-\psi_r(x)|\le|v|$. In particular $\E^PM_T=0$.
The $R_r$ term is integrable for each $r>0$ by the bound
$R_r\le v^2/(2r)$; the second-derivative term is integrable as well.
These bounds justify the formula by localization and passage to
expectation: the martingales converge in $L^2$, the drift terms by
domination, and the positive remainders by monotone convergence.

\emph{Step 3: uniform estimate.}
Positivity of $R_r$ and the Lipschitz property of $\psi_r$ give
\begin{align*}
\frac12\E^P\int_0^T\psi_r''(V_{t-})(c_t^P)_{kk}dt
&\le\E^P[\psi_r(V_T)-\psi_r(V_0)]
+\E^P\int_0^T|b_t^k|dt\\
&\le\E^P|V_T-V_0|+\E^P\int_0^T|b_t^k|dt\le C_0,
\end{align*}
where $C_0$ is uniform in $P,k,r$ by Cauchy--Schwarz and
\eqref{eq:Fsecond}. Hence
\[
\E^P\int_0^T\one_{\{|X_{t-}^{P,k}|\le r\}}dt
\le \frac{4\sqrt2 C_0}{c_*}r.
\]
Since $\Sigma$ is the union of the coordinate hyperplanes,
$\{\operatorname{dist}(X_{t-}^P,\Sigma)\le r\}
=\bigcup_k\{|X_{t-}^{P,k}|\le r\}$. The union bound gives
\eqref{eq:occ-jumps-proved} with $C=4\sqrt2 dC_0/c_*$.

\emph{Step 4: source moment.}
Under \ref{HqJ}, for $2+\delta\le q$, the growth condition and Jensen's
inequality yield
\[
\sup_P\E^P\int_0^T|A^0(X_{t-}^P,\mu_{t-}^P)|^{2+\delta}dt
\le C\left(1+\sup_P\E^P\sup_{t\le T}|X_t^P|^{2+\delta}\right)<\infty.
\]
This verifies both clauses of \ref{HoccJ} with $\theta=1$.
\end{proof}

\begin{proposition}[Density criterion]\label{prop:density}
In the prototype \eqref{eq:ALAphi}, assume that each coordinate of
$X_{t-}^P$ admits a density $p_{t,k}^P$ and that
\[
\sup_{P\in\Pcal}\max_{1\le k\le d}
\int_0^T\|p_{t,k}^P\|_\infty dt\le C_{\rm dens}.
\]
Then \eqref{eq:occ} holds with $\theta=1$ and
$C_{\rm occ}=2dC_{\rm dens}$.
\end{proposition}

\begin{proof}
By \eqref{eq:Sigma},
\[
\{\operatorname{dist}(X_{t-}^P,\Sigma)\le r\}
\subset\bigcup_{k=1}^d\{|X_{t-}^{P,k}|\le r\}.
\]
By the union bound,
\[
P(\operatorname{dist}(X_{t-}^P,\Sigma)\le r)
\le\sum_{k=1}^d\int_{-r}^rp_{t,k}^P(x)dx
\le2r\sum_{k=1}^d\|p_{t,k}^P\|_\infty.
\]
Tonelli's theorem and the supremum over $P$ yield the result.
\end{proof}

\begin{proposition}[Yosida rate]\label{prop:Yrate}
Under Assumptions~\ref{HF}, \ref{HB}, \ref{HD}, \ref{HfacesJ},
and~\ref{HoccJ}, with
\begin{equation}
\gamma:=\min\left\{2,\frac{\theta\delta}{2+\delta+\theta}\right\},
\label{eq:gamma}
\end{equation}
for $0<\lambda\le1$ we have
\begin{equation}
\sup_{P\in\Pcal}\E^P\int_0^T
|A^\lambda(X_{t-}^P,\mu_{t-}^P)
-A^0(X_{t-}^P,\mu_{t-}^P)|^2dt
\le C\lambda^\gamma.
\label{eq:Arate}
\end{equation}
Consequently, the left-hand side of \eqref{eq:Yconv} is bounded by
$C\lambda^\gamma$.
\end{proposition}

\begin{proof}
Fix $P$ and write $X=X_{t-}^P$, $\mu=\mu_{t-}^P$,
$A^\lambda=A^\lambda(X,\mu)$ and $A^0=A^0(X,\mu)$. For $r\in(0,1]$,
set
\[
G_{\lambda,r}=\{\operatorname{dist}(X,\Sigma)>r,\ 
\lambda|A^0|\le r/2\}.
\]
On this set,
$|J_\lambda(X,\mu)-X|=\lambda|A^\lambda|\le r/2$.
Since $\operatorname{dist}(X,\Sigma)>r$, the ball $B(X,r)$ is contained
in $\R^d\setminus\Sigma$. It is connected; hence $X$ and
$J_\lambda(X,\mu)$ belong to the same component $O_\ell$. On that
component, $A(\cdot,\mu)$ is single-valued and
$A^\lambda(X,\mu)=A(J_\lambda(X,\mu),\mu)$. Therefore,
\[
|A^\lambda-A^0|
=|A(J_\lambda(X,\mu),\mu)-A(X,\mu)|
\le L_{\rm face}\lambda|A^0|.
\]
On the complement, $|A^\lambda-A^0|^2\le4|A^0|^2$. Thus
\begin{align}
|A^\lambda-A^0|^2
&\le L_{\rm face}^2\lambda^2|A^0|^2
+4|A^0|^2\one_{\{\operatorname{dist}(X,\Sigma)\le r\}}\nonumber\\
&\quad+4|A^0|^2\one_{\{|A^0|>r/(2\lambda)\}}.
\label{eq:three}
\end{align}
Hölder's inequality with exponents $(2+\delta)/2$ and
$(2+\delta)/\delta$ yields
\[
\E^P\int_0^T|A^0|^2
\one_{\{\operatorname{dist}(X,\Sigma)\le r\}}dt
\le Cr^{\theta\delta/(2+\delta)}.
\]
The truncation
$|A^0|^2\one_{\{|A^0|>M\}}\le M^{-\delta}|A^0|^{2+\delta}$
gives
\[
\E^P\int_0^T|A^0|^2
\one_{\{|A^0|>r/(2\lambda)\}}dt
\le C(\lambda/r)^\delta.
\]
After integrating \eqref{eq:three},
\[
\sup_P\E^P\int_0^T|A^\lambda-A^0|^2dt
\le C\{\lambda^2+r^{\theta\delta/(2+\delta)}
+(\lambda/r)^\delta\}.
\]
The choice
$r=\lambda^{(2+\delta)/(2+\delta+\theta)}$ balances the last two terms
and gives \eqref{eq:Arate}. The final conclusion follows from the proof of
Theorem~\ref{thm:Yconv}. In particular, \eqref{eq:Arate} implies
Assumption~\ref{HY}.
\end{proof}

\section{Particle system and propagation of chaos}

In this section and in the section devoted to the joint approximation,
Assumption~\ref{HqJ} is imposed in addition to the assumptions stated in
each result. The constants are uniform for data whose moments of order $q$
are uniformly bounded.

\begin{hypothesis}[Product compatibility]\label{Hprod}
For each $P\in\Pcal$ and each $N\ge1$, under $P^{\otimes N}$ we are given
independent copies
\[
(\xi^i,B^i,N^i,\alpha^i),\qquad 1\le i\le N,
\]
of $(\xi,B,N,\alpha)$. The control is product-compatible: there exists a
predictable functional $\mathfrak a$ such that
\[
\alpha_t=\mathfrak a_t(\xi,B_{\cdot\wedge t},
N|_{[0,t)\times E}),\qquad
\alpha_t^i=\mathfrak a_t(\xi^i,B^i_{\cdot\wedge t},
N^i|_{[0,t)\times E}).
\]
On $\Omega^N$, with projections $\pi_i$, we use the raw filtration
$\mathcal G_t^{0,N}=\bigvee_{i=1}^N\pi_i^{-1}(\mathcal F_t^0)$ and then
its usual augmentation under $P^{\otimes N}$, denoted by
$\mathcal F_t^{P,\otimes N}$. More precisely, if $\mathcal N^{P,N}$ is
the ideal of null sets of the terminal product sigma-field,
\[
\mathcal F_t^{P,\otimes N}
=\bigcap_{t<u\le T}(\mathcal G_u^{0,N}\vee\mathcal N^{P,N})\quad(t<T),
\qquad
\mathcal F_T^{P,\otimes N}=\mathcal G_T^{0,N}\vee\mathcal N^{P,N}.
\]
The controls are generated by a predictable functional for the individual
raw filtration. By assumption, this product filtration has the
representation property with respect to
$(B^1,\ldots,B^N,N^1,\ldots,N^N)$.
\end{hypothesis}

This class contains deterministic controls and predictable functionals of
the individual noise. In the coupling, the control $\alpha^i$ is the same
for $X^{i,N}$ and $\bar X^i$. Hence a feedback control
$\alpha_t^{i,N}=a(t,X_{t-}^{i,N},\mu_{t-}^N)$ does not automatically
belong to this class: it must be incorporated into the coefficients and
their Lipschitz assumptions must then be verified. No common noise is
imposed on the different product coordinates.

Under $P^{\otimes N}$, the density of the quadratic variation of $B^j$ is
denoted by $a_t^{j,P}=a_t^P(\omega^j)$. Martingales associated with two
distinct coordinates are orthogonal. For estimates on the product space,
we use the normalized norms
\begin{align*}
|x|_N^2&=\frac1N\sum_{i=1}^N|x_i|^2,\\
\|z\|_{N,a}^2&=\frac1N\sum_{i=1}^N\sum_{j=1}^N
|z^{i,j}|_{a^{j,P}}^2,\\
\|u\|_{N,\Pi}^2&=\frac1N\sum_{i=1}^N\sum_{j=1}^N
\|u^{i,j}\|_\Pi^2.
\end{align*}

\begin{definition}[Propagation of chaos used here]\label{def:chaos}
Let $V^P$ be a limiting process and, for each $N$, let
$(V^{i,N})_{1\le i\le N}$ be an exchangeable family. We say that quadratic
propagation of chaos holds uniformly over $\Pcal$ if, on the product space,
there exist independent copies $\bar V^i$ of $V^P$, with a jointly
exchangeable coupling $((V^{i,N},\bar V^i))_{i\le N}$, such that
\[
\lim_{N\to\infty}\sup_{P\in\Pcal}\frac1N\sum_{i=1}^N
\E^{P^{\otimes N}}d(V^{i,N},\bar V^i)^2=0,
\]
where $d$ is the metric specified in the result under consideration. For
the state component, we take
$d(X,\bar X)=\sup_{t\le T}|X_t-\bar X_t|$. This property implies that,
for every fixed $k$ and $t$, the law of
$(V_t^{1,N},\ldots,V_t^{k,N})$ converges to
$\Lcal^P(V_t^P)^{\otimes k}$ in quadratic Wasserstein distance. Indeed,
the coupling
$(V_t^{1,N},\ldots,V_t^{k,N};\bar V_t^1,\ldots,\bar V_t^k)$ gives
\[
\Wtwo^2\left(\Lcal^{P^{\otimes N}}(V_t^{1,N},\ldots,V_t^{k,N}),
\Lcal^P(V_t^P)^{\otimes k}\right)
\le\sum_{i=1}^k\E^{P^{\otimes N}}|V_t^{i,N}-\bar V_t^i|^2.
\]

Joint exchangeability gives
\[
\sum_{i=1}^k\E|V_t^{i,N}-\bar V_t^i|^2
=\frac{k}{N}\sum_{i=1}^N\E|V_t^{i,N}-\bar V_t^i|^2.
\]
It is satisfied by the couplings used in this paper: the coefficients are
invariant under permutations of the indices, the data are i.i.d., and
strong uniqueness identifies the permuted solution. More generally, a
coupling can be symmetrized by applying an independent uniform random
permutation; its marginals and its mean error remain unchanged.
\end{definition}

Under $P^{\otimes N}$, write
\[
\widetilde N^i(dt,de)=N^i(dt,de)-\Pi(de)dt.
\]
The forward particle system is
\begin{align}
X_t^{i,N}
={}&\xi^i+\int_0^t b(s,X_{s-}^{i,N},\mu_{s-}^{N},\alpha_s^i)ds
+\int_0^t\sigma(s,X_{s-}^{i,N},\mu_{s-}^{N},\alpha_s^i)dB_s^i\nonumber\\
&+\int_0^t\int_E\beta(s,X_{s-}^{i,N},\mu_{s-}^{N},
\alpha_s^i,e)\widetilde N^i(ds,de),\label{eq:particleF}\\
\mu_t^N&=\frac1N\sum_{j=1}^N\delta_{X_t^{j,N}}.
\label{eq:muN}
\end{align}

For $\lambda>0$, the regularized backward particle system is
\begin{align}
Y_t^{i,N,\lambda}
={}&g(X_T^{i,N},\mu_T^N)
+\int_t^T f(s,X_{s-}^{i,N},\mu_{s-}^N,Y_s^{i,N,\lambda},
Z_s^{i,i,N,\lambda},U_s^{i,i,N,\lambda},
\eta_s^{N,\lambda},\alpha_s^i)ds\nonumber\\
&+\mathsf K_T^{i,N,\lambda}-\mathsf K_t^{i,N,\lambda}
-\sum_{j=1}^N\int_t^TZ_s^{i,j,N,\lambda}dB_s^j\nonumber\\
&-\sum_{j=1}^N\int_t^T\int_EU_s^{i,j,N,\lambda}(e)
\widetilde N^j(ds,de),\label{eq:particleB}\\
\eta_t^{N,\lambda}
={}&\frac1N\sum_{j=1}^N\delta_{Y_t^{j,N,\lambda}},\label{eq:etaN}\\
\mathsf K_t^{i,N,\lambda}
={}&\int_0^t\langle\varrho,
A^\lambda(X_{s-}^{i,N},\mu_{s-}^N)\rangle ds.
\label{eq:particleK}
\end{align}
The generator uses the diagonal components $Z^{i,i,N,\lambda}$ and
$U^{i,i,N,\lambda}$, while the martingale representation retains all
off-diagonal components.

\begin{proposition}[Well-posedness of the particle system]\label{prop:particlewell}
Under Assumptions~\ref{HF}, \ref{HB}, \ref{HD} and~\ref{Hprod}, the
system \eqref{eq:particleF}--\eqref{eq:particleK} has a unique
square-integrable solution for every $N$ and $\lambda>0$.  The higher
moment assumption \ref{HqJ}, imposed later for quantitative particle
rates, is not needed for this well-posedness statement.
\end{proposition}

\begin{proof}
For fixed $N$, the forward particle vector is a finite-dimensional jump SDE.
The empirical-law coupling
\[
\Wtwo^2(\mu^N(x),\mu^N(\bar x))
\le \frac1N\sum_{j=1}^N|x_j-\bar x_j|^2
\]
turns \ref{HF} into a global Lipschitz estimate on the product space.
Picard iteration, Lemma~\ref{lem:mart}, and Gronwall therefore give a unique
forward solution in the quadratic product space.

Condition on this forward system.  The regularized source is square
integrable by Proposition~\ref{prop:Yosida} and the forward moment estimate.
For a given input vector of backward processes, solve the $N$ terminal-value
martingale representation problems using \ref{Hprod}.  The resulting map on
the product of the weighted spaces
$\mathbb S^2\times\mathbb H_B^2\times\mathbb H_\Pi^2$ satisfies, by the same
jump Itô calculation as in Theorem~\ref{thm:Bwell},
\[
\|\Psi_N(v)-\Psi_N(\bar v)\|_\gamma^2
\le \frac{C L_B^2}{\gamma}\|v-\bar v\|_\gamma^2.
\]
The empirical-law terms are controlled by the natural particle coupling and
the constant is independent of $N$.  Choosing $\gamma$ sufficiently large
makes $\Psi_N$ a contraction.  Banach's theorem gives the unique backward
tuple, including all matrix components $(Z^{i,j,N,\lambda})_{i,j}$ and
$(U^{i,j,N,\lambda})_{i,j}$.  The cumulative process is then defined by
\eqref{eq:particleK}.  This proves quadratic well-posedness; no use of
\ref{HqJ} is made.
\end{proof}

\subsection{Moments and empirical rate}

Define
\begin{equation}
\tau_d(N)=
\begin{cases}
N^{-1/2},&d<4,\\
N^{-1/2}\log(1+N),&d=4,\\
N^{-2/d},&d>4.
\end{cases}
\label{eq:tau}
\end{equation}
The theorem of Fournier--Guillin \cite{FournierGuillin2015} gives, for a
law $m\in\mathcal P_q(\R^d)$ and its empirical sample $m^N$,
\begin{equation}
\E\Wtwo^2(m^N,m)\le C_{d,q}(1+M_q(m)^2)\tau_d(N),
\qquad q>4.
\label{eq:FG}
\end{equation}
In the general formulation of the cited result, a tail term
$N^{-(q-2)/q}$ is added to the principal rate. Since $q>4$, its exponent
is strictly larger than $1/2$; this term is therefore dominated by
$\tau_d(N)$ in all three dimension regimes. In dimension one, the
right-hand side is thus $CN^{-1/2}$.

\begin{lemma}[Uniform moments of the particles and copies]
\label{lem:moments}
Under \ref{HF}, \ref{HB}, \ref{HD}, \ref{HqJ} and \ref{Hprod},
\begin{align}
\sup_{N,P}\frac1N\sum_{i=1}^N
\E^{P^{\otimes N}}\sup_{t\le T}|X_t^{i,N}|^q&<\infty,\label{eq:XNmom}\\
\sup_{\lambda\in[0,1]}\sup_{P\in\Pcal}
\E^P\sup_{t\le T}|Y_t^{\lambda,P}|^q&<\infty.
\label{eq:Ymom}
\end{align}
In \eqref{eq:Ymom}, $\lambda=0$ denotes the selected solution.
\end{lemma}

\begin{proof}
For \eqref{eq:XNmom}, first justify that the moments being manipulated are
finite. Return to the Picard iterates $X^{i,N,n}$ of
Proposition~\ref{prop:particlewell} and set
\[
u_n(t)=\frac1N\sum_i\E\sup_{s\le t}|X_s^{i,N,n}|^q.
\]
The growth of the coefficients, Hölder's inequality, BDG and \eqref{eq:BJ}
give, since the empirical measures are those of the previous iterate,
\[
u_{n+1}(t)\le C_0+C\int_0^t u_n(s)ds,\qquad
C_0=C\left(1+\E|\xi|^q+\E\int_0^T|\alpha_s|^qds\right).
\]
With the initialization $X^{i,N,0}=\xi^i$, this inequality shows by
induction that $u_n(T)<\infty$. For
$v_J(t)=\max_{0\le n\le J}u_n(t)$, after increasing $C_0$ to include
$u_0$, we obtain $v_J(t)\le C_0+C\int_0^tv_J(s)ds$, hence
$v_J(T)\le C_0e^{CT}$. These constants are independent of $P,N,J$. For
fixed $P,N$, a subsequence of the iterates converges uniformly almost
surely to the already constructed forward particle system. Fatou's lemma
yields the $q$-moment of this limit with the same uniform bound. We may
therefore apply directly the Brownian estimates of order $q$ and
\eqref{eq:BJ} to \eqref{eq:particleF}, and then average over $i$.
The law terms reduce to
\[
M_q(\mu_t^N)^q=\frac1N\sum_{j=1}^N|X_t^{j,N}|^q.
\]
Thus
\[
\frac1N\sum_i\E\sup_{r\le t}|X_r^{i,N}|^q
\le C\left(1+\E|\xi|^q+\E\int_0^t|\alpha_s|^qds
+\int_0^t\frac1N\sum_i\E\sup_{r\le s}|X_r^{i,N}|^qds\right).
\]
For $u_N(t)=N^{-1}\sum_i\E\sup_{r\le t}|X_r^{i,N}|^q$ and
$h(t)=C[1+\E|\xi|^q+\E\int_0^t|\alpha_s|^qds]$,
\eqref{eq:GronwallExplicit} gives $u_N(T)\le h(T)e^{CT}$.
The constant is independent of $N,P$; this proves \eqref{eq:XNmom}.

For \eqref{eq:Ymom}, fix $P$ and $0\le\lambda\le1$. Write
\begin{align*}
\zeta^P&=g(X_T^P,\mu_T^P),\\
\ell_t^P&=
|f(t,X_{t-}^P,\mu_{t-}^P,0,0,0,\delta_0,\alpha_t)|
+|\varrho|\,|A^0(X_{t-}^P,\mu_{t-}^P)|,\\
m_t^P&=M_2(\Lcal^P(Y_t^{\lambda,P})).
\end{align*}
The property $|A^\lambda|\le|A^0|$, \eqref{eq:fLip}, and
$|z|\le\lambda_{\min}(\underline a)^{-1/2}|z|_{a_t^P}$ give
\begin{equation}
|F_t^\lambda|
\le\ell_t^P+C\bigl(|Y_t^{\lambda,P}|
+|Z_t^{\lambda,P}|_{a_t^P}
+\|U_t^{\lambda,P}\|_\Pi+m_t^P\bigr).
\label{eq:Fqlin}
\end{equation}

We establish simultaneously membership in the order-$q$ spaces and the
uniform bounds directly from the integral formulation. On an interval
$I=[s_0,s_1]$, set
\begin{align*}
\|Y\|_{\mathbb S^q(I)}^q
&=\E^P\sup_{t\in I}|Y_t|^q,\\
\|Z\|_{\mathbb H_B^q(I)}^q
&=\E^P\left(\int_I|Z_t|_{a_t^P}^2dt\right)^{q/2},\\
\|U\|_{\mathbb J_\Pi^q(I)}^q
&=\E^P\left(\int_I\|U_t\|_\Pi^2dt\right)^{q/2}
+\E^P\int_I\|U_t\|_{L_\Pi^q}^qdt.
\end{align*}
For an input $(y,z,u)$ in the product of these three spaces, the
construction by conditional expectation and martingale representation used
in the proof of Theorem~\ref{thm:Bwell} defines the linear output solution.
More precisely, if the terminal value at the right endpoint of $I$ is
$\chi\in L^q(\mathcal F_{s_1}^P)$, set
\[
G_t^{y,z,u}=f(t,X_{t-}^P,\mu_{t-}^P,y_t,z_t,u_t,
\Lcal^P(y_t),\alpha_t)+\langle\varrho,
A^\lambda(X_{t-}^P,\mu_{t-}^P)\rangle
\]
and
\[
M_t=\E^P\left[\chi+\int_{s_0}^{s_1}G_r^{y,z,u}dr
\,\middle|\,\mathcal F_t^P\right],\qquad t\in I.
\]
The representation property gives the unique pair $(Z,U)$ such that
\[
M_t=M_{s_0}+\int_{s_0}^tZ_r\,dB_r
+\int_{s_0}^t\int_EU_r(e)\widetilde N^P(dr,de),
\]
and
\[
Y_t=M_t-\int_{s_0}^tG_r^{y,z,u}dr.
\]
For two inputs with the same terminal value, Doob's inequality and
Lemma~\ref{lem:martp} first give
\begin{align*}
&\|\Delta Y\|_{\mathbb S^q(I)}^q
+\|\Delta Z\|_{\mathbb H_B^q(I)}^q
+\|\Delta U\|_{\mathbb J_\Pi^q(I)}^q
\le C_q\E^P\left(\int_I|\Delta G_t|dt\right)^q.
\end{align*}
By \eqref{eq:fLip}, Hölder's inequality and \eqref{eq:coupling},
\begin{align}
&\|\Delta Y\|_{\mathbb S^q(I)}^q
+\|\Delta Z\|_{\mathbb H_B^q(I)}^q
+\|\Delta U\|_{\mathbb J_\Pi^q(I)}^q\nonumber\\
&\quad\le C_qL_B^q
\bigl(|I|^q+|I|^{q/2}\bigr)
\left(
\|\Delta y\|_{\mathbb S^q(I)}^q
+\|\Delta z\|_{\mathbb H_B^q(I)}^q
+\|\Delta u\|_{\mathbb J_\Pi^q(I)}^q
\right).
\label{eq:Lqcontraction}
\end{align}
Indeed,
\[
\left(\int_I|\Delta y_t|dt\right)^q
\le |I|^q\sup_{t\in I}|\Delta y_t|^q,
\quad
\left(\int_I|\Delta z_t|dt\right)^q
\le |I|^{q/2}
\left(\int_I|\Delta z_t|^2dt\right)^{q/2},
\]
and the same estimate holds for $\|\Delta u_t\|_\Pi$; the difference of
the laws is controlled explicitly by
\[
\Wtwo(\Lcal^P(y_t),\Lcal^P(\bar y_t))
\le(\E^P|y_t-\bar y_t|^2)^{1/2}
\le\|y-\bar y\|_{\mathbb S^q(I)}.
\]
Let $\|\cdot\|_{q,I}$ denote the $q$th root of the sum of the three norm
powers. This is a complete norm on the product of the preceding spaces;
the first component is adapted càdlàg and the other two are predictable.
The generator, integrated against $dt$, may use $y$ or its predictable
version $y_-$ without changing the equation. Choose $h\in(0,T]$ such that
\[
C_qL_B^q(h^q+h^{q/2})\le 2^{-q}.
\]
The constants include the norm equivalence for $z$ induced by
$\underline a$; $h$ is independent of $P$ and $\lambda$. For $|I|\le h$,
the map $\Psi_{I,\chi}$ constructed above is therefore a contraction with
ratio at most $1/2$ in norm. For the zero input,
$|G_t^{0,0,0}|\le\ell_t^P$. Doob's inequality followed by
\eqref{eq:martplower} applied to $M_t-M_{s_0}$ gives
\[
\|\Psi_{I,\chi}(0)\|_{q,I}^q
\le C_q\E^P\left[|\chi|^q+
\left(\int_I\ell_t^Pdt\right)^q\right].
\]
The same estimates with an arbitrary input show that
$\Psi_{I,\chi}$ maps this complete space into itself. The fixed point
theorem therefore yields $v_I=(Y,Z,U)$ and
\[
\|v_I\|_{q,I}\le\tfrac12\|v_I\|_{q,I}
+\|\Psi_{I,\chi}(0)\|_{q,I}.
\]
There is thus $K\ge1$, independent of $I,P,\lambda$, such that
\begin{equation}
\|v_I\|_{q,I}^q\le
K\E^P\left[|\chi|^q+\left(\int_I\ell_t^Pdt\right)^q\right].
\label{eq:Yq-local}
\end{equation}
The law term has already been absorbed into the contraction
\eqref{eq:Lqcontraction}; no prior moment bound on the solution has been
used.

Take a deterministic partition $0=t_0<\cdots<t_J=T$ with mesh at most
$h$, where $J=\lceil T/h\rceil$. Starting with $\chi=\zeta^P$ on the last
interval, solve successively backward, taking as terminal value on
$I_j=[t_{j-1},t_j]$ the already constructed value $Y_{t_j}$. Endpoint
values agree; the integral identities concatenate and define a global
solution. It belongs to the quadratic spaces and therefore coincides, by
Theorem~\ref{thm:Bwell}, with the selected or regularized solution already
constructed.

Set $a_j=\|v\|_{q,I_j}^q$, $a_{J+1}=\E^P|\zeta^P|^q$ and
$b_j=\E^P(\int_{I_j}\ell_t^Pdt)^q$. Since
$\E^P|Y_{t_j}|^q\le a_{j+1}$ for $j<J$, \eqref{eq:Yq-local} gives
\[
a_j\le K(a_{j+1}+b_j)
\le K^{J-j+1}a_{J+1}+\sum_{l=j}^JK^{l-j+1}b_l.
\]
Summing and using $K\ge1$ and $\ell\ge0$,
\[
\sum_{j=1}^Ja_j\le JK^J\left(a_{J+1}+\sum_{j=1}^Jb_j\right)
\le JK^J\E^P\left[|\zeta^P|^q+
\left(\int_0^T\ell_t^Pdt\right)^q\right].
\]
The global supremum is bounded by the sum of the local suprema, hence
\begin{equation}
\E^P\sup_{t\le T}|Y_t^{\lambda,P}|^q
\le C_q\E^P\left[|\zeta^P|^q+
\left(\int_0^T\ell_t^Pdt\right)^q\right].
\label{eq:YqSup}
\end{equation}
For the energy terms, the inequality
$(\sum_{j=1}^Jc_j)^{q/2}\le J^{q/2-1}\sum_{j=1}^Jc_j^{q/2}$,
$c_j\ge0$, similarly gives
\begin{align}
&\E^P\left(\int_0^T|Z_t^{\lambda,P}|_{a_t^P}^2dt\right)^{q/2}
+\E^P\left(\int_0^T\|U_t^{\lambda,P}\|_\Pi^2dt\right)^{q/2}
+\E^P\int_0^T\|U_t^{\lambda,P}\|_{L_\Pi^q}^qdt\nonumber\\
&\qquad\le C_q\E^P\left[|\zeta^P|^q+
\left(\int_0^T\ell_t^Pdt\right)^q\right].
\label{eq:ZUq}
\end{align}
Finally, \eqref{eq:Fq}, \eqref{eq:fLip}, \eqref{eq:gLip},
\ref{HD} and \eqref{eq:controlq} uniformly bound the right-hand side of
\eqref{eq:YqSup}--\eqref{eq:ZUq}. Taking the supremum over
$(P,\lambda)$ gives \eqref{eq:Ymom}.
\end{proof}

\subsection{Forward propagation of chaos}

On the product space, let $\bar X^i$ be the limiting solution
\eqref{eq:FWD} driven by $(\xi^i,B^i,N^i,\alpha^i)$ and set
\[
\bar\mu_t^N=\frac1N\sum_{i=1}^N\delta_{\bar X_t^i}.
\]

\begin{theorem}[Forward propagation of chaos]\label{thm:Fchaos}
Under Assumptions~\ref{HF}, \ref{HD}, \ref{HqJ} and~\ref{Hprod},
\begin{align}
\sup_{P\in\Pcal}\frac1N\sum_{i=1}^N
\E^{P^{\otimes N}}\sup_{t\le T}|X_t^{i,N}-\bar X_t^i|^2
&\le C\tau_d(N),\label{eq:Fchaos}\\
\sup_{P\in\Pcal}\sup_{t\le T}
\E^{P^{\otimes N}}\Wtwo^2(\mu_t^N,\mu_t^P)
&\le C\tau_d(N).\label{eq:mulaw}
\end{align}
\end{theorem}

\begin{proof}
Set $\Delta X^i=X^{i,N}-\bar X^i$ and
\[
D_X(t)=\frac1N\sum_i\E^{P^{\otimes N}}
\sup_{r\le t}|\Delta X_r^i|^2.
\]
The initial data and the controls are identical within each pair.
Corollary~\ref{cor:Fstabrandom}, applied to the random flows $\mu^N$ and
$\mu^P$, gives
\[
D_X(t)\le C\int_0^t
\left[D_X(s)+\E^{P^{\otimes N}}
\Wtwo^2(\mu_s^N,\mu_s^P)\right]ds.
\]
The atomic coupling $N^{-1}\sum_i\delta_{(X_s^{i,N},\bar X_s^i)}$ and the
inequality $(a+b)^2\le2a^2+2b^2$ give
\begin{equation}
\Wtwo^2(\mu_s^N,\mu_s^P)
\le\frac2N\sum_i|\Delta X_s^i|^2
+2\Wtwo^2(\bar\mu_s^N,\mu_s^P).
\label{eq:WdecompX}
\end{equation}
The $\bar X_s^i$ are i.i.d. with law $\mu_s^P$. By \eqref{eq:Fq}, their
moments of order $q$ are uniform in $(s,P)$; \eqref{eq:FG} therefore
yields
\[
\sup_{s,P}\E^{P^{\otimes N}}
\Wtwo^2(\bar\mu_s^N,\mu_s^P)\le C\tau_d(N).
\]
Hence
\[
D_X(t)\le C\int_0^tD_X(s)ds+C\tau_d(N),
\]
and \eqref{eq:GronwallExplicit}, with $h(t)=C\tau_d(N)$, gives
$D_X(t)\le C e^{Ct}\tau_d(N)\le C e^{CT}\tau_d(N)$.
The constants are independent of $P,N$, proving \eqref{eq:Fchaos}.
Returning to \eqref{eq:WdecompX}, then taking the supremum over $t$ and
$P$, yields \eqref{eq:mulaw}.
\end{proof}

Thus \eqref{eq:Fchaos} establishes forward propagation of chaos in the
sense of Definition~\ref{def:chaos}, uniformly over $\Pcal$; every finite
marginal therefore converges at each fixed time.

\subsection{Backward propagation of chaos}

Let $(\bar Y^{i,\lambda},\bar Z^{i,\lambda},\bar U^{i,\lambda})$ be the
copy of the regularized limiting solution driven by the individual noise
$i$. Its accumulated process is
\[
\bar{\mathsf K}_t^{i,\lambda}
=\int_0^t\langle\varrho,
A^\lambda(\bar X_{s-}^i,\mu_{s-}^P)\rangle ds.
\]
In the product filtration it is represented by
\[
\bar Z^{i,j,\lambda}=\one_{\{i=j\}}\bar Z^{i,\lambda},
\qquad
\bar U^{i,j,\lambda}=\one_{\{i=j\}}\bar U^{i,\lambda}.
\]
Set
\[
\bar\eta_t^{N,\lambda}
=\frac1N\sum_i\delta_{\bar Y_t^{i,\lambda}}.
\]

\begin{theorem}[Backward propagation with full matrices]
\label{thm:Bchaos}
Under Assumptions~\ref{HF}, \ref{HB}, \ref{HD}, \ref{HqJ} and~\ref{Hprod},
for $0<\lambda\le1$,
\begin{align}
&\sup_{P\in\Pcal}\frac1N\sum_i\E^{P^{\otimes N}}\Biggl[
\sup_{t\le T}|Y_t^{i,N,\lambda}-\bar Y_t^{i,\lambda}|^2
+\sup_{t\le T}|\mathsf K_t^{i,N,\lambda}
-\bar{\mathsf K}_t^{i,\lambda}|^2\nonumber\\
&\quad+\sum_{j=1}^N\int_0^T
|Z_t^{i,j,N,\lambda}
-\one_{\{i=j\}}\bar Z_t^{i,\lambda}|_{a_t^{j,P}}^2dt\nonumber\\
&\quad+\sum_{j=1}^N\int_0^T
\|U_t^{i,j,N,\lambda}
-\one_{\{i=j\}}\bar U_t^{i,\lambda}\|_\Pi^2dt
\Biggr]\le C(1+\lambda^{-2})\tau_d(N).
\label{eq:Bchaos}\\
&\sup_{P\in\Pcal}\sup_{t\le T}
\E^{P^{\otimes N}}\Wtwo^2(\eta_t^{N,\lambda},\eta_t^{P,\lambda})
\le C(1+\lambda^{-2})\tau_d(N).
\label{eq:etalaw}
\end{align}
\end{theorem}

\begin{proof}
Fix $P$. Write
\begin{align*}
\Delta Y^i&=Y^{i,N,\lambda}-\bar Y^{i,\lambda},\\
\Delta Z^{i,j}&=Z^{i,j,N,\lambda}
-\one_{\{i=j\}}\bar Z^{i,\lambda},\\
\Delta U^{i,j}&=U^{i,j,N,\lambda}
-\one_{\{i=j\}}\bar U^{i,\lambda}.
\end{align*}
The terminal difference $\Delta\zeta^i$ satisfies
\begin{equation}
|\Delta\zeta^i|^2\le
2L_B^2\{|X_T^{i,N}-\bar X_T^i|^2
+\Wtwo^2(\mu_T^N,\mu_T^P)\}.
\label{eq:terminal}
\end{equation}
The source difference is
\[
\Delta A_t^i=A^\lambda(X_{t-}^{i,N},\mu_{t-}^N)
-A^\lambda(\bar X_{t-}^i,\mu_{t-}^P),
\]
and \eqref{eq:AYlip} gives
\begin{equation}
|\Delta A_t^i|
\le\lambda^{-1}|X_{t-}^{i,N}-\bar X_{t-}^i|
+L_K\Wtwo(\mu_{t-}^N,\mu_{t-}^P).
\label{eq:sourceparticle}
\end{equation}

Let $\Delta F_t^i$ denote the difference of the generator $f$ together
with the source $\langle\varrho,A^\lambda\rangle$. By \eqref{eq:fLip},
\begin{align}
|\Delta F_t^i|
&\le L_B\bigl(
|\Delta X_t^i|+\Wtwo(\mu_t^N,\mu_t^P)+|\Delta Y_t^i|
+|\Delta Z_t^{i,i}|+\|\Delta U_t^{i,i}\|_\Pi\nonumber\\
&\hspace{39mm}+\Wtwo(\eta_t^{N,\lambda},\eta_t^{P,\lambda})
\bigr)\nonumber\\
&\quad+|\varrho|\{\lambda^{-1}|\Delta X_t^i|
+L_K\Wtwo(\mu_t^N,\mu_t^P)\}.
\label{eq:Fdiff}
\end{align}

Apply Itô's formula with jumps to $e^{\gamma t}|\Delta Y_t^i|^2$ and
then take expectations. The quadratic part is
\[
\sum_j\int_t^Te^{\gamma s}|\Delta Z_s^{i,j}|_{a_s^{j,P}}^2ds
+\sum_j\int_t^T\int_Ee^{\gamma s}
|\Delta U_s^{i,j}(e)|^2N^j(ds,de).
\]
After expectation, the second term becomes
$\sum_j\E\int e^{\gamma s}\|\Delta U_s^{i,j}\|_\Pi^2ds$.
Young's inequality allows one quarter of the diagonal components of $Z$
and $U$ to be absorbed. There exists a constant $C_0$, independent of
$N,P,\lambda$, such that
\begin{align}
&e^{\gamma t}\E|\Delta Y_t^i|^2
+\frac34\sum_j\E\int_t^Te^{\gamma s}
\{|\Delta Z_s^{i,j}|_{a_s^{j,P}}^2
+\|\Delta U_s^{i,j}\|_\Pi^2\}ds\nonumber\\
&\quad+(\gamma-C_0)\E\int_t^Te^{\gamma s}|\Delta Y_s^i|^2ds\nonumber\\
&\le e^{\gamma T}\E|\Delta\zeta^i|^2
+C(1+\lambda^{-2})\E\int_t^Te^{\gamma s}
\{|\Delta X_s^i|^2+\Wtwo^2(\mu_s^N,\mu_s^P)\}ds\nonumber\\
&\quad+C\E\int_t^Te^{\gamma s}
\Wtwo^2(\eta_s^{N,\lambda},\eta_s^{P,\lambda})ds.
\label{eq:BenergyN}
\end{align}
The equivalence \eqref{eq:qv} was used to replace
$|\Delta Z^{i,i}|^2$ by a constant times
$|\Delta Z^{i,i}|_{a^{i,P}}^2$.

Set
\[
D_Y(t)=\frac1N\sum_i\E|\Delta Y_t^i|^2.
\]
The atomic coupling and the i.i.d. copies give
\begin{equation}
\E\Wtwo^2(\eta_t^{N,\lambda},\eta_t^{P,\lambda})
\le2D_Y(t)+2\E\Wtwo^2(\bar\eta_t^{N,\lambda},
\eta_t^{P,\lambda}).
\label{eq:WdecompY}
\end{equation}
By Lemma~\ref{lem:moments} and Fournier--Guillin in dimension one,
\[
\sup_{t,P,\lambda}\E\Wtwo^2(\bar\eta_t^{N,\lambda},
\eta_t^{P,\lambda})\le CN^{-1/2}\le C\tau_d(N).
\]
After averaging over $i$, choose $\gamma>C_0+2C+1$ in
\eqref{eq:BenergyN}. The term $2C\int D_Y$ arising from
\eqref{eq:WdecompY} is absorbed on the left-hand side. Estimates
\eqref{eq:Fchaos}, \eqref{eq:mulaw}, and \eqref{eq:terminal} yield
\begin{align}
&\sup_{t\le T}D_Y(t)
+\frac1N\sum_{i,j}\E\int_0^T
\{|\Delta Z_t^{i,j}|_{a_t^{j,P}}^2
+\|\Delta U_t^{i,j}\|_\Pi^2\}dt\nonumber\\
&\le C(1+\lambda^{-2})\tau_d(N).
\label{eq:BenergyFinal}
\end{align}

To control the supremum of $\Delta Y^i$, write the difference in integral
form. The martingales
\[
\sum_j\int_0^\cdot\Delta Z_s^{i,j}dB_s^j,\qquad
\sum_j\int_0^\cdot\int_E\Delta U_s^{i,j}(e)
\widetilde N^j(ds,de)
\]
are orthogonal. Doob's inequality and the isometries give
\begin{align*}
&\E\sup_{t\le T}\left|
\sum_j\int_0^t\Delta Z_s^{i,j}dB_s^j
+\sum_j\int_0^t\int_E\Delta U_s^{i,j}(e)
\widetilde N^j(ds,de)\right|^2\\
&\quad\le C\sum_j\E\int_0^T
\{|\Delta Z_s^{i,j}|_{a_s^{j,P}}^2
+\|\Delta U_s^{i,j}\|_\Pi^2\}ds.
\end{align*}
Together with
\[
\sup_{t\le T}|\Delta Y_t^i|^2
\le3|\Delta\zeta^i|^2
+3T\int_0^T|\Delta F_s^i|^2ds
+12\sup_{t\le T}|\Delta M_t^i|^2,
\]
\eqref{eq:Fdiff}, \eqref{eq:WdecompY}, and
\eqref{eq:BenergyFinal} imply
\[
\frac1N\sum_i\E\sup_{t\le T}|\Delta Y_t^i|^2
\le C(1+\lambda^{-2})\tau_d(N).
\]

Finally,
\[
\sup_{t\le T}|\mathsf K_t^{i,N,\lambda}
-\bar{\mathsf K}_t^{i,\lambda}|^2
\le T|\varrho|^2\int_0^T|\Delta A_s^i|^2ds.
\]
The bound \eqref{eq:sourceparticle}, followed by
\eqref{eq:Fchaos}--\eqref{eq:mulaw}, gives the same order
$C(1+\lambda^{-2})\tau_d(N)$. This proves \eqref{eq:Bchaos};
\eqref{eq:etalaw} then follows from \eqref{eq:WdecompY}.
\end{proof}

For each fixed $\lambda>0$, \eqref{eq:Bchaos} in particular yields
propagation of chaos of $(Y^{i,N,\lambda})_i$ in the uniform-in-time
metric of Definition~\ref{def:chaos}. The estimates on $Z$ and $U$ are
stronger than mere convergence of the laws of $Y$: they control the full
martingale representation in the product filtration.

\begin{corollary}[Convergence of laws on path space]
\label{cor:pathchaos}
Equip $D([0,T];\R^k)$ with the Skorokhod $J_1$ distance, denoted by
$d_{J_1}$. On products of $\ell$ trajectories, we use the quadratic
distance $((x_i),(y_i))\mapsto
(\sum_{i=1}^\ell d_{J_1}(x_i,y_i)^2)^{1/2}$, and
$W_2(\cdot,\cdot;d_{J_1})$ denotes the associated Wasserstein distance.
For every fixed $\lambda>0$ and every integer $\ell\ge1$,
\begin{align*}
\sup_{P\in\Pcal}
W_2^2\Bigl(&\Lcal^{P^{\otimes N}}
((X^{1,N},\ldots,X^{\ell,N})),
\Lcal^P(X^P)^{\otimes\ell};d_{J_1}\Bigr)
\longrightarrow0,\\
\sup_{P\in\Pcal}
W_2^2\Bigl(&\Lcal^{P^{\otimes N}}
((Y^{1,N,\lambda},\ldots,Y^{\ell,N,\lambda})),
\Lcal^P(Y^{\lambda,P})^{\otimes\ell};d_{J_1}\Bigr)
\longrightarrow0.
\end{align*}
The same conclusion holds for the accumulated processes
$\mathsf K^{i,N,\lambda}$, which are continuous.
\end{corollary}

\begin{proof}
In the definition of $d_{J_1}$, choosing the time-change function equal to
the identity gives, for $x,\bar x\in D([0,T];\R^k)$,
\[
d_{J_1}(x,\bar x)\le\sup_{t\le T}|x_t-\bar x_t|.
\]
Couple the first $\ell$ particles with the independent copies used in
Theorems~\ref{thm:Fchaos} and~\ref{thm:Bchaos}. By the definition of the
Wasserstein distance on the product space,
\begin{align*}
&W_2^2\Bigl(\Lcal((X^{1,N},\ldots,X^{\ell,N})),
\Lcal(X^P)^{\otimes\ell};d_{J_1}\Bigr)\\
&\qquad\le\sum_{i=1}^\ell
\E^{P^{\otimes N}}\sup_{t\le T}|X_t^{i,N}-\bar X_t^i|^2.
\end{align*}
Exchangeability turns the right-hand side into $\ell$ times the mean error
of Theorem~\ref{thm:Fchaos}; it converges uniformly to zero. The same
construction yields, for $V=Y^{\lambda}$ or $V=\mathsf K^{\lambda}$,
\[
\sup_P W_2^2\bigl(\Lcal^{P^{\otimes N}}(V^{1,N},\ldots,V^{\ell,N}),
\Lcal^P(V)^{\otimes\ell};d_{J_1}\bigr)
\le C\ell(1+\lambda^{-2})\tau_d(N)\longrightarrow0,
\]
for fixed $\ell$ and $\lambda>0$, by \eqref{eq:Bchaos}. This conclusion
concerns the path laws of $X$, $Y$, and $\mathsf K$; the integrands $Z$
and $U$ converge in the energy norms displayed above, not in a Skorokhod
topology.
\end{proof}

\section{Generic joint particle--Yosida approximation}

Let
$(\bar X^i,\bar Y^i,\bar Z^i,\bar U^i,\bar{\mathsf K}^i)$ be an independent
copy of the selected solution of Theorem~\ref{thm:well}, constructed on the
$i$th coordinate.

\begin{theorem}[Qualitative diagonal convergence]\label{thm:diagonal}
Assume Hypotheses~\ref{HF}, \ref{HB}, \ref{HD}, \ref{HqJ},
\ref{HY}, and~\ref{Hprod}. Let $\lambda_N\downarrow0$ be such that
\[
(1+\lambda_N^{-2})\tau_d(N)\longrightarrow0.
\]
An explicit choice is $\lambda_N=\tau_d(N)^{1/4}$.
Then
\begin{align}
\lim_{N\to\infty}\sup_{P\in\Pcal}\frac1N\sum_i
\E^{P^{\otimes N}}\Biggl[
&\sup_{t\le T}|Y_t^{i,N,\lambda_N}-\bar Y_t^i|^2
+\sup_{t\le T}|\mathsf K_t^{i,N,\lambda_N}
-\bar{\mathsf K}_t^i|^2\nonumber\\
&+\sum_j\int_0^T
|Z_t^{i,j,N,\lambda_N}
-\one_{\{i=j\}}\bar Z_t^i|_{a_t^{j,P}}^2dt\nonumber\\
&+\sum_j\int_0^T
\|U_t^{i,j,N,\lambda_N}
-\one_{\{i=j\}}\bar U_t^i\|_\Pi^2dt
\Biggr]=0.
\label{eq:diagonal}
\end{align}
\end{theorem}

\begin{proof}
Set $V^{i,N,\lambda}=(Y^{i,N,\lambda},\mathsf K^{i,N,\lambda},
(Z^{i,j,N,\lambda})_j,(U^{i,j,N,\lambda})_j)$ and
\[
\overline V^{i,\lambda}
=(\bar Y^{i,\lambda},\bar{\mathsf K}^{i,\lambda},
(\one_{\{i=j\}}\bar Z^{i,\lambda})_j,
(\one_{\{i=j\}}\bar U^{i,\lambda})_j).
\]
On this product space, $\|\cdot\|_{P,N}^2$ denotes the sum of the four
quadratic norms appearing in \eqref{eq:diagonal}. The identity
\[
V^{i,N,\lambda}-\overline V^{i,0}
=(V^{i,N,\lambda}-\overline V^{i,\lambda})
+(\overline V^{i,\lambda}-\overline V^{i,0})
\]
together with Theorems~\ref{thm:Bchaos} and~\ref{thm:Yconv}, yields
\begin{align*}
\sup_P\frac1N\sum_i
\|V^{i,N,\lambda_N}-\overline V^{i,0}\|_{P,N}^2
&\le2\sup_P\frac1N\sum_i
\|V^{i,N,\lambda_N}-\overline V^{i,\lambda_N}\|_{P,N}^2
+2\sup_P\frac1N\sum_i
\|\overline V^{i,\lambda_N}-\overline V^{i,0}\|_{P,N}^2\\
&\le C(1+\lambda_N^{-2})\tau_d(N)+2r(\lambda_N)
\longrightarrow0.
\end{align*}
Here $r(\lambda)$ is the uniform error from Theorem~\ref{thm:Yconv};
the equality in law of the copies preserves this bound on each product space.
\end{proof}

\begin{theorem}[Joint rate with jumps]\label{thm:joint}
Assume Hypotheses~\ref{HF}, \ref{HB}, \ref{HD}, \ref{HqJ},
\ref{Hprod}, \ref{HfacesJ}, and~\ref{HoccJ}. Let $\gamma$ be defined by
\eqref{eq:gamma} and
\begin{equation}
\lambda_N=\tau_d(N)^{1/(\gamma+2)}.
\label{eq:lambdaopt}
\end{equation}
Then, for $N$ sufficiently large,
\begin{align}
&\sup_{P\in\Pcal}\frac1N\sum_i
\E^{P^{\otimes N}}\Biggl[
\sup_{t\le T}|X_t^{i,N}-\bar X_t^i|^2
+\sup_{t\le T}|Y_t^{i,N,\lambda_N}-\bar Y_t^i|^2\nonumber\\
&\quad+\sup_{t\le T}|\mathsf K_t^{i,N,\lambda_N}
-\bar{\mathsf K}_t^i|^2
+\sum_j\int_0^T
|Z_t^{i,j,N,\lambda_N}
-\one_{\{i=j\}}\bar Z_t^i|_{a_t^{j,P}}^2dt\nonumber\\
&\quad+\sum_j\int_0^T
\|U_t^{i,j,N,\lambda_N}
-\one_{\{i=j\}}\bar U_t^i\|_\Pi^2dt
\Biggr]\le C\tau_d(N)^{\gamma/(\gamma+2)}.
\label{eq:joint}\\
&\sup_{P\in\Pcal}\sup_{t\le T}
\E^{P^{\otimes N}}\left[
\Wtwo^2(\mu_t^N,\mu_t^P)
+\Wtwo^2(\eta_t^{N,\lambda_N},\eta_t^P)
\right]\le C\tau_d(N)^{\gamma/(\gamma+2)}.
\label{eq:jointlaws}
\end{align}
\end{theorem}

\begin{proof}
Theorem~\ref{thm:Bchaos} and Proposition~\ref{prop:Yrate} give, for
$0<\lambda\le1$,
\begin{align}
\mathcal E_{N,\lambda}
\le C\left\{(1+\lambda^{-2})\tau_d(N)+\lambda^\gamma\right\},
\label{eq:preopt}
\end{align}
where $\mathcal E_{N,\lambda}$ denotes the sum of the errors in
$Y,\mathsf K,Z$, and $U$ appearing in \eqref{eq:joint}. Indeed, the
differences between the regularized copy and the selected copy are bounded by
$C\lambda^\gamma$, including the norm of $\bar U^{i,\lambda}-\bar U^i$.

With \eqref{eq:lambdaopt},
\[
\lambda_N^\gamma
=\tau_d(N)^{\gamma/(\gamma+2)},\qquad
\lambda_N^{-2}\tau_d(N)
=\tau_d(N)^{\gamma/(\gamma+2)}.
\]
Since $0<\gamma/(\gamma+2)<1$, the term $\tau_d(N)$ is dominated by the
same order for $N$ sufficiently large. The forward term is controlled by
\eqref{eq:Fchaos}, which yields \eqref{eq:joint}.

The first law estimate in \eqref{eq:jointlaws} follows from
\eqref{eq:mulaw}. For the second, set
$\bar\eta_t^N=N^{-1}\sum_i\delta_{\bar Y_t^i}$. The atomic coupling gives
\[
\Wtwo^2(\eta_t^{N,\lambda_N},\eta_t^P)
\le\frac2N\sum_i|Y_t^{i,N,\lambda_N}-\bar Y_t^i|^2
+2\Wtwo^2(\bar\eta_t^N,\eta_t^P).
\]
The first term is controlled by \eqref{eq:joint}. The second is at most
$CN^{-1/2}$ by \eqref{eq:Ymom} and the Fournier--Guillin estimate in dimension
one. Moreover,
$N^{-1/2}\le\tau_d(N)\le\tau_d(N)^{\gamma/(\gamma+2)}$ for all sufficiently
large $N$. This proves \eqref{eq:jointlaws}.
\end{proof}

\begin{corollary}[Stability of the particle approximation]
\label{cor:stablechaos}
Let $(\mathfrak D^n)_n$ be a sequence of data satisfying Hypotheses
~\ref{HF}, \ref{HB}, \ref{HD}, \ref{HqJ}, and~\ref{Hprod} uniformly. Fix
$\lambda>0$ and a sequence $N_n\to\infty$. Denote by
\[
\mathbf S^{i,N_n,n,\lambda}
:=\bigl(X^{i,N_n,n},Y^{i,N_n,n,\lambda},Z^{i,\cdot,N_n,n,\lambda},
U^{i,\cdot,N_n,n,\lambda},\mathsf K^{i,N_n,n,\lambda}\bigr)
\]
the particle system associated with $\mathfrak D^n$, and by
\[
\bar{\mathbf S}^{i,\lambda}
:=\bigl(\bar X^i,\bar Y^{i,\lambda},\bar Z^{i,\lambda},
\bar U^{i,\lambda},\bar{\mathsf K}^{i,\lambda}\bigr)
\]
independent copies of the limiting solution associated with $\mathfrak D$.
Define explicitly the normalized error
\begin{align*}
\mathcal Q_{n,N_n}^\lambda
:=\sup_{P\in\Pcal}\frac1{N_n}\sum_{i=1}^{N_n}
\E^{P^{\otimes N_n}}\Biggl[&
\sup_{t\le T}|X_t^{i,N_n,n}-\bar X_t^i|^2
+\sup_{t\le T}|Y_t^{i,N_n,n,\lambda}-\bar Y_t^{i,\lambda}|^2\\
&+\sup_{t\le T}|\mathsf K_t^{i,N_n,n,\lambda}
-\bar{\mathsf K}_t^{i,\lambda}|^2\\
&+\sum_{j=1}^{N_n}\int_0^T
|Z_t^{i,j,N_n,n,\lambda}
-\one_{\{i=j\}}\bar Z_t^{i,\lambda}|_{a_t^{j,P}}^2dt\\
&+\sum_{j=1}^{N_n}\int_0^T
\|U_t^{i,j,N_n,n,\lambda}
-\one_{\{i=j\}}\bar U_t^{i,\lambda}\|_\Pi^2dt\Biggr].
\end{align*}
Then
\begin{align}
\mathcal Q_{n,N_n}^\lambda
\le C\Bigl[&(1+\lambda^{-2})\tau_d(N_n)
+(1+\lambda^{-2})\sup_{P\in\Pcal}\mathfrak R_F^n(P)
+\sup_{P\in\Pcal}\mathfrak R_B^{n,\lambda}(P)\Bigr].
\label{eq:stablechaos}
\end{align}
In particular, if both residuals converge uniformly to zero, then
$\mathcal Q_{n,N_n}^\lambda\to0$.
\end{corollary}

\begin{proof}
For each $n$, construct, using the same individual noises, copies
$(\bar X^{i,n},\bar Y^{i,n,\lambda},\bar Z^{i,n,\lambda},
\bar U^{i,n,\lambda},\bar{\mathsf K}^{i,n,\lambda})$ of the limiting solution associated with $\mathfrak D^n$. For the forward
component,
\begin{align*}
\frac1{N_n}\sum_i\E\sup_{t\le T}
|X_t^{i,N_n,n}-\bar X_t^i|^2
&\le\frac2{N_n}\sum_i\E\sup_{t\le T}
|X_t^{i,N_n,n}-\bar X_t^{i,n}|^2\\
&\quad+\frac2{N_n}\sum_i\E\sup_{t\le T}
|\bar X_t^{i,n}-\bar X_t^i|^2.
\end{align*}
The first term is bounded by $C\tau_d(N_n)$ by
Theorem~\ref{thm:Fchaos}, uniformly in $n$ because the constants and the
moments of order $q$ are uniform. The second is bounded by
$C\sup_P\mathfrak R_F^n(P)$ by \eqref{eq:dataXstab}.

Apply the same decomposition to the components $Y$, $\mathsf K$, $Z$, and
$U$. For $Z$, for example,
\begin{align*}
&\frac1{N_n}\sum_{i,j}\E\int_0^T
|Z_t^{i,j,N_n,n,\lambda}-\one_{\{i=j\}}\bar Z_t^{i,\lambda}
|_{a_t^{j,P}}^2dt\\
&\quad\le\frac2{N_n}\sum_{i,j}\E\int_0^T
|Z_t^{i,j,N_n,n,\lambda}-\one_{\{i=j\}}
\bar Z_t^{i,n,\lambda}|_{a_t^{j,P}}^2dt\\
&\qquad+\frac2{N_n}\sum_i\E\int_0^T
|\bar Z_t^{i,n,\lambda}-\bar Z_t^{i,\lambda}|_{a_t^{i,P}}^2dt.
\end{align*}
Theorem~\ref{thm:Bchaos} controls the first difference by
$C(1+\lambda^{-2})\tau_d(N_n)$. Theorem~\ref{thm:datastability} controls the
second by
\[
C\left\{(1+\lambda^{-2})\sup_P\mathfrak R_F^n(P)
+\sup_P\mathfrak R_B^{n,\lambda}(P)\right\}.
\]
The calculations for $Y$, $U$, and $\mathsf K$ are identical in their
respective norms. Summing them gives \eqref{eq:stablechaos}.
\end{proof}

This corollary is a stability result for propagation of chaos in a setting
with fixed integrators. It does not include stability under convergence of
filtrations or driving martingales established in the preprint by
Papapantoleon, Saplaouras and Theodorakopoulos (arXiv:2506.03562).

Substituting \eqref{eq:tau} into \eqref{eq:joint} gives
$N^{-\gamma/[2(\gamma+2)]}$ for $d<4$,
$\{N^{-1/2}\log(1+N)\}^{\gamma/(\gamma+2)}$ for $d=4$, and
$N^{-2\gamma/[d(\gamma+2)]}$ for $d>4$; under the ALA density criterion,
$\theta=1$ and $\gamma=\delta/(3+\delta)$.

\section{Occupation-based direct propagation and cumulative-source estimates}
\label{sec:direct}

In this section, the operator is the prototype
\eqref{eq:ALAphi}--\eqref{eq:ALAK}. Assumptions
\ref{HF}, \ref{HB}, \ref{HD}, \ref{HqJ} and \ref{Hprod} remain in force.
The estimates concern $A^0$ directly; no Yosida parameter enters the
definition of the selected particle system.

\begin{hypothesis}[Occupation for the direct comparison]\label{HoccD}
There exist $\theta>0$ and $C_{\rm occ}<\infty$ such that
\[
\sup_{P\in\Pcal}\E^P\int_0^T
\one_{\{\operatorname{dist}(X_{t-}^P,\Sigma)\le r\}}dt
\le C_{\rm occ}r^\theta,\qquad 0<r\le1,
\]
where $\Sigma$ is defined in \eqref{eq:Sigma}.
\end{hypothesis}
This assumption is only the occupation clause of \ref{HoccJ}. The stronger
moments in \ref{HqJ} are used here for the empirical rate, not to truncate
the discontinuous part of the source.

\subsection{Comparison of the sources without regularization}

\begin{lemma}[Direct estimate by separation of the interfaces]
\label{lem:directsource}
Write
\[
A^0(x,\mu)=\alpha_Ax+F_\mu(x)+s(x,\mu),\qquad
M_s=\sqrt d\max\{\kappa_+,\kappa_-\}.
\]
Then $|s(x,\mu)|\le M_s$. For all $x,x',\mu,\mu'$ and $r>0$,
\begin{align}
|A^0(x',\mu')-A^0(x,\mu)|^2
\le C\bigl(&|x'-x|^2+\Wtwo^2(\mu',\mu)\nonumber\\
&+\one_{\{\operatorname{dist}(x,\Sigma)\le r\}}
+\one_{\{|x'-x|\ge r\}}\bigr).
\label{eq:pointdirect}
\end{align}
In particular, set
\[
D_{A,N}:=\sup_P\frac1N\sum_i\E^{P^{\otimes N}}\int_0^T
|A^0(X_{t-}^{i,N},\mu_{t-}^N)
-A^0(\bar X_{t-}^i,\mu_{t-}^P)|^2dt.
\]
Under \ref{HoccD},
\begin{equation}
D_{A,N}\le C\tau_d(N)^{\theta/(\theta+2)}.
\label{eq:directsource}
\end{equation}
\end{lemma}
\begin{proof}
\emph{Step 1: bounded discontinuous part.}
Formula \eqref{eq:A0explicit} places each coordinate of $s$ in
$[-\kappa_-,\kappa_+]$, including when $x_k=0$. Away from $\Sigma$, $s$
depends only on the signs of the coordinates. If
$\operatorname{dist}(x,\Sigma)>r$ and $|x'-x|<r$, the ball $B(x,r)$ is
contained in an orthant; hence $s(x',\mu')=s(x,\mu)$. On the complement,
$|s(x',\mu')-s(x,\mu)|\le2M_s$. Using
\[
|\alpha_A(x'-x)+F_{\mu'}(x')-F_\mu(x)|
\le(\alpha_A+L_K)|x'-x|+L_K\Wtwo(\mu',\mu),
\]
and then $|v+w|^2\le2|v|^2+2|w|^2$, we obtain
\eqref{eq:pointdirect}, with a constant depending only on
$d,\alpha_A,L_K,\kappa_+,\kappa_-$.

\emph{Step 2: integration on the product space.}
Apply \eqref{eq:pointdirect} to
$(x',\mu';x,\mu)=(X_{t-}^{i,N},\mu_{t-}^N;\bar X_{t-}^i,\mu_{t-}^P)$.
Each $\bar X^i$ has the law of the limiting forward process under $P$;
thus \ref{HoccD} controls the mean of the events near the interfaces.
Markov's inequality gives
\[
\frac1N\sum_i\E^{P^{\otimes N}}\int_0^T
\one_{\{|X_{t-}^{i,N}-\bar X_{t-}^i|\ge r\}}dt
\le r^{-2}\frac1N\sum_i\E^{P^{\otimes N}}\int_0^T
|X_{t-}^{i,N}-\bar X_{t-}^i|^2dt.
\]
Theorem~\ref{thm:Fchaos} therefore yields
\begin{equation}
D_{A,N}\le C\{\tau_d(N)+r^\theta+\tau_d(N)r^{-2}\},
\qquad 0<r\le1.
\label{eq:directbalance}
\end{equation}
Left-limit evaluations do not change the time integrals: every càdlàg
trajectory has at most a countable set of discontinuities. Tonelli's
theorem justifies this substitution under expectations.

\emph{Step 3: choice of the width.}
Whenever $\tau_d(N)\le1$, choose
$r=\tau_d(N)^{1/(\theta+2)}$. The last two terms in
\eqref{eq:directbalance} are equal to
$\tau_d(N)^{\theta/(\theta+2)}$; the first term is no larger than this
quantity. Any remaining small values of $N$ are absorbed into $C$ by the
quadratic moment bounds. This proves \eqref{eq:directsource}.
\end{proof}

\subsection{Selected system and full martingale representation}

With the forward system \eqref{eq:particleF}, define
\begin{align}
\mathsf K_t^{i,N,0}
&=\int_0^t\langle\varrho,A^0(X_{s-}^{i,N},\mu_{s-}^N)\rangle ds,
\qquad \eta_t^{N,0}=\frac1N\sum_i\delta_{Y_t^{i,N,0}},
\label{eq:directK}\\
Y_t^{i,N,0}
&=g(X_T^{i,N},\mu_T^N)
+\int_t^T f(s,X_{s-}^{i,N},\mu_{s-}^N,Y_s^{i,N,0},
Z_s^{i,i,N,0},U_s^{i,i,N,0},\eta_s^{N,0},\alpha_s^i)ds\nonumber\\
&\quad+\mathsf K_T^{i,N,0}-\mathsf K_t^{i,N,0}
-\sum_j\int_t^TZ_s^{i,j,N,0}dB_s^j
-\sum_j\int_t^T\int_EU_s^{i,j,N,0}(e)\widetilde N^j(ds,de).
\label{eq:directB}
\end{align}
The superscript $0$ denotes the minimal selection and not a formal
substitution into an estimate containing $\lambda^{-1}$.

\begin{proposition}[Well-posedness of the selected particle system]
\label{prop:directwell}
Under Assumptions~\ref{HF}, \ref{HB}, \ref{HD}, and~\ref{Hprod}, for the
ALA prototype the system \eqref{eq:directK}--\eqref{eq:directB} has a
unique solution in the quadratic product spaces.  Neither the occupation
assumption \ref{HoccD} nor the higher-moment assumption \ref{HqJ} is
needed for this well-posedness statement.
\end{proposition}
\begin{proof}
Fix $P,N$ and the already constructed forward system. The source
$h_t^i=\langle\varrho,A^0(X_{t-}^{i,N},\mu_{t-}^N)\rangle$
is predictable by measurability of $A^0$ and satisfies
\[
\frac1N\sum_i\E\int_0^T|h_t^i|^2dt
\le C\left(1+\frac1N\sum_i\E\sup_{t\le T}|X_t^{i,N}|^2\right)<\infty.
\]
For an input $(y,z,u)$ in the complete space endowed with the norm
$\|\cdot\|_{\gamma,N}$ defined in the proof of
Proposition~\ref{prop:particlewell}, set
\[
G_t^i=f(t,X_{t-}^{i,N},\mu_{t-}^N,y_t^i,z_t^{i,i},u_t^{i,i},
N^{-1}\textstyle\sum_j\delta_{y_t^j},\alpha_t^i)+h_t^i,
\quad \zeta^i=g(X_T^{i,N},\mu_T^N).
\]
The generator $G^i$ belongs to $L^2(dt\otimes P^{\otimes N})$.
By \ref{Hprod}, the martingale
$M_t^i=\E[\zeta^i+\int_0^TG_s^ids\mid\mathcal F_t^{P,\otimes N}]$
has the full integrands $(Z^{i,j},U^{i,j})_j$. The output is
$Y_t^i=M_t^i-\int_0^tG_s^ids$. Doob's inequality, the isometries and
Cauchy--Schwarz ensure that it belongs to the quadratic spaces.

For two inputs, the sources $h^i$ and terminal data cancel. The empirical
natural empirical coupling and \ref{HB} give
\[
\frac1N\sum_i|\Delta G_t^i|^2\le C
\left(|\Delta y_t|_N^2+
\frac1N\sum_{i,j}|\Delta z_t^{i,j}|_{a_t^{j,P}}^2+
\frac1N\sum_{i,j}\|\Delta u_t^{i,j}\|_\Pi^2\right).
\]
Itô's formula for $e^{\gamma t}|\Delta Y_t^i|^2$, followed by
$2|\Delta Y^i\Delta G^i|\le\gamma|\Delta Y^i|^2/2+
2|\Delta G^i|^2/\gamma$, gives, for $\gamma\ge2$,
\[
\|\Delta Y,\Delta Z,\Delta U\|_{\gamma,N}^2
\le\frac C\gamma\|\Delta y,\Delta z,\Delta u\|_{\gamma,N}^2.
\]
Choosing $\gamma>\max\{2,2C\}$ yields a unique fixed point. Convergence of
the integrands takes place in the quadratic norms: the isometries and
Doob's inequality make the stochastic integrals converge in
$\mathbb S^2$; Lipschitz continuity of $f$ makes the drift integrals
converge in the same space. The limit therefore satisfies
\eqref{eq:directB}. Uniqueness and invariance of the data under
permutations give exchangeability.
\end{proof}

Let $(\bar Y^i,\bar Z^i,\bar U^i,\bar{\mathsf K}^i)$ denote independent
copies of the selected limiting solution, constructed with
$(\xi^i,B^i,N^i,\alpha^i)$. Set
\begin{align*}
\Delta Y^i&=Y^{i,N,0}-\bar Y^i,
&\Delta\mathsf K^i&=\mathsf K^{i,N,0}-\bar{\mathsf K}^i,\\
\Delta Z^{i,j}&=Z^{i,j,N,0}-\one_{\{i=j\}}\bar Z^i,
&\Delta U^{i,j}&=U^{i,j,N,0}-\one_{\{i=j\}}\bar U^i.
\end{align*}
We use the quadratic error
\begin{align}
\mathcal D_N:=\sup_P\frac1N\sum_i\E^{P^{\otimes N}}\Bigl[
&\sup_{t\le T}|X_t^{i,N}-\bar X_t^i|^2
+\sup_{t\le T}|\Delta Y_t^i|^2
+\sup_{t\le T}|\Delta\mathsf K_t^i|^2\nonumber\\
&+\sum_j\int_0^T\{
|\Delta Z_t^{i,j}|_{a_t^{j,P}}^2+
\|\Delta U_t^{i,j}\|_\Pi^2\}dt\Bigr].
\label{eq:directerror}
\end{align}

\begin{theorem}[Direct propagation for the selected ALA source]
\label{thm:directchaos}
Under Assumptions~\ref{HF}, \ref{HB}, \ref{HD}, \ref{HqJ},
\ref{Hprod}, and~\ref{HoccD}, for the ALA prototype,
\begin{align}
\mathcal D_N&\le C\tau_d(N)^{\theta/(\theta+2)},
\label{eq:directchaos}\\
\sup_{P,t}\E^{P^{\otimes N}}\Wtwo^2(\eta_t^{N,0},\eta_t^P)
&\le C\tau_d(N)^{\theta/(\theta+2)}.
\label{eq:directlaw}
\end{align}
The constant is independent of $P,N$ and of any regularization parameter.
No uniform occupation estimate for the particles $X^{i,N}$ is required.
\end{theorem}
\begin{proof}
\emph{Step 1: integral differences and residuals.}
Fix $P$; all expectations below are under $P^{\otimes N}$. Set
$\Delta A_t^i=A^0(X_{t-}^{i,N},\mu_{t-}^N)
-A^0(\bar X_{t-}^i,\mu_{t-}^P)$,
$\Delta\zeta^i=g(X_T^{i,N},\mu_T^N)-g(\bar X_T^i,\mu_T^P)$,
and let $\Delta f^i$ denote the difference of the two generators $f$.
With $G^i=\Delta f^i+\langle\varrho,\Delta A^i\rangle$,
\[
\Delta Y_t^i=\Delta\zeta^i+\int_t^TG_s^ids-(M_T^i-M_t^i),
\quad
M_t^i=\sum_j\int_0^t\Delta Z_s^{i,j}dB_s^j
+\sum_j\int_0^t\int_E\Delta U_s^{i,j}(e)\widetilde N^j(ds,de).
\]
Define
\[
R_t^i=|X_{t-}^{i,N}-\bar X_{t-}^i|^2+
\Wtwo^2(\mu_{t-}^N,\mu_{t-}^P)+|\Delta A_t^i|^2,
\quad Q_t^i=\sum_j\{|\Delta Z_t^{i,j}|_{a_t^{j,P}}^2+
\|\Delta U_t^{i,j}\|_\Pi^2\}.
\]
Lemma~\ref{lem:directsource}, the terminal estimate \eqref{eq:terminal},
and Theorem~\ref{thm:Fchaos} give
\[
\frac1N\sum_i\E\left[|\Delta\zeta^i|^2+
\int_0^TR_t^idt\right]\le C b_N,
\quad b_N:=\tau_d(N)+\tau_d(N)^{\theta/(\theta+2)}.
\]

\emph{Step 2: energy identity with all noises.}
Itô's formula, orthogonality of the coordinates, and compensation yield
\begin{align}
&e^{\gamma t}\E|\Delta Y_t^i|^2+
\E\int_t^Te^{\gamma s}\{\gamma|\Delta Y_s^i|^2+Q_s^i\}ds\nonumber\\
&\qquad=e^{\gamma T}\E|\Delta\zeta^i|^2+
2\E\int_t^Te^{\gamma s}\Delta Y_s^iG_s^ids.
\label{eq:directenergy}
\end{align}
We spell out the passage from the localized formula to this identity. For
the continuous martingale on the product space, BDG and Cauchy--Schwarz
bound the expected supremum by
\[
C_\gamma(\E\sup_s|\Delta Y_s^i|^2)^{1/2}
\left(\E\sum_j\int_0^T|\Delta Z_s^{i,j}|_{a_s^{j,P}}^2ds\right)^{1/2}<\infty.
\]
For the jump martingale on the product space, the same bound holds after
replacing the continuous energy by
$\E\sum_j\int|\Delta U^{i,j}|^2N^j$, which equals the compensated energy.
These martingales are therefore in $H^1$; their stopped versions converge
in $L^1$ and have zero expectation. The drift terms are integrable by
Cauchy--Schwarz, while the positive quadratic terms pass to the limit by
monotone convergence. Finally, the terminal term passes by domination by
$e^{\gamma T}\sup_s|\Delta Y_s^i|^2$. The stopping times can be chosen
equal to $T$ once the localized quantities, finite almost surely, are below
the threshold. No fourth moment of $\Delta U$ is used.

\emph{Step 3: absorption of the law interaction.}
Let $w_s=\Wtwo(\eta_s^{N,0},\eta_s^P)$. By \ref{HB}, Young's inequality,
and \eqref{eq:qv},
\[
2\Delta Y_s^iG_s^i\le C_0|\Delta Y_s^i|^2+
\tfrac12 Q_s^i+C_1(R_s^i+w_s^2).
\]
This inequality retains the law term before expectation. Set
$D_Y(s)=N^{-1}\sum_i\E|\Delta Y_s^i|^2$ and
$\bar\eta_s^N=N^{-1}\sum_i\delta_{\bar Y_s^i}$. The coupling gives
\[
\E w_s^2\le2D_Y(s)+2\E\Wtwo^2(\bar\eta_s^N,\eta_s^P)
\le2D_Y(s)+CN^{-1/2}.
\]
The last bound uses \eqref{eq:Ymom} for the selected solution and
\eqref{eq:FG} in dimension one. Average \eqref{eq:directenergy} and choose
$\gamma\ge C_0+2C_1+1$. Then
\[
\sup_{t\le T}D_Y(t)+\frac1N\sum_i\E\int_0^T
(|\Delta Y_s^i|^2+Q_s^i)ds\le C(b_N+N^{-1/2})\le Cb_N.
\]
The constants depend on $T$ and the common constants in the assumptions,
but not on the number of martingale coordinates.

\emph{Step 4: time supremum and accumulated process.}
The Lipschitz property gives
$|G_s^i|^2\le C(R_s^i+|\Delta Y_s^i|^2+Q_s^i+w_s^2)$.
By the integral identity, $|M_T^i-M_t^i|\le2\sup_s|M_s^i|$, hence
\[
\E\sup_t|\Delta Y_t^i|^2\le3\E|\Delta\zeta^i|^2+
3T\E\int_0^T|G_s^i|^2ds+12\E\sup_t|M_t^i|^2.
\]
Doob's inequality and the isometries give
$\E\sup_t|M_t^i|^2\le4\E\int_0^TQ_s^ids$.
After averaging, the previous steps bound this supremum by $Cb_N$.
Moreover,
\[
\frac1N\sum_i\E\sup_t|\Delta\mathsf K_t^i|^2
\le T|\varrho|^2\frac1N\sum_i\E\int_0^T|\Delta A_s^i|^2ds
\le C\tau_d(N)^{\theta/(\theta+2)}.
\]
Together with the forward bound, this yields \eqref{eq:directchaos};
\eqref{eq:directlaw} follows from the coupling in Step 3. Exchangeability
and the coupling of the first $\ell$ particles also give, for every fixed
$\ell$, convergence of the path laws of $(X,Y,\mathsf K)$ under the
$J_1$ distance, by the argument of Corollary~\ref{cor:pathchaos}.
\end{proof}

\subsection{Improvement of the Yosida error in the prototype}

\begin{proposition}[Deterministic width of the regularization layer]
\label{prop:sharpALA}
Let $\alpha\in\mathcal A^2$. In the ALA prototype, under
Assumptions~\ref{HF}, \ref{HD}, and~\ref{HoccD},
\begin{equation}
\sup_P\E^P\int_0^T|A^\lambda(X_{t-}^P,\mu_{t-}^P)
-A^0(X_{t-}^P,\mu_{t-}^P)|^2dt
\le C(\lambda^2+\lambda^\theta),\quad 0<\lambda\le1.
\label{eq:sharpALA}
\end{equation}
The additional source moment appearing in \ref{HoccJ} is not required for
this estimate.
\end{proposition}
\begin{proof}
Fix $x,\mu$, set $y=J_\lambda(x,\mu)$, and let
$M_0=\max\{\kappa_+,\kappa_-\}+\eta$, where $\eta$ is the positive
parameter of the kernel \eqref{eq:ALAK}. There exists
$v\in\partial\sum_k\{\kappa_+(y_k)_++\kappa_-(y_k)_-\}$ such that
\[
A^\lambda(x,\mu)=\alpha_Ay+F_\mu(y)+v,
\qquad x=(1+\lambda\alpha_A)y+\lambda(F_\mu(y)+v).
\]
Each $|v_k|\le\max\{\kappa_+,\kappa_-\}$ and $|F_\mu(y)|\le\eta$.
If $x_k>\lambda M_0$, then $y_k>0$: otherwise the last identity would
give $x_k\le\lambda M_0$. If $x_k<-\lambda M_0$, the same reasoning
gives $y_k<0$. Consequently, outside
$\{\operatorname{dist}(x,\Sigma)\le\lambda M_0\}$, all signs agree and
$v=s(x,\mu)$. Everywhere, $|v-s(x,\mu)|\le2M_s$. Hence
\begin{align*}
|A^\lambda(x,\mu)-A^0(x,\mu)|^2
&\le2(\alpha_A+L_K)^2|y-x|^2+
8M_s^2\one_{\{\operatorname{dist}(x,\Sigma)\le\lambda M_0\}}\\
&\le C\lambda^2|A^0(x,\mu)|^2+
8M_s^2\one_{\{\operatorname{dist}(x,\Sigma)\le\lambda M_0\}}.
\end{align*}
The second inequality uses
$|y-x|=\lambda|A^\lambda(x,\mu)|\le\lambda|A^0(x,\mu)|$.
Integrate under $dt\otimes P$ and then take the supremum over $P$.
The growth of $A^0$ and \eqref{eq:Fsecond} control the first term;
\ref{HoccD} controls the second when $\lambda M_0\le1$.
If $\lambda M_0>1$, its measure is at most
$T\le TM_0^\theta\lambda^\theta$. This proves \eqref{eq:sharpALA}.
\end{proof}

\begin{theorem}[Joint approximation without Yosida amplification]
\label{thm:jointuniform}
Under the assumptions of Theorem~\ref{thm:directchaos}, for the ALA
prototype, set $a_\theta=\theta/(\theta+2)$ and, for $0<\lambda\le1$,
\begin{align*}
\Delta Y^{i,\lambda}&=Y^{i,N,\lambda}-\bar Y^i,&
\Delta\mathsf K^{i,\lambda}&=\mathsf K^{i,N,\lambda}-\bar{\mathsf K}^i,\\
\Delta Z^{i,j,\lambda}&=Z^{i,j,N,\lambda}-\one_{\{i=j\}}\bar Z^i,&
\Delta U^{i,j,\lambda}&=U^{i,j,N,\lambda}-\one_{\{i=j\}}\bar U^i.
\end{align*}
The barred copies are the selected limiting solutions, without
regularization. Define the full error
\begin{align}
\mathcal D_{N,\lambda}^{\rm reg}:=\sup_P\frac1N\sum_i
\E^{P^{\otimes N}}\Big[&\sup_{t\le T}|X_t^{i,N}-\bar X_t^i|^2
+\sup_{t\le T}|\Delta Y_t^{i,\lambda}|^2
+\sup_{t\le T}|\Delta\mathsf K_t^{i,\lambda}|^2\nonumber\\
&+\sum_j\int_0^T\big\{|\Delta Z_t^{i,j,\lambda}|_{a_t^{j,P}}^2
+\|\Delta U_t^{i,j,\lambda}\|_\Pi^2\big\}dt\Big].
\label{eq:regerrorfull}
\end{align}
There exists $C$, independent of $P,N,\lambda$, such that
\begin{equation}
\mathcal D_{N,\lambda}^{\rm reg}
\le C\{\tau_d(N)^{a_\theta}+\lambda^2+\lambda^\theta\}.
\label{eq:jointuniform}
\end{equation}
Only occupation of the limiting forward process is assumed; no occupation
estimate for the particles is required.
\end{theorem}
\begin{proof}
\emph{Step 1: a two-parameter pointwise inequality.}
The proof of Proposition~\ref{prop:sharpALA}, before integration, gives
\[
|A^\lambda(x',\mu')-A^0(x',\mu')|^2
\le C\lambda^2|A^0(x',\mu')|^2+
8M_s^2\one_{\{\operatorname{dist}(x',\Sigma)\le M_0\lambda\}}.
\]
The distance to $\Sigma$ is $1$-Lipschitz. For every $r>0$,
\[
\{\operatorname{dist}(x',\Sigma)\le M_0\lambda\}
\subset\{\operatorname{dist}(x,\Sigma)\le M_0\lambda+r\}
\cup\{|x'-x|\ge r\}.
\]
Inserting $A^0(x',\mu')$ and applying \eqref{eq:pointdirect}, we obtain
\begin{align}
|A^\lambda(x',\mu')-A^0(x,\mu)|^2
\le C\big(&\lambda^2|A^0(x',\mu')|^2+|x'-x|^2+\Wtwo^2(\mu',\mu)
\nonumber\\
&+\one_{\{\operatorname{dist}(x,\Sigma)\le M_0\lambda+r\}}
+\one_{\{|x'-x|\ge r\}}\big).
\label{eq:jointpoint}
\end{align}
Thus the layer is transferred to the limiting state before taking any
expectation.

\emph{Step 2: integration and choice of the width.}
Fix $P$ and write $\E=\E^{P^{\otimes N}}$. Set
\[
\Delta A_t^{i,\lambda}=A^\lambda(X_{t-}^{i,N},\mu_{t-}^N)
-A^0(\bar X_{t-}^i,\mu_{t-}^P),\qquad
S_{N,\lambda}=\sup_P\frac1N\sum_i\E\int_0^T|\Delta A_t^{i,\lambda}|^2dt.
\]
The growth of the selection and the uniform moments of the forward system
give
\[
\sup_{P,N}\frac1N\sum_i\E\int_0^T
|A^0(X_{t-}^{i,N},\mu_{t-}^N)|^2dt<\infty.
\]
For $s>1$, the measure of a layer is at most $T\le Ts^\theta$;
thus \ref{HoccD} extends to every $s>0$ after increasing its constant.
Theorem~\ref{thm:Fchaos} and Markov's inequality applied to
\eqref{eq:jointpoint} imply, with $\tau=\tau_d(N)$,
\begin{align*}
S_{N,\lambda}
&\le C\{\lambda^2+\tau+(M_0\lambda+r)^\theta+\tau r^{-2}\}\\
&\le C\{\lambda^2+\lambda^\theta+\tau+r^\theta+\tau r^{-2}\}.
\end{align*}
Here $(u+v)^\theta\le2^{(\theta-1)_+}(u^\theta+v^\theta)$.
If $\tau\le1$, choose $r=\tau^{1/(\theta+2)}$. Then
$r^\theta=\tau r^{-2}=\tau^{a_\theta}$ and
$\tau\le\tau^{a_\theta}$; the other values of $N$ are controlled by the
uniform moments. Hence
\begin{equation}
S_{N,\lambda}\le C\{\tau^{a_\theta}+\lambda^2+\lambda^\theta\}
=:C b_{N,\lambda}.
\label{eq:jointsource}
\end{equation}
No Lipschitz constant of $A^\lambda$ has been used.

\emph{Step 3: backward integral differences.}
Suppress the index $\lambda$ on the solution differences and set
$\Delta\zeta^i=g(X_T^{i,N},\mu_T^N)-g(\bar X_T^i,\mu_T^P)$.
If $\Delta f^i$ denotes the difference of the generators evaluated along
the two systems, then exactly
\begin{align*}
\Delta Y_t^i&=\Delta\zeta^i+\int_t^T G_s^ids-(M_T^i-M_t^i),
& G_s^i&=\Delta f_s^i+\langle\varrho,\Delta A_s^{i,\lambda}\rangle,\\
\Delta\mathsf K_t^i&=\int_0^t\langle\varrho,\Delta A_s^{i,\lambda}\rangle ds,\\
M_t^i&=\sum_j\int_0^t\Delta Z_s^{i,j}dB_s^j
+\sum_j\int_0^t\int_E\Delta U_s^{i,j}(e)\widetilde N^j(ds,de).
\end{align*}
Define
\[
Q_s^i=\sum_j\{|\Delta Z_s^{i,j}|_{a_s^{j,P}}^2+
\|\Delta U_s^{i,j}\|_\Pi^2\},\quad
R_s^i=|X_{s-}^{i,N}-\bar X_{s-}^i|^2+
\Wtwo^2(\mu_{s-}^N,\mu_{s-}^P)+|\Delta A_s^{i,\lambda}|^2.
\]
The terminal estimate and \eqref{eq:jointsource} give
\[
\frac1N\sum_i\E\Big[|\Delta\zeta^i|^2+\int_0^T R_s^ids\Big]
\le C b_{N,\lambda}.
\]

\emph{Step 4: energy identity and absorption of the law term.}
Set $w_s=\Wtwo(\eta_s^{N,\lambda},\eta_s^P)$ and
$D_Y(s)=N^{-1}\sum_i\E|\Delta Y_s^i|^2$. Assumption \ref{HB} and
Young's inequality give, with constants independent of $\lambda$,
\[
2\Delta Y_s^iG_s^i\le C_0|\Delta Y_s^i|^2+
\tfrac12Q_s^i+C_1(R_s^i+w_s^2).
\]
Coupling with $\bar\eta_s^N=N^{-1}\sum_i\delta_{\bar Y_s^i}$ and using
the order-$q>4$ moments of the selected solution yield
\[
\E w_s^2\le2D_Y(s)+2\E\Wtwo^2(\bar\eta_s^N,\eta_s^P)
\le2D_Y(s)+CN^{-1/2}.
\]
For $\gamma>C_0+2C_1+1$, Itô's formula and compensation give
\begin{align*}
&e^{\gamma t}\E|\Delta Y_t^i|^2+
\E\int_t^T e^{\gamma s}\{\gamma|\Delta Y_s^i|^2+Q_s^i\}ds\\
&\hspace{1cm}=e^{\gamma T}\E|\Delta\zeta^i|^2+
2\E\int_t^T e^{\gamma s}\Delta Y_s^iG_s^ids.
\end{align*}
This identity is justified in the quadratic spaces: by BDG and
Cauchy--Schwarz, the linear product martingales have an $H^1$ norm bounded
by
\[
C_\gamma(\E\sup_s|\Delta Y_s^i|^2)^{1/2}
(\E\int_0^TQ_s^ids)^{1/2}<\infty.
\]
The positive term $\sum_j\int|\Delta U^{i,j}|^2N^j$ is kept in its
uncompensated form before taking expectations. Its expectation is the
Poisson energy; no fourth moment of $U$ is needed. Localization, followed
by $L^1$ convergence of the martingales and monotone convergence of the
energy terms, justifies the identity.
Set $\overline Q_s=N^{-1}\sum_i\E Q_s^i$,
$\overline R_s=N^{-1}\sum_i\E R_s^i$ and
$\overline\zeta_2=N^{-1}\sum_i\E|\Delta\zeta^i|^2$.
After averaging and inserting $\E w_s^2\le2D_Y(s)+CN^{-1/2}$,
\begin{align*}
e^{\gamma t}D_Y(t)
+\int_t^Te^{\gamma s}\bigl[
(\gamma-C_0-2C_1)D_Y(s)+\tfrac12\overline Q_s\bigr]ds
&\le e^{\gamma T}\overline\zeta_2
+C_1\int_t^Te^{\gamma s}\overline R_sds
+CN^{-1/2}\frac{e^{\gamma T}-e^{\gamma t}}{\gamma}\\
&\le C(b_{N,\lambda}+N^{-1/2}).
\end{align*}
The choice $\gamma>C_0+2C_1+1$, and then $t=0$ for the integral terms,
gives
\[
\sup_tD_Y(t)+\int_0^T(D_Y(s)+\overline Q_s)ds
\le C(b_{N,\lambda}+N^{-1/2})\le Cb_{N,\lambda}.
\]
The last inequality uses $N^{-1/2}\le C\tau_d(N)$.

\emph{Step 5: maximal norms and uniform conclusion.}
The Lipschitz property of $f$ implies
$|G_s^i|^2\le C(R_s^i+|\Delta Y_s^i|^2+Q_s^i+w_s^2)$.
The integral formulation and Doob's inequality successively give
\begin{align*}
\E\sup_t|\Delta Y_t^i|^2
&\le3\E|\Delta\zeta^i|^2+3T\E\int_0^T|G_s^i|^2ds
+12\E\sup_t|M_t^i|^2,\\
\E\sup_t|M_t^i|^2&\le4\E|M_T^i|^2
=4\E\int_0^TQ_s^ids,\\
\frac1N\sum_i\E\sup_t|\Delta\mathsf K_t^i|^2
&\le T|\varrho|^2\frac1N\sum_i\E\int_0^T
|\Delta A_s^{i,\lambda}|^2ds\le Cb_{N,\lambda}.
\end{align*}
Average the first inequality and use Steps 2--4, then add the forward
error $C\tau_d(N)$. All constants are common to the models; taking their
supremum proves \eqref{eq:jointuniform}.
\end{proof}

\subsection{Strengthening through the second occupation moment}
\label{subsec:quadraticoccupation}

Theorem~\ref{thm:jointuniform} relies only on a first occupation moment and
remains valid under that weaker assumption. We now show that quadratic
control of the occupation duration makes it possible to work directly with
the accumulated process $\mathsf K$ and improves the exponent in
$\tau_d(N)$. The maximal-monotonicity, resolvent, minimal-selection and
measurability properties required below have already been established in
Proposition~\ref{prop:Yosida} and in \eqref{eq:A0explicit}; they are not
repeated.

\begin{hypothesis}[Second occupation moment of the limiting forward process]
\label{Hocc2}
There exist $\theta\in(0,1]$ and $C_{\rm occ,2}<\infty$ such that, for
$0<r\le1$,
\begin{equation}
\sup_{P\in\Pcal}\E^P\left[
\left(\int_0^T
\one_{\{\operatorname{dist}(X_{t-}^P,\Sigma)\le r\}}\,dt
\right)^2\right]
\le C_{\rm occ,2}r^{2\theta}.
\label{eq:Hocc2}
\end{equation}
This condition concerns only the limiting solution. It is stronger than
\ref{HoccD}; no occupation estimate for the $N$-particle system is assumed.
\end{hypothesis}

\begin{lemma}[Conditional criterion for \ref{Hocc2}]
\label{lem:conditional-occ2}
Let $V$ be an adapted càdlàg process and
$I_t(r)=\one_{\{\operatorname{dist}(V_{t-},\Sigma)\le r\}}$. Assume that
there exists $C_1<\infty$ such that, for every deterministic time
$s\in[0,T]$ and $0<r\le1$,
\[
\E\left[\int_s^T I_t(r)\,dt\,\middle|\,\mathcal F_s\right]
\le C_1r^\theta\qquad\text{a.s.}
\]
Then
\[
\E\left(\int_0^T I_t(r)\,dt\right)^2\le2C_1^2r^{2\theta}.
\]
\end{lemma}
\begin{proof}
Since the diagonal of $[0,T]^2$ is negligible, Tonelli's theorem and
symmetry give
\[
\left(\int_0^T I_t(r)dt\right)^2
=2\int_0^T I_s(r)\int_s^T I_t(r)dt\,ds.
\]
Since $I_s(r)$ is $\mathcal F_s$-measurable,
\[
\E\left[I_s(r)\int_s^T I_t(r)dt\right]
\le C_1r^\theta\E I_s(r).
\]
Integrating in $s$ and then using the conditional assumption at $s=0$
yields the stated bound. No independence between two occupation times is
required.
\end{proof}

\begin{lemma}[Pointwise control adapted to the accumulated source]
\label{lem:cumulative-point}
In the ALA prototype, there exists $C<\infty$, independent of
$x,x',\mu,\mu',r$ and $\lambda$, such that, for
$x,x'\in\R^d$, $\mu,\mu'\in\mathcal P_2(\R^d)$,
$0<\lambda\le1$ and $r>0$,
\begin{align}
|A^\lambda(x',\mu')-A^0(x,\mu)|
\le C\Big(&\lambda|A^0(x',\mu')|+|x'-x|+\Wtwo(\mu',\mu)\nonumber\\
&+\one_{\{\operatorname{dist}(x,\Sigma)\le M_0\lambda+r\}}
+\one_{\{|x'-x|\ge r\}}\Big),
\label{eq:cumulative-point}
\end{align}
where $M_0=\max(\kappa_+,\kappa_-)+\eta$. For $\lambda=0$, the same
inequality holds after removing the first term.
\end{lemma}

\begin{proof}
First fix $\lambda>0$ and set $y=J_\lambda(x',\mu')$. By
Proposition~\ref{prop:Yosida},
\[
x'-y=\lambda A^\lambda(x',\mu'),\qquad
|x'-y|\le\lambda|A^0(x',\mu')|.
\]
There exists a vector $v$ in the subdifferential of the piecewise linear
part of $\phi_{\rm ALA}$ at $y$ such that
\[
A^\lambda(x',\mu')=\alpha_Ay+F_{\mu'}(y)+v,
\qquad |v|\le M_s:=\sqrt d\,\max(\kappa_+,\kappa_-).
\]
Likewise,
\[
A^0(x,\mu)=\alpha_Ax+F_\mu(x)+s(x,\mu),\qquad |s(x,\mu)|\le M_s.
\]
By \eqref{eq:Klip}, the regular parts satisfy
\begin{align*}
|\alpha_A(y-x)+F_{\mu'}(y)-F_\mu(x)|
&\le C\{|y-x|+\Wtwo(\mu',\mu)\}\\
&\le C\{\lambda|A^0(x',\mu')|+|x'-x|+\Wtwo(\mu',\mu)\}.
\end{align*}

It remains to compare the threshold vectors. If
$\operatorname{dist}(x,\Sigma)>M_0\lambda+r$ and $|x'-x|<r$, then every
coordinate $x'_k$ has the same sign as $x_k$ and satisfies
$|x'_k|>M_0\lambda$. The resolvent identity reads
\[
x'_k=(1+\lambda\alpha_A)y_k+
\lambda\{[F_{\mu'}(y)]_k+v_k\},
\qquad |[F_{\mu'}(y)]_k+v_k|\le M_0.
\]
Thus $x'_k>M_0\lambda$ implies $y_k>0$, while
$x'_k<-M_0\lambda$ implies $y_k<0$. Outside the two exceptional sets,
$x_k,x'_k,y_k$ therefore have the same sign and the threshold selections
coincide: $v=s(x,\mu)$. On their complement,
$|v-s(x,\mu)|\le2M_s$. Combining these estimates gives
\eqref{eq:cumulative-point}. For $\lambda=0$, take $y=x'$ directly and
apply the same sign argument.
\end{proof}

For the regularized system, set
\[
k_t^{i,\lambda}
:=\mathsf K_t^{i,N,\lambda}-\bar{\mathsf K}_t^i,
\quad
y_t^{i,\lambda}:=Y_t^{i,N,\lambda}-\bar Y_t^i,
\]
and retain the notation $\Delta Z^{i,j,\lambda}$ and
$\Delta U^{i,j,\lambda}$ from \eqref{eq:regerrorfull}. Define
\begin{equation}
Q_t^{i,\lambda}
:=\sum_{j=1}^N\left(
|\Delta Z_t^{i,j,\lambda}|_{a_t^{j,P}}^2+
\|\Delta U_t^{i,j,\lambda}\|_\Pi^2\right).
\label{eq:Qcumulative}
\end{equation}

\begin{lemma}[Backward stability through the accumulated process]
\label{lem:cumulative-stability}
Under \ref{HF}, \ref{HB}, \ref{HD}, \ref{HqJ} and \ref{Hprod}, there
exists $C<\infty$, independent of $P,N,\lambda$, such that
\begin{align}
\frac1N\sum_{i=1}^N\E^{P^{\otimes N}}\left[
\sup_{t\le T}|y_t^{i,\lambda}|^2+
\int_0^TQ_t^{i,\lambda}\,dt\right]
\le C\{\tau_d(N)+N^{-1/2}+R_K(P)\},
\label{eq:cumulative-stability}
\end{align}
where
\[
R_K(P):=\frac1N\sum_{i=1}^N
\E^{P^{\otimes N}}\sup_{t\le T}|k_t^{i,\lambda}|^2.
\]
\end{lemma}

\begin{proof}
Set
\[
v_t^i=y_t^{i,\lambda}+k_t^{i,\lambda},
\qquad
\zeta^i=g(X_T^{i,N},\mu_T^N)-g(\bar X_T^i,\mu_T^P),
\]
and let $G^i$ denote the difference of the two generators $f$, excluding
the source $A$. If
\[
M_t^i=\sum_j\int_0^t\Delta Z_s^{i,j,\lambda}\,dB_s^j+
\sum_j\int_0^t\int_E\Delta U_s^{i,j,\lambda}(e)\,
\widetilde N^j(ds,de),
\]
subtracting the two backward equations gives exactly
\begin{equation}
v_t^i=\zeta^i+k_T^{i,\lambda}
+\int_t^TG_s^i\,ds-(M_T^i-M_t^i).
\label{eq:v-cumulative}
\end{equation}
The law entering $f$ remains the law of $Y$, not that of $Y+\mathsf K$.

Let
\[
V(s)=\frac1N\sum_i\E|v_s^i|^2,\qquad
H(s)=\frac1N\sum_i\E|k_s^{i,\lambda}|^2,\qquad
\overline Q(s)=\frac1N\sum_i\E Q_s^{i,\lambda}.
\]
Introducing the empirical measure
$\bar\eta_s^N=N^{-1}\sum_i\delta_{\bar Y_s^i}$ and using the atomic
coupling,
\[
\E\Wtwo^2(\eta_s^{N,\lambda},\eta_s^P)
\le4V(s)+4H(s)+2\E\Wtwo^2(\bar\eta_s^N,\eta_s^P).
\]
The moments of order $q>4$ give
$\sup_{P,s}\E\Wtwo^2(\bar\eta_s^N,\eta_s^P)\le CN^{-1/2}$.
The Lipschitz property of $f$ then implies
\[
\frac1N\sum_i\E|G_s^i|^2
\le C_g\{F(s)+V(s)+H(s)+\overline Q(s)+CN^{-1/2}\},
\]
where
\[
F(s)=\frac1N\sum_i\E|X_{s-}^{i,N}-\bar X_{s-}^i|^2+
\E\Wtwo^2(\mu_{s-}^N,\mu_{s-}^P).
\]
By Young's inequality, with $\varepsilon=(2C_g)^{-1}$,
\[
2\frac1N\sum_i\E[v_s^iG_s^i]
\le(2C_g+\tfrac12)V(s)+\tfrac12\overline Q(s)
+C\{F(s)+H(s)+N^{-1/2}\}.
\]

Apply Itô's formula to $e^{\gamma t}|v_t^i|^2$ with
$\gamma=2C_g+2$. After localization, compensation and averaging over $i$,
the martingale terms have zero expectation. The justification follows from
BDG and Cauchy--Schwarz:
\[
\E\!\left[\sup_{t\le T}|v_t^i|[M^i]_T^{1/2}\right]
\le(\E\sup_t|v_t^i|^2)^{1/2}(\E[M^i]_T)^{1/2}<\infty.
\]
The quadratic jump remainder is kept positive before expectation; its
expectation is the Poisson energy. After absorption, we obtain
\[
\sup_tV(t)+\int_0^T\{V(s)+\overline Q(s)\}\,ds
\le C\{\tau_d(N)+N^{-1/2}+R_K(P)\}.
\]
Finally, \eqref{eq:v-cumulative}, Cauchy--Schwarz in time and Doob's
inequality give the same bound for
$N^{-1}\sum_i\E\sup_t|v_t^i|^2$. Since
$|y_t^{i,\lambda}|^2\le2|v_t^i|^2+2|k_t^{i,\lambda}|^2$,
\eqref{eq:cumulative-stability} follows.
\end{proof}

\begin{theorem}[Joint rate strengthened by quadratic occupation]
\label{thm:jointquadratic}
Assume \ref{HF}, \ref{HB}, \ref{HD}, \ref{HqJ}, \ref{Hprod} and
\ref{Hocc2}, and consider the ALA prototype
\eqref{eq:ALAphi}--\eqref{eq:ALAK}. For $0<\lambda\le1$, the error
$\mathcal D_{N,\lambda}^{\rm reg}$ defined in \eqref{eq:regerrorfull}
satisfies
\begin{equation}
\boxed{\;
\mathcal D_{N,\lambda}^{\rm reg}
\le C\left\{
\tau_d(N)^{\theta/(\theta+1)}
+\lambda^2+\lambda^{2\theta}
\right\},
\;}
\label{eq:jointquadratic}
\end{equation}
where $C$ is independent of $P$, $N$, and $\lambda$.

For the selected particle system, which is well posed by
Proposition~\ref{prop:directwell}, the estimate at $\lambda=0$ is obtained
directly from the selected-source version of the cumulative comparison:
the resolvent displacement term is absent, while the occupation and bad
coupling terms are unchanged. Consequently,
\begin{equation}
\mathcal D_{N,0}
\le C\tau_d(N)^{\theta/(\theta+1)}.
\label{eq:jointquadratic0}
\end{equation}
In particular, for every deterministic sequence $\lambda_N\downarrow0$,
\[
\mathcal D_{N,\lambda_N}^{\rm reg}\longrightarrow0
\qquad\text{as }N\to\infty.
\]
The convergence is strong in the quadratic norms entering
\eqref{eq:regerrorfull}, uniformly in the sense of the supremum over
$P\in\Pcal$; no simultaneous almost-sure convergence under all models is
asserted.
\end{theorem}

\begin{proof}
Fix $P\in\Pcal$, $i\in\{1,\ldots,N\}$ and write
$\E=\E^{P^{\otimes N}}$. Set
\[
D_i=\sup_{t\le T}|X_t^{i,N}-\bar X_t^i|,
\quad
O_i(R)=\int_0^T
\one_{\{\operatorname{dist}(\bar X_{t-}^i,\Sigma)\le R\}}dt,
\]
\[
B_i(r)=\int_0^T
\one_{\{|X_{t-}^{i,N}-\bar X_{t-}^i|\ge r\}}dt,
\]
and
\[
L_i=\int_0^T\!\left[
\lambda|A^0(X_{t-}^{i,N},\mu_{t-}^N)|
+|X_{t-}^{i,N}-\bar X_{t-}^i|
+\Wtwo(\mu_{t-}^N,\mu_{t-}^P)\right]dt.
\]
Lemma~\ref{lem:cumulative-point}, integrated from $0$ to $t$, gives
\begin{equation}
\sup_{t\le T}|k_t^{i,\lambda}|
\le C|\varrho|\{L_i+O_i(M_0\lambda+r)+B_i(r)\}.
\label{eq:kpath}
\end{equation}

By Cauchy--Schwarz in time, \eqref{eq:A0growth}, the uniform moments,
Theorem~\ref{thm:Fchaos}, and \eqref{eq:mulaw},
\begin{equation}
\sup_P\frac1N\sum_i\E L_i^2
\le C\{\lambda^2+\tau_d(N)\}.
\label{eq:L2cumulative}
\end{equation}
Each copy $\bar X^i$ has under $P^{\otimes N}$ the same law as the
limiting forward process under $P$. Assumption \ref{Hocc2}, extended to
radii larger than $1$ by increasing the constant, implies
\begin{equation}
\E O_i(M_0\lambda+r)^2
\le C(M_0\lambda+r)^{2\theta}
\le C\{\lambda^{2\theta}+r^{2\theta}\}.
\label{eq:O2cumulative}
\end{equation}
On the other hand,
$B_i(r)\le T\one_{\{D_i\ge r\}}$; by Markov's inequality and the forward
rate,
\begin{equation}
\sup_P\frac1N\sum_i\E B_i(r)^2
\le C\tau_d(N)r^{-2}.
\label{eq:B2cumulative}
\end{equation}
No independence between $D_i$ and $O_i$ is used.

Squaring \eqref{eq:kpath} and using
$(a+b+c)^2\le3(a^2+b^2+c^2)$ gives
\begin{equation}
\sup_{P\in\Pcal}R_K(P)
\le C\{\tau_d(N)+\lambda^2+\lambda^{2\theta}
+r^{2\theta}+\tau_d(N)r^{-2}\}.
\label{eq:RKquadratic}
\end{equation}
For $0<\tau_d(N)\le1$, choose
\[
r_N=\tau_d(N)^{1/(2\theta+2)}.
\]
Then
\[
r_N^{2\theta}
=\tau_d(N)r_N^{-2}
=\tau_d(N)^{\theta/(\theta+1)},
\qquad
\tau_d(N)\le\tau_d(N)^{\theta/(\theta+1)}.
\]
The remaining values of $N$ are absorbed into the constant using the
uniform moment bounds. Consequently,
\begin{equation}
\sup_PR_K(P)
\le C\left\{
\tau_d(N)^{\theta/(\theta+1)}
+\lambda^2+\lambda^{2\theta}\right\}.
\label{eq:RKoptimized}
\end{equation}

Lemma~\ref{lem:cumulative-stability}, the fact that
$N^{-1/2}\le C\tau_d(N)$, and the forward bound then give
\eqref{eq:jointquadratic}. We now spell out the selected case
$\lambda=0$.  In Lemma~\ref{lem:cumulative-point}, the Yosida displacement
$\lambda|A^0(X_{t-}^{i,N},\mu_{t-}^N)|$ disappears, so that the analogue of
$L_i$ contains only
\[
|X_{t-}^{i,N}-\bar X_{t-}^i|
+\Wtwo(\mu_{t-}^N,\mu_{t-}^P).
\]
Hence \eqref{eq:L2cumulative} becomes
\[
\sup_P\frac1N\sum_i\E L_i^2\le C\tau_d(N),
\]
whereas \eqref{eq:O2cumulative} is replaced by
$C r^{2\theta}$ and \eqref{eq:B2cumulative} is unchanged.  Therefore
\[
\sup_PR_K^{(0)}(P)
\le C\{\tau_d(N)+r^{2\theta}+\tau_d(N)r^{-2}\}.
\]
Choosing again $r=\tau_d(N)^{1/(2\theta+2)}$ gives
$\sup_PR_K^{(0)}(P)\le C\tau_d(N)^{\theta/(\theta+1)}$.  The selected
particle system is well posed by Proposition~\ref{prop:directwell}, and
Lemma~\ref{lem:cumulative-stability} with the selected sources then yields
\eqref{eq:jointquadratic0}.  No limiting argument
$\lambda\downarrow0$ is used in this step.
Finally, if $\lambda_N\to0$, each of the three terms on the right-hand side
of \eqref{eq:jointquadratic} tends to zero. Since every component of
\eqref{eq:regerrorfull} is nonnegative, convergence in each announced norm
follows directly.
\end{proof}

In particular, when $\theta=1$,
\[
\mathcal D_{N,\lambda}^{\rm reg}
\le C\{\tau_d(N)^{1/2}+\lambda^2\};
\]
for $d<4$ this is $C\{N^{-1/4}+\lambda^2\}$.  These are squared errors.

\begin{corollary}[Independent choice of the regularization parameter]
\label{cor:sharpjoint}
Under the assumptions of Theorem~\ref{thm:jointuniform}, every sequence
$\lambda_N\downarrow0$ satisfies
$\mathcal D_{N,\lambda_N}^{\rm reg}\to0$.
With $\gamma_*:=\min\{2,\theta\}$, any choice such that
$\lambda_N^{\gamma_*}\le C\tau_d(N)^{a_\theta}$ preserves the bound
$C\tau_d(N)^{a_\theta}$. If $\theta=1$, in particular,
\[
\mathcal D_{N,\lambda}^{\rm reg}\le C\{\tau_d(N)^{1/3}+\lambda\}.
\]
\end{corollary}
\begin{proof}
For $0<\lambda\le1$,
$\lambda^2+\lambda^\theta\le2\lambda^{\gamma_*}$.
The three assertions follow from \eqref{eq:jointuniform} and
$\tau_d(N)\to0$.
\end{proof}

The decomposition through a regularized limiting copy gives
\[
C\{(1+\lambda^{-2})\tau_d(N)+\lambda^{\gamma_*}\}.
\]
The factor $\lambda^{-2}$ comes from quadratic control of the source using
the Lipschitz constant of $A^\lambda$. For $\theta=1$,
\[
\lambda=\tau_d(N)^{1/3}
\quad\Longrightarrow\quad
\lambda^{-2}\tau_d(N)=\lambda=\tau_d(N)^{1/3}.
\]
The optimized exponent is therefore unchanged. The new bound removes the
constraint $\lambda_N^{-2}\tau_d(N)\to0$ for the continuous-time
approximation; it does not address time-discretization error. Its target
remains the selected limiting solution. For general operators, the rates
in Proposition~\ref{prop:Yrate} retain their own assumptions.

\paragraph{Necessity of the order-$\lambda$ term in the prototype.}
Take $d=1$, $\varrho=1$, $b=\sigma=\beta=0$ and $X_t=\xi$, where $\xi$
is $\mathcal F_0$-measurable with law
$\mu(dx)=\tfrac12\one_{[-1,1]}(x)dx$. Then
\[
\E\int_0^T\one_{\{|X_{t-}|\le r\}}dt=Tr,
\qquad 0<r\le1.
\]
The kernel \eqref{eq:ALAK} satisfies $F_\mu(0)=0$ by symmetry and
\[
F_\mu'(x)=\frac\eta2\int_{-1}^1
\frac{dy}{(1+(x-y)^2)^{3/2}}>0.
\]
Differentiation under the integral sign is permitted because the
derivative is bounded by $\eta/2$ on the integration interval. Hence
$F_\mu(x)\ge0$ for $x\ge0$. If
$0<\lambda\le\min\{1,2/\kappa_+\}$ and
$0<x\le\lambda\kappa_+/2$, then
\[
x/\lambda\in[-\kappa_-,\kappa_+]=A(0,\mu),\qquad
J_\lambda(x,\mu)=0,\qquad A^\lambda(x,\mu)=x/\lambda\le\kappa_+/2.
\]
On the other hand,
$A^0(x,\mu)=\alpha_Ax+F_\mu(x)+\kappa_+\ge\kappa_+$.
Consequently,
\[
\E|A^\lambda(\xi,\mu)-A^0(\xi,\mu)|^2
\ge\frac12\int_0^{\lambda\kappa_+/2}\frac{\kappa_+^2}{4}dx
=\frac{\kappa_+^3}{16}\lambda.
\]
For $f=g=0$ and $a\in\{0,\lambda\}$, the solutions are
\[
Y_t^a=(T-t)A^a(\xi,\mu),\qquad
\mathsf K_t^a=tA^a(\xi,\mu),\qquad Z^a=U^a=0.
\]
They satisfy the integral formulation because
$Y_t^a=\mathsf K_T^a-\mathsf K_t^a$. Therefore,
\[
\E\sup_{t\le T}|Y_t^\lambda-Y_t^0|^2
=T^2\E|A^\lambda(\xi,\mu)-A^0(\xi,\mu)|^2
\ge\frac{T^2\kappa_+^3}{16}\lambda.
\]
The term $\lambda$ cannot be replaced uniformly by $o(\lambda)$ in this
class. This example does not provide a particle lower bound and does not
impose effective jump activity.

\subsection{Prototype consequences under continuous nondegeneracy}

\begin{corollary}[Prototype rate under continuous nondegeneracy]
Under \ref{HF}, \ref{HB}, \ref{HD}, \ref{HqJ}, \ref{Hprod}, and
\ref{HellJ}, for the ALA prototype,
\[
\mathcal D_N\le C\tau_d(N)^{1/3},\qquad
\mathcal D_{N,\lambda}^{\rm reg}\le C\{\tau_d(N)^{1/3}+\lambda\},
\quad 0<\lambda\le1.
\]
The constants are independent of $P,N,\lambda$.
\end{corollary}
\begin{proof}
Theorem~\ref{thm:occupation-jumps} gives
$\sup_P\E^P\int_0^T\one_{\{\operatorname{dist}(X_{s-}^P,\Sigma)\le r\}}ds
\le Cr$ : it verifies \ref{HoccD} with $\theta=1$.
Theorem~\ref{thm:directchaos} therefore gives the first bound.
Theorem~\ref{thm:jointuniform} gives
$\mathcal D_{N,\lambda}^{\rm reg}\le
C\{\tau_d(N)^{1/3}+\lambda^2+\lambda\}$; since
$\lambda^2\le\lambda$ on $(0,1]$, the second bound follows.
This corollary uses the structure of the prototype; the preceding facewise
rates remain intended for general operators.
\end{proof}

\paragraph{Brownian--Poisson example with law dependence.}
Take $d=m=r$, $D=\R^d$, $\Pi=\ell\vartheta$, where $\ell>0$ and
$\vartheta$ is a probability measure supported on $0<|e|\le1$. Choose
\[
b(t,x,\mu,\alpha)=-x+\int y\mu(dy)+\alpha,
\qquad \sigma=I_d,\qquad \beta(t,x,\mu,\alpha,e)=e.
\]
For fixed vectors $v,w\in\R^d$, a scalar $c$, and a function
$h\in L^2_\Pi$, set
\begin{align*}
g(x,\mu)&=\langle v,x\rangle,\\
f(t,x,\mu,y,z,u,\eta,\alpha)
&=cy+\langle v,x\rangle+zw+\int_Eu(e)h(e)\Pi(de)+\int a\eta(da).
\end{align*}
All law dependencies are Lipschitz in $W_2$, since
$|\int y\mu(dy)-\int y\nu(dy)|\le W_2(\mu,\nu)$.
Cauchy--Schwarz gives
$|\int(u-u')h\,d\Pi|\le\|h\|_\Pi\|u-u'\|_\Pi$.
Finally, $\int|\beta|^q d\Pi\le\ell$ and
$(\sigma a^P\sigma^\top)_{kk}\ge\lambda_{\min}(\underline a)$.
With data and controls satisfying \ref{HqJ} and \ref{Hprod}, the ALA
prototype therefore satisfies the assumptions of the preceding corollary.
This example includes a nonzero jump term and law dependence in both
equations.

\section{Occupation generated by a purely discontinuous component}
\label{sec:purejump}

Condition \ref{HellJ} uses the continuous diffusion. We now verify
\ref{HoccD} through a different mechanism, for an affine class with additive
noise. The result below includes $\sigma=0$, and its constants do not
deteriorate when an additive diffusion tends to zero.

\begin{hypothesis}[Truncated stable small jumps and affine dynamics]
\label{HtruncJ}
Take $m=r=d$, $D=\R^d$, $\upsilon\in(1,2)$, and
$c_1,\ldots,c_d>0$. With $(\mathbf e_j)_{j\le d}$ denoting the canonical
basis, the common Lévy measure is defined, for every nonnegative function
$v$, by
\begin{equation}
\int_Ev(e)\Pi(de)=\sum_{j=1}^dc_j
\int_{0<|z|\le1}v(z\mathbf e_j)|z|^{-1-\upsilon}dz.
\label{eq:truncatedPi}
\end{equation}
Under each $P$, the data $\xi$, $B$, and $N$ are mutually independent;
$N$ is a Poisson random measure and $B$ is a continuous Gaussian martingale
with independent increments and deterministic covariance
$\int_0^ta_s^Pds$, satisfying \eqref{eq:qv}. The filtration is the augmented
filtration generated by these data, and the representation properties of the
framework and \ref{Hprod} are maintained. The common initial law has a
moment of order $q>4$.

Fix $\kappa\ge0$, a deterministic map
$H:\mathcal P_2(\R^d)\to\R^d$ such that
$|H(\mu)-H(\nu)|\le L_H\Wtwo(\mu,\nu)$, and a deterministic control
$h\in L^q([0,T];\R^d)$. For $\varepsilon\in[0,1]$, consider
\begin{align}
X_t^{\varepsilon,P}={}&\xi+
\int_0^t[-\kappa X_{s-}^{\varepsilon,P}+H(\mu_s^{\varepsilon,P})+h_s]ds
+\varepsilon B_t+L_t^P,\label{eq:affinelevy}\\
L_t^P={}&\int_0^t\int_Ee\widetilde N^P(ds,de),
\qquad \mu_t^{\varepsilon,P}=\Lcal^P(X_t^{\varepsilon,P}).
\nonumber
\end{align}
The symbol $\varepsilon$ denotes only the Brownian amplitude, whereas
$\upsilon$ denotes the small-jump index. The backward coefficients satisfy
\ref{HB}, and the operator is the ALA prototype.
\end{hypothesis}
The compensator remains fixed throughout the family. Variation in
$\varepsilon$ is a variation of coefficients with common constants, not a
variation of the compensator. A family of deterministic covariances may be
non-dominated as in the initial framework; this non-domination plays no role
in the density proof.

\begin{theorem}[Marginal densities and occupation without continuous ellipticity]
\label{thm:levyoccupation}
Under \ref{HtruncJ}, each coordinate of $X_t^{\varepsilon,P}$, $t>0$,
has a density $p_{t,j}^{\varepsilon,P}$ satisfying
\begin{align}
\sup_{\varepsilon\in[0,1],P,j}\|p_{t,j}^{\varepsilon,P}\|_\infty
&\le C_T(1+t^{-1/\upsilon}),\qquad 0<t\le T,
\label{eq:levydensity}\\
\sup_{\varepsilon\in[0,1],P}\E^P\int_0^T
\one_{\{\operatorname{dist}(X_{t-}^{\varepsilon,P},\Sigma)\le r\}}dt
&\le C_T r,\qquad 0<r\le1.
\label{eq:levyocc}
\end{align}
The constants depend on $T,d,\kappa,\upsilon$, and the $c_j$, but not on
the Brownian amplitude, the covariance $a^P$, the deterministic shift, or
the initial law.
\end{theorem}
\begin{proof}
\emph{Step 1: construction of the integrals and moment bounds.}
For every $p>\upsilon$, the measure \eqref{eq:truncatedPi} satisfies
\[
\int_E|e|^p\Pi(de)=\frac{2\sum_jc_j}{p-\upsilon}<\infty,
\qquad \Pi(E)=\infty.
\]
In particular, the noise is square-integrable, has infinite activity, and
has jump moments of order $q$. The integral defining $L$ is the limit in
$\mathbb S^2$ of the integrals over $|e|>a$, as $a\downarrow0$, since
Doob's inequality and the isometry give
\[
\E\sup_{s\le T}|L_s-L_s^{(a)}|^2
\le4T\int_{|e|\le a}|e|^2\Pi(de)
=\frac{8T\sum_jc_j}{2-\upsilon}a^{2-\upsilon}.
\]
The growth of $H$ follows from
$|H(\mu)|\le|H(\delta_0)|+L_HM_2(\mu)$.
Hypotheses \ref{HF} and \ref{HqJ} are therefore satisfied by
$b(t,x,\mu,\alpha)=-\kappa x+H(\mu)+\alpha$, $\sigma=\varepsilon I_d$, $\beta(e)=e$,
with the control $\alpha=h$ and common constants for
$0\le\varepsilon\le1$. In particular,
$b(t,0,\delta_0,0)=H(\delta_0)$ is bounded independently of $t$; the control $h$ satisfies the integrated
condition in \ref{HqJ}. Theorem~\ref{thm:Fwell} constructs
$X^{\varepsilon,P}$ and gives
\[
\sup_{\varepsilon,P}\E^P\sup_{t\le T}|X_t^{\varepsilon,P}|^q<\infty.
\]

\emph{Step 2: explicit integral representation.}
Apply the product rule to $e^{\kappa t}X_t^{\varepsilon,P}$. Since the
factor is deterministic, continuous, and of finite variation, it creates no
covariation term and no additional jump term. We obtain
\begin{align}
X_t^{\varepsilon,P}
={}&e^{-\kappa t}\xi+m_t^{\varepsilon,P}
+\varepsilon\int_0^te^{-\kappa(t-s)}dB_s
+V_t^P,\label{eq:affineexplicit}\\
m_t^{\varepsilon,P}
={}&\int_0^te^{-\kappa(t-s)}[H(\mu_s^{\varepsilon,P})+h_s]ds,
\qquad
V_t^P=\int_0^t\int_Ee^{-\kappa(t-s)}e\widetilde N^P(ds,de).
\nonumber
\end{align}
The term $m_t^{\varepsilon,P}$ is deterministic: the unconditional law
$\mu_s^{\varepsilon,P}$ is deterministic for each fixed model. The three
random terms in \eqref{eq:affineexplicit} are independent. This independence
is an assumption of \ref{HtruncJ}; it is not deduced merely from
orthogonality of the martingales.

\emph{Step 3: Fourier damping by the small jumps.}
For the $j$th coordinate of $L$, the real characteristic exponent is
\[
\psi_j(u)=2c_j\int_0^1(1-\cos(uz))z^{-1-\upsilon}dz.
\]
This formula first follows from that of the truncated compound Poisson
process. Passing to the limit $a\downarrow0$ is justified by
$1-\cos(uz)\le u^2z^2/2$ and by $L^2$ convergence of the truncated noise.
With
\[
b_j:=2c_j\int_0^1(1-\cos v)v^{-1-\upsilon}dv>0,
\]
the integral is finite because $1-\cos v\le v^2/2$ and $\upsilon<2$.
For $|u|\ge1$, the change of variable $v=|u|z$ gives
\[
\psi_j(u)=2c_j|u|^\upsilon
\int_0^{|u|}(1-\cos v)v^{-1-\upsilon}dv
\ge b_j|u|^\upsilon.
\]
The characteristic function of $V_t^{P,j}$ is
\[
\widehat p_{V,t,j}(u)=
\exp\left(-\int_0^t\psi_j(e^{-\kappa(t-s)}u)ds\right).
\]
To justify this identity, write it first for the truncated Poisson integral
and then pass to the limit using the domination
$|u|^2e^{-2\kappa(t-s)}z^2/2$ under $ds\otimes\Pi$ and the $L^2$ convergence of the integrals. If
$R_T=e^{\kappa T}$ and $|u|\ge R_T$, then
\[
\int_0^t\psi_j(e^{-\kappa(t-s)}u)ds
\ge b_je^{-\upsilon\kappa T}t|u|^\upsilon.
\]
With $b_*=(\min_jb_j)e^{-\upsilon\kappa T}>0$, it follows that
\begin{align}
\int_\R|\widehat p_{V,t,j}(u)|du
&\le2R_T+2\int_0^\infty e^{-b_*tu^\upsilon}du\nonumber\\
&=2R_T+2(b_*t)^{-1/\upsilon}
\int_0^\infty e^{-v^\upsilon}dv.
\label{eq:Fourierbound}
\end{align}
The final integral is finite. Fourier inversion yields a bounded continuous
density $p_{V,t,j}$, with
$\|p_{V,t,j}\|_\infty\le(2\pi)^{-1}
\|\widehat p_{V,t,j}\|_{L^1}$.
This inversion can be justified by first convolving the law of $V$ with a
Gaussian distribution of variance $a>0$: its Fourier integrand is
$\widehat p_{V,t,j}(u)e^{-au^2/2}$, dominated by the integrable bound in \eqref{eq:Fourierbound}. The densities
converge uniformly as $a\downarrow0$. For every continuous compactly
supported function, uniform convergence of the densities together with weak
convergence of the convolutions identifies this limit as the density of $V$.

\emph{Step 4: convolution and occupation.}
The law of the $j$th coordinate in \eqref{eq:affineexplicit} is the
convolution of $p_{V,t,j}$ with a probability measure, followed by a
translation. For every bounded density $p$ and probability measure $\nu$,
$\|p*\nu\|_\infty\le\|p\|_\infty$, by integrating $|p(x-y)|\le\|p\|_\infty$. Thus
\eqref{eq:Fourierbound} proves \eqref{eq:levydensity}, including when
$\varepsilon=0$.

At every deterministic time $t$, $X_t=X_{t-}$ almost surely, by the
mean-square continuity established in the initial framework. The union bound
and \eqref{eq:levydensity} give
\[
P(\operatorname{dist}(X_{t-}^{\varepsilon,P},\Sigma)\le r)
\le\sum_j\int_{-r}^rp_{t,j}^{\varepsilon,P}(x)dx
\le2dC_T r(1+t^{-1/\upsilon}).
\]
Tonelli's theorem applies to this nonnegative integrand. Since
$\upsilon>1$,
\[
\int_0^T(1+t^{-1/\upsilon})dt
=T+\frac{\upsilon}{\upsilon-1}T^{1-1/\upsilon}<\infty.
\]
This proves \eqref{eq:levyocc}. The initial law need not have a density,
even if it is supported on an interface.
\end{proof}

\begin{proposition}[Second-moment occupation in the affine class]
\label{prop:affineocc2}
Under \ref{HtruncJ}, there exists $C<\infty$, independent of
$\varepsilon\in[0,1]$, $P\in\Pcal$, $s\in[0,T]$, and $r\in(0,1]$, such
that, for every $P\in\Pcal$ and every deterministic $s\in[0,T]$,
\[
\E^P\left[\int_s^T
\one_{\{\operatorname{dist}(X_{t-}^{\varepsilon,P},\Sigma)\le r\}}dt
\,\middle|\,\mathcal F_s^P\right]\le Cr
\quad P\text{-a.s.}
\]
Consequently, \ref{Hocc2} holds with $\theta=1$, uniformly in
$\varepsilon$ and $P$.
\end{proposition}
\begin{proof}
Fix $P\in\Pcal$, $\varepsilon\in[0,1]$, and a deterministic
$s\in[0,T]$. Solving the affine equation from time $s$ gives, for
$t\ge s$,
\[
X_t^{\varepsilon,P}
=e^{-\kappa(t-s)}X_s^{\varepsilon,P}
+m_{s,t}^{\varepsilon,P}
+\varepsilon\int_s^t e^{-\kappa(t-u)}\,dB_u
+\int_s^t\int_E e^{-\kappa(t-u)}e\,\widetilde N^P(du,de),
\]
where
\[
m_{s,t}^{\varepsilon,P}
=\int_s^t e^{-\kappa(t-u)}
\{H(\mu_u^{\varepsilon,P})+h_u\}\,du
\]
is deterministic under the fixed model $P$.  By \ref{HtruncJ}, the
increments of $B$ and $N$ on $(s,t]$ are independent of
$\mathcal F_s^P$, and the two future noise terms are independent of each
other.  Hence, conditionally on $\mathcal F_s^P$, the first two terms above
form only a translation, while the conditional characteristic function of
the future stochastic increment factors into its Gaussian and jump parts.

For the $j$th coordinate, discard the Gaussian factor, whose modulus is at
most one.  The jump factor has modulus
\[
\exp\{-\Psi_{s,t,j}^P(u)\},
\]
where, by symmetry of the truncated stable measure,
\[
\Psi_{s,t,j}^P(u)
=2c_j\int_s^t\int_0^1
\{1-\cos(u e^{-\kappa(t-v)}z)\}z^{-1-\upsilon}\,dz\,dv .
\]
Since $e^{-\kappa T}\le e^{-\kappa(t-v)}\le1$, the same change of variables
as in Theorem~\ref{thm:levyoccupation} yields constants $c_0,C_0>0$,
depending only on the common structural parameters, such that
\[
\Psi_{s,t,j}^P(u)\ge c_0(t-s)|u|^\upsilon
\qquad\text{for }|u|\ge C_0 .
\]
Therefore the conditional characteristic function is integrable and
Fourier inversion gives a conditional density
$p_{s,t,j}^{\varepsilon,P}(\cdot,\omega)$ satisfying
\[
\|p_{s,t,j}^{\varepsilon,P}(\cdot,\omega)\|_\infty
\le C\{1+(t-s)^{-1/\upsilon}\},
\qquad t>s,
\]
for $P$-almost every $\omega$, with $C$ independent of
$P,\varepsilon,s,t$, and $j$.  The random
$\mathcal F_s^P$-measurable translation does not affect this bound.

Since
$\{\operatorname{dist}(x,\Sigma)\le r\}\subset
\bigcup_{j=1}^d\{|x_j|\le r\}$, the union bound gives
\[
P\!\left(
\operatorname{dist}(X_{t-}^{\varepsilon,P},\Sigma)\le r
\,\middle|\,\mathcal F_s^P\right)
\le Cr\{1+(t-s)^{-1/\upsilon}\}
\quad P\text{-a.s.}
\]
for $dt$-almost every $t>s$; replacing $X_t$ by $X_{t-}$ does not change
the time integral. Conditional Tonelli and $\upsilon>1$ now yield
\[
\E^P\left[\int_s^T
\one_{\{\operatorname{dist}(X_{t-}^{\varepsilon,P},\Sigma)\le r\}}dt
\,\middle|\,\mathcal F_s^P\right]
\le Cr\int_s^T\{1+(t-s)^{-1/\upsilon}\}\,dt
\le C_T r .
\]
The constant is uniform in the displayed parameters.  Applying
Lemma~\ref{lem:conditional-occ2} to the occupation indicator gives
\[
\E^P\left(
\int_0^T
\one_{\{\operatorname{dist}(X_{t-}^{\varepsilon,P},\Sigma)\le r\}}dt
\right)^2
\le C r^2,
\]
which is \ref{Hocc2} with $\theta=1$.
\end{proof}

\begin{corollary}[Strengthened propagation and regularization in the affine class]
\label{cor:purejumpchaos2}
Under \ref{HtruncJ}, \ref{HB}, \ref{HD}, \ref{Hprod}, and the common
moment conditions with $q>4$,
\[
\sup_{\varepsilon\in[0,1]}\mathcal D_N^\varepsilon
\le C\tau_d(N)^{1/2},
\qquad
\sup_{\varepsilon\in[0,1]}\mathcal D_{N,\lambda}^{\rm reg,\varepsilon}
\le C\{\tau_d(N)^{1/2}+\lambda^2\}.
\]
Thus the choice $\lambda_N=\tau_d(N)^{1/4}$ balances the two terms. In
dimension $d<4$ the bound is $CN^{-1/4}$; in dimension $d=4$ it is
$C\{N^{-1/2}\log(1+N)\}^{1/2}$; and in dimension $d>4$ it is $CN^{-1/d}$.
\end{corollary}
\begin{proof}
Proposition~\ref{prop:affineocc2} verifies \ref{Hocc2} with $\theta=1$.
It suffices to apply Theorem~\ref{thm:jointquadratic} and then
\eqref{eq:tau}. The case $\lambda=0$ is included through
\eqref{eq:jointquadratic0}.
\end{proof}

\begin{corollary}[Direct propagation and regularization without diffusion]
\label{cor:purejumpchaos}
Under \ref{HtruncJ}, \ref{HB}, \ref{HD}, \ref{Hprod}, and the common
moment conditions with $q>4$, the argument based only on the first moment of
occupation gives, for the selected particle systems associated with
\eqref{eq:affinelevy},
\begin{equation}
\sup_{\varepsilon\in[0,1]}\mathcal D_N^\varepsilon
\le C\tau_d(N)^{1/3}.
\label{eq:purejumpchaos}
\end{equation}
For the regularized systems, the choice $\lambda_N=\tau_d(N)^{1/3}$
gives the same quadratic upper bound uniformly in $\varepsilon$. In
dimension $d<4$, it is $CN^{-1/6}$; in dimension $d=4$,
$C(N^{-1/2}\log(1+N))^{1/3}$; and in dimension $d>4$, $CN^{-2/(3d)}$.
\end{corollary}
\begin{proof}
Step 1 of Theorem~\ref{thm:levyoccupation} verifies the forward and moment
assumptions with constants independent of $\varepsilon$. Formula
\eqref{eq:levyocc} verifies \ref{HoccD} with $\theta=1$ and a constant also
independent of $\varepsilon$. Apply Theorem~\ref{thm:directchaos} and then
Corollary~\ref{cor:sharpjoint} with $\gamma_*=1$. The dimension-dependent
expressions follow from \eqref{eq:tau}. These bounds are bounds on the
squared error, not on its square root. Proposition~\ref{prop:affineocc2}
allows the strengthening in Corollary~\ref{cor:purejumpchaos2}; the present
corollary is retained to isolate what follows from the first moment alone.
\end{proof}

\begin{proposition}[Absence of a Brownian component when $\varepsilon=0$]
\label{prop:Zpurezero}
Under the assumptions of Corollary~\ref{cor:purejumpchaos}, when
$\varepsilon=0$, the selected limiting and particle solutions can be
constructed in the filtrations generated only by the initial data and the
Poisson random measures. In the full filtration, $Z=0$ and
$Z^{i,j,N,0}=0$ in their quadratic norms.
\end{proposition}
\begin{proof}
The forward process, its copies, and its particles are measurable in these
reduced filtrations by their strong construction, since $h$ is deterministic.
In the initial--Poisson filtration, apply the fixed-point construction to the
generator $f$ with its $z$ argument fixed at zero. The martingale
representation then uses only the Poisson random measure (all Poisson random
measures for the product system). The contraction estimates remain valid
after removing the $Z$ terms; the source and terminal condition are
measurable in this filtration. This yields a quadratic solution
$(Y,U,\mathsf K)$. Independence of the Brownian noise implies that these
martingales remain martingales in the full filtration: for an integrable
random variable measurable with respect to the terminal reduced filtration,
conditioning additionally on the independent Brownian past does not change
the conditional expectation. The resulting solution, completed by $Z=0$,
therefore solves the BSDE in the full framework
The uniqueness result in that framework forces coincidence with the
previously constructed solution. The argument applies to each row of the
product representation and gives $Z^{i,j,N,0}=0$ for all $i,j$.
\end{proof}

\begin{corollary}[Stability as the additive diffusion vanishes]
\label{cor:vanishingdiffusion}
Under the assumptions of Corollary~\ref{cor:purejumpchaos}, couple the
solutions with amplitudes $\varepsilon$ and zero using the same data
$(\xi,B,N,h)$. Writing $\Delta_\varepsilon V=V^{\varepsilon,P}-V^{0,P}$,
we have
\begin{align}
\sup_P\E^P\Bigl[&\sup_t|\Delta_\varepsilon X_t|^2+
\sup_t|\Delta_\varepsilon Y_t|^2+
\sup_t|\Delta_\varepsilon\mathsf K_t|^2\nonumber\\
&+\int_0^T|Z_t^{\varepsilon,P}|_{a_t^P}^2dt+
\int_0^T\|\Delta_\varepsilon U_t\|_\Pi^2dt\Bigr]
\le C\varepsilon,\qquad 0\le\varepsilon\le1.
\label{eq:vanishingdiffusion}
\end{align}
Moreover, the squared error of the selected particle system with amplitude
$\varepsilon$ relative to copies of the purely discontinuous limiting
solution is bounded by $C\{\tau_d(N)^{1/2}+\varepsilon\}$, in the full
norms of \eqref{eq:directerror} with this new target.
\end{corollary}
\begin{proof}
\emph{Step 1: comparison of the forward processes.}
The Poisson integrals cancel in the coupled difference:
\[
\Delta_\varepsilon X_t=
-\kappa\int_0^t\Delta_\varepsilon X_sds+
\int_0^t[H(\mu_s^{\varepsilon,P})-H(\mu_s^{0,P})]ds+
\varepsilon B_t.
\]
Set $d_P(t)=\E^P\sup_{s\le t}|\Delta_\varepsilon X_s|^2$.
The coupling and the Lipschitz property of $H$ give
$|H(\mu_s^{\varepsilon,P})-H(\mu_s^{0,P})|^2\le L_H^2d_P(s)$.
Cauchy--Schwarz in the time integrals and Doob's inequality give
\[
d_P(t)\le3T(\kappa^2+L_H^2)\int_0^td_P(s)ds+
12\varepsilon^2\E^P|B_t|^2
\le C\int_0^td_P(s)ds+C\varepsilon^2.
\]
The bound $\E|B_t|^2\le T\operatorname{Tr}(\overline a)$ is uniform in
$P$. Writing the last estimate in the form
$d_P(t)\le C_1\int_0^td_P(s)ds+C_2\varepsilon^2$,
\eqref{eq:GronwallExplicit} gives $d_P(T)\le C_2e^{C_1T}\varepsilon^2$. The coupling
$W_2^2(\mu_t^{\varepsilon,P},\mu_t^{0,P})\le d_P(t)$ then yields
\[
\sup_P\left[\E^P\sup_t|\Delta_\varepsilon X_t|^2+
\sup_t\Wtwo^2(\mu_t^{\varepsilon,P},\mu_t^{0,P})\right]
\le C\varepsilon^2.
\]

\emph{Step 2: cumulative control of the source.}
Proposition~\ref{prop:affineocc2}, applied to $X^{0,P}$, provides the
second occupation moment with $\theta=1$. Apply
Lemma~\ref{lem:cumulative-point} to
$(X^{\varepsilon,P},\mu^{\varepsilon,P};X^{0,P},\mu^{0,P})$ with
$\lambda=0$. If
$D_\varepsilon=\sup_t|X_t^{\varepsilon,P}-X_t^{0,P}|$, then, for every
$r>0$,
\[
\sup_t|\mathsf K_t^{\varepsilon,P}-\mathsf K_t^{0,P}|
\le C\left\{L_\varepsilon+O_0(r)+T\one_{\{D_\varepsilon\ge r\}}\right\},
\]
where $L_\varepsilon$ contains the integrals of the state and law errors and
satisfies $\E L_\varepsilon^2\le C\varepsilon^2$, while
$\E O_0(r)^2\le Cr^2$. Markov's inequality and Step 1 give
\[
\P(D_\varepsilon\ge r)\le C\varepsilon^2r^{-2}.
\]
Consequently,
\[
\sup_P\E^P\sup_t|\mathsf K_t^{\varepsilon,P}-\mathsf K_t^{0,P}|^2
\le C\{\varepsilon^2+r^2+\varepsilon^2r^{-2}\}.
\]
The choice $r=\varepsilon^{1/2}$ gives the bound $C\varepsilon$.

\emph{Step 3: transfer to the backward and particle systems.}
By \ref{HB}, the squared terminal difference is at most
$2L_B^2(|\Delta_\varepsilon X_T|^2+
\Wtwo^2(\mu_T^{\varepsilon,P},\mu_T^{0,P}))$.
Use the cumulative transformation from the proof of
Lemma~\ref{lem:cumulative-stability}; the finite-variation term is controlled
in $\mathbb S^2$ by Step 2, while the forward and terminal errors are of
order $\varepsilon^2$. This yields a total backward error bounded by
$C\varepsilon$. Proposition~\ref{prop:Zpurezero} identifies $Z^{0,P}=0$,
which proves \eqref{eq:vanishingdiffusion}.

On the product space, insert independent copies of the limiting solution
with amplitude $\varepsilon$ between the particles and the purely
discontinuous copies. The quadratic triangle inequality, applied to each
supremum and each energy integral, gives twice the strengthened error from
Corollary~\ref{cor:purejumpchaos2} plus twice the error in
\eqref{eq:vanishingdiffusion}. For the matrices, the difference between the
two limiting solutions appears only on the diagonal; its normalized sum is
exactly the error of one copy. This proves the last assertion. The resulting
convergences are strong in the displayed norms; no almost-sure convergence
uniform over $\Pcal$ is deduced.
\end{proof}

\subsection{Predictable jump amplitudes and nonlinear dynamics}
\label{sec:variablejumps}

The convolution used above requires independence of the random terms in the
affine representation. The following argument relies on the positive
remainder in Itô's formula. It allows variable predictable amplitudes, with
a weaker occupation exponent.

\begin{lemma}[Scalar occupation through coercivity of the small jumps]
\label{lem:adaptedoccupation}
On a filtered probability space satisfying the usual conditions, let
\[
V_t=V_0+\int_0^t b_sds+M_t^c+
\int_0^t\int_{0<|z|\le1}\gamma_s z\widetilde N(ds,dz),
\]
where $M^c$ is a square-integrable continuous martingale starting from zero,
and $N$ is a Poisson random measure with compensator $ds\,\nu(dz)$, where
\[
\nu(dz)=\ell |z|^{-1-\upsilon}\one_{\{0<|z|\le1\}}dz,
\qquad \ell>0,\quad 1<\upsilon<2.
\]
Assume $V_0\in L^2$ and that $b,\gamma$ are predictable, with
\[
0<\underline\gamma\le|\gamma_s|\le\overline\gamma,
\qquad \E\int_0^T|b_s|^2ds+\E\langle M^c\rangle_T<\infty.
\]
For $G=\max\{1,\overline\gamma\}$, set
\[
c_*:=\frac{\ell\underline\gamma^2}
{5^{3/2}(2-\upsilon)G^{2-\upsilon}},\qquad
C_0:=\E|V_T-V_0|+\E\int_0^T|b_s|ds.
\]
Then $C_0<\infty$ and
\begin{equation}
\E\int_0^T\one_{\{|V_{s-}-a|\le r\}}ds
\le c_*^{-1}C_0r^{\upsilon-1},
\qquad a\in\R,\quad 0<r\le1.
\label{eq:adaptedoccupation}
\end{equation}
No independence between $b,\gamma$ and the noises is required. The
constant is uniform over a family of models when the bounds above are
common.
\end{lemma}

\begin{proof}
\emph{Step 1: integrability of the process.}
Since
\[
\int z^2\nu(dz)=\frac{2\ell}{2-\upsilon},
\]
the isometry and Doob's inequality give
\[
\E\sup_{t\le T}\left|\int_0^t\int\gamma_s z\widetilde N(ds,dz)\right|^2
\le\frac{8\ell\overline\gamma^2T}{2-\upsilon}.
\]
Cauchy--Schwarz and Doob's inequality applied to the other two integrals
show that $V\in\mathbb S^2$ and that $C_0$ is finite. More explicitly,
\[
C_0\le2\sqrt{T}\left(\E\int_0^T|b_s|^2ds\right)^{1/2}
+(\E\langle M^c\rangle_T)^{1/2}
+\overline\gamma\sqrt{\frac{2\ell T}{2-\upsilon}}.
\]
This bound depends on neither $a$ nor $r$.

\emph{Step 2: local coercivity of the jump remainder.}
Set $\psi_r(x)=\sqrt{x^2+r^2}$. Then
\[
|\psi_r'|\le1,\qquad
0\le\psi_r''(x)=\frac{r^2}{(x^2+r^2)^{3/2}}\le r^{-1}.
\]
The integral form of Taylor's formula gives
\[
R_r(x,v):=\psi_r(x+v)-\psi_r(x)-\psi_r'(x)v
=v^2\int_0^1(1-u)\psi_r''(x+uv)du.
\]
Thus $0\le R_r(x,v)\le v^2/(2r)$ everywhere. If $|x|\le r$ and
$|v|\le r$, then $|x+uv|\le2r$ and
\[
R_r(x,v)\ge\frac{v^2}{2\,5^{3/2}r}.
\]
For $|z|\le r/G$, we have $|\gamma_s z|\le r$ and $r/G\le1$.
Since the remainder is also nonnegative for the other marks,
\begin{align*}
\int R_r(x,\gamma_s z)\nu(dz)
&\ge\one_{\{|x|\le r\}}\frac{\underline\gamma^2}{2\,5^{3/2}r}
\int_{|z|\le r/G}z^2\nu(dz)\\
&=c_*r^{1-\upsilon}\one_{\{|x|\le r\}}.
\end{align*}
This inequality holds pathwise for every admissible amplitude.

\emph{Step 3: Itô formula and justification of taking expectations.}
Fix $a,r$. The integral formula is
\begin{align*}
\psi_r(V_T-a)-\psi_r(V_0-a)
={}&\int_0^T\psi_r'(V_{s-}-a)b_sds
+\frac12\int_0^T\psi_r''(V_{s-}-a)d\langle M^c\rangle_s\\
&+\int_0^T\int R_r(V_{s-}-a,\gamma_s z)\nu(dz)ds+Q_T,
\end{align*}
where
\begin{align*}
Q_t={}&\int_0^t\psi_r'(V_{s-}-a)dM_s^c\\
&+\int_0^t\int
[\psi_r(V_{s-}-a+\gamma_s z)-\psi_r(V_{s-}-a)]\widetilde N(ds,dz).
\end{align*}
Both integrals in $Q$ are square-integrable: their isometries are bounded,
respectively, by $\E\langle M^c\rangle_T$ and
$2\ell\overline\gamma^2T/(2-\upsilon)$, since $\psi_r$ is
$1$-Lipschitz. Hence $\E Q_T=0$.
The drift is dominated by $\int|b_s|ds$; the two second-order terms are
dominated by
\[
\frac{\langle M^c\rangle_T}{2r},\qquad
\frac1{2r}\int_0^T\int|\gamma_s z|^2\nu(dz)ds,
\]
which are integrable for fixed $r$. We now justify the limiting passage.
Set
\[
A_t=\sup_{u\le t}|V_u|^2+\int_0^t|b_s|^2ds+\langle M^c\rangle_t,
\qquad \tau_n=\inf\{t\in[0,T]:A_t\ge n\}\wedge T,
\]
with $\inf\varnothing=+\infty$. Since $A_T<\infty$ almost surely,
$\tau_n=T$ for every $n>A_T$. Apply Itô's formula at $T\wedge\tau_n$.
If $Q^c,Q^d$ are the two parts of $Q$, Doob's inequality and the isometries
give
\begin{align*}
\E\sup_{t\le T}|Q_t^c-Q_{t\wedge\tau_n}^c|^2
&\le4\E\int_0^T\one_{\{s>\tau_n\}}
|\psi_r'(V_{s-}-a)|^2d\langle M^c\rangle_s\longrightarrow0,\\
\E\sup_{t\le T}|Q_t^d-Q_{t\wedge\tau_n}^d|^2
&\le4\E\int_0^T\int\one_{\{s>\tau_n\}}
|\psi_r(V_{s-}-a+\gamma_sz)-\psi_r(V_{s-}-a)|^2\nu(dz)ds
\longrightarrow0.
\end{align*}
Dominated convergence applies, respectively, with the integrable majorants
$\langle M^c\rangle_T$ and $\overline\gamma^2T\int z^2\nu(dz)$. The
drift converges in $L^1$ by domination with $\int_0^T|b_s|ds$; the two
second-order terms converge in $L^1$ by the preceding majorants. Finally,
\[
\psi_r(V_{T\wedge\tau_n}-a)\to\psi_r(V_T-a)\quad\text{in }L^1,
\]
because the convergence is almost sure and
$|\psi_r(V_{T\wedge\tau_n}-a)|\le\sup_t|V_t|+|a|+r\in L^1$.
All limits are taken for fixed $r>0$; no bound uniform in $r$ is required
at this step.

\emph{Step 4: estimate and uniformity.}
The continuous second-order term is nonnegative. After taking expectations
and using Step 2, we obtain
\begin{align*}
c_*r^{1-\upsilon}\E\int_0^T\one_{\{|V_{s-}-a|\le r\}}ds
&\le\E[\psi_r(V_T-a)-\psi_r(V_0-a)]
+\E\int_0^T|b_s|ds\\
&\le\E|V_T-V_0|+\E\int_0^T|b_s|ds=C_0.
\end{align*}
The last inequality again uses the $1$-Lipschitz property. Dividing by
$c_*r^{1-\upsilon}>0$ proves \eqref{eq:adaptedoccupation}. Steps 1 and 2
identify all dependencies of the constants; taking a common bound then
allows the supremum over models.
\end{proof}

To place this occupation estimate precisely relative to the generalized
Avikainen estimate in Taguchi's preprint (arXiv:2001.05608), define
$F_V(a)=\E\int_0^T\one_{\{V_{s-}\le a\}}ds$.
For $a<b$ and $b-a\le2$, the interval $(a,b]$ is contained in the ball
centered at $(a+b)/2$ with radius $(b-a)/2$. The lemma therefore gives
\[
0\le F_V(b)-F_V(a)\le c_*^{-1}C_0\left(\frac{b-a}{2}\right)^{\upsilon-1}.
\]
For $b-a>2$, use $F_V(b)-F_V(a)\le T$. Thus the distribution function of
the occupation measure is globally $(\upsilon-1)$-Hölder continuous. The
lemma therefore verifies an assumption allowing Taguchi's result to be
applied; the transfer through thresholds is not a distinct interpolation
method.

\begin{hypothesis}[Variable diagonal amplitudes]\label{HmultJ}
Keep the measure \eqref{eq:truncatedPi}, with $m=r=d$, the initial
probabilistic framework, and \ref{Hprod}. For $\varepsilon\in[0,1]$, the
forward coefficients are $b$, $\varepsilon\sigma$, and, on the axes
supporting $\Pi$,
\[
\beta(t,x,\mu,\alpha,z\mathbf e_j)
=\gamma_j(t,x,\mu,\alpha)z\mathbf e_j.
\]
The deterministic functions $\gamma_j$ are measurable and uniformly
Lipschitz in $(x,\mu,\alpha)$, with $W_2$ for the law variable, and
\[
0<\underline\gamma\le|\gamma_j(t,x,\mu,\alpha)|
\le\overline\gamma<\infty.
\]
The coefficients $b,\sigma$ satisfy \ref{HF}. The data and controls
satisfy \ref{HqJ}, with common bounds; the backward coefficients satisfy
\ref{HB}, the operator is the ALA prototype, and \ref{HD} is maintained.
The controls are the same as $\varepsilon$ varies. The independence,
affinity, and deterministic-control assumptions specific to \ref{HtruncJ}
are not imposed here.
\end{hypothesis}

\begin{theorem}[Occupation and chaos for variable amplitudes]
\label{thm:multchaos}
Under \ref{HmultJ}, the selected limiting and particle systems are
well posed. For $0<r\le1$,
\begin{align}
\sup_{\varepsilon,P}\E^P\int_0^T
\one_{\{\operatorname{dist}(X_{t-}^{\varepsilon,P},\Sigma)\le r\}}dt
&\le Cr^{\upsilon-1},\label{eq:multocc}\\
\sup_{\varepsilon\in[0,1]}\mathcal D_N^\varepsilon
&\le C\tau_d(N)^{(\upsilon-1)/(\upsilon+1)}.
\label{eq:multchaos}
\end{align}
The constants are independent of $P,N,\varepsilon$. For the Yosida source
and the regularized particles,
\begin{align}
\sup_{\varepsilon,P}\E^P\int_0^T
|A^\lambda(X_{t-}^{\varepsilon,P},\mu_{t-}^{\varepsilon,P})
-A^0(X_{t-}^{\varepsilon,P},\mu_{t-}^{\varepsilon,P})|^2dt
&\le C\lambda^{\upsilon-1},\label{eq:multYosida}\\
\lambda_N=\tau_d(N)^{1/(\upsilon+1)}\quad\Longrightarrow\quad
\sup_\varepsilon\mathcal D_{N,\lambda_N}^{\mathrm{reg},\varepsilon}
&\le C\tau_d(N)^{(\upsilon-1)/(\upsilon+1)}
\label{eq:multjoint}
\end{align}
for $N$ sufficiently large. The errors are quadratic and use the full norms
of \eqref{eq:directerror}.
\end{theorem}

\begin{proof}
\emph{Step 1: verification of the forward assumptions.}
With $C_p^\Pi=2\sum_jc_j/(p-\upsilon)$, $p>\upsilon$, we have
\[
\int_E|\beta(t,x,\mu,\alpha,e)|^p\Pi(de)
\le\overline\gamma^p C_p^\Pi.
\]
For two sets of arguments, write
$D=|x-x'|+W_2(\mu,\mu')+|\alpha-\alpha'|$. If $L_\gamma$ is the common Lipschitz constant, then
\[
\int_E|\beta(t,x,\mu,\alpha,e)-\beta(t,x',\mu',\alpha',e)|^2\Pi(de)
\le L_\gamma^2 C_2^\Pi D^2.
\]
The growth and order-$q$ conditions are therefore satisfied. The factor
$\varepsilon\le1$ preserves the bounds on $\sigma$. Theorem~\ref{thm:Fwell}
and the particle constructions yield the forward processes and their moments,
uniformly in $\varepsilon,P,N$. The selected source is predictable and has
linear growth; Theorem~\ref{thm:Bwell} and Proposition~\ref{prop:directwell}
give backward well-posedness. Occupation is not used in this construction.

\emph{Step 2: coordinatewise application of the lemma.}
Fix $j$. The restriction of $N$ to the $j$th axis, expressed in the
variable $z$, has compensator $c_j|z|^{-1-\upsilon}dz\,dt$. The other axes
do not contribute to the $j$th jump coordinate. Thus
$V=X^{\varepsilon,P,j}$ satisfies Lemma~\ref{lem:adaptedoccupation}, with
amplitude $\gamma_j(t,X_{t-},\mu_{t-},\alpha_t)$ and continuous part
\[
M_t^{c,j}=\varepsilon\int_0^t\sigma_{j,\cdot}
(s,X_{s-},\mu_{s-},\alpha_s)dB_s.
\]
By growth, \eqref{eq:qv}, the forward moments, and the control moments,
\[
\sup_{\varepsilon,P,j}\left(
\E^P\int_0^T|b_s^j|^2ds+\E^P\langle M^{c,j}\rangle_T\right)<\infty.
\]
The lower bound on the $c_j$ and the amplitude bounds make $c_*$ uniform.
The lemma with $a=0$, followed by
\[
\one_{\{\operatorname{dist}(x,\Sigma)\le r\}}
\le\sum_{j=1}^d\one_{\{|x_j|\le r\}},
\]
give \eqref{eq:multocc}. No lower bound on the continuous covariance is
used; the calculation includes $\varepsilon=0$.

\emph{Step 3: direct comparison of the sources.}
Under \ref{Hprod}, the copies of the limiting forward process have the same
moments and occupation as $X^{\varepsilon,P}$. The pointwise bound
\eqref{eq:pointdirect}, Markov's inequality, and
Theorem~\ref{thm:Fchaos} give
\[
\sup_\varepsilon D_{A,N}^\varepsilon
\le C\{\tau_d(N)+r^{\upsilon-1}+\tau_d(N)r^{-2}\}.
\]
For $N$ sufficiently large, $r_N=\tau_d(N)^{1/(\upsilon+1)}\le1$.
The last two terms are of order
$\tau_d(N)^{(\upsilon-1)/(\upsilon+1)}$, which dominates the first. Small
values of $N$ are absorbed by the uniform moment bounds.

\emph{Step 4: backward transfer and regularization.}
The energy estimate in the proof of Theorem~\ref{thm:directchaos} gives,
with the same Lipschitz constants,
\[
\mathcal D_N^\varepsilon
\le C\{\tau_d(N)+D_{A,N}^\varepsilon+N^{-1/2}\}.
\]
The last term is the empirical error of the backward copies, whose
$q$th moment is uniform by Lemma~\ref{lem:moments}. Since
$N^{-1/2}\le C\tau_d(N)$, Step 3 proves \eqref{eq:multchaos}. For the
cumulative process, the transfer is explicitly
\[
\frac1N\sum_i\E\sup_t|\Delta\mathsf K_t^i|^2
\le T|\varrho|^2\frac1N\sum_i\E\int_0^T|\Delta A_s^i|^2ds.
\]
The off-diagonal components remain included in the $Z,U$ norms of the
invoked theorem. Finally, \eqref{eq:multocc} verifies \ref{HoccD} with
$\theta=\upsilon-1$. Proposition~\ref{prop:sharpALA} gives
$C(\lambda^2+\lambda^{\upsilon-1})\le2C\lambda^{\upsilon-1}$
for $0<\lambda\le1$, which is \eqref{eq:multYosida}.
Theorem~\ref{thm:jointuniform} gives the stronger bound
\[
\mathcal D_{N,\lambda}^{\mathrm{reg},\varepsilon}
\le C\{\tau_d(N)^{(\upsilon-1)/(\upsilon+1)}+\lambda^{\upsilon-1}\}.
\]
The choice in \eqref{eq:multjoint}, or any smaller choice of $\lambda_N$,
ensures $\lambda_N^{\upsilon-1}\le\tau_d(N)^{(\upsilon-1)/(\upsilon+1)}$.
This completes the proof without a formal passage to $\lambda=0$.  A
separate strengthened result based on a second occupation moment is stated
below under additional assumptions.
\end{proof}

\begin{hypothesis}[Variable amplitudes: second-moment occupation regime]
\label{HmultJ2}
In addition to \ref{HmultJ}, each coordinate of the limiting forward
process admits a decomposition as in Lemma~\ref{lem:adaptedoccupation} for
which, uniformly in $P$, $\varepsilon$, and $j$,
\[
|b_t^j|\le C_b,
\qquad
\frac{d\langle M^{c,j}\rangle_t}{dt}\le C_c
\quad dt\otimes P\text{-p.p.}
\]
The predictable amplitudes retain the bounds
$0<\underline\gamma\le|\gamma_j|\le\overline\gamma$.
\end{hypothesis}

\begin{proposition}[Second-moment occupation for variable amplitudes]
\label{prop:multocc2}
Under \ref{HmultJ2}, there exists $C<\infty$, depending only on the common
constants in \ref{HmultJ2}, on $T,\upsilon,d$, and on the fixed
coefficients $(c_j)_{j\le d}$ of \eqref{eq:truncatedPi}, such that, for every
$P\in\Pcal$, $\varepsilon\in[0,1]$, deterministic $s\in[0,T]$, and
$0<r\le1$,
\[
\E^P\left[\int_s^T
\one_{\{\operatorname{dist}(X_{t-}^{\varepsilon,P},\Sigma)\le r\}}dt
\,\middle|\,\mathcal F_s^P\right]
\le Cr^{\upsilon-1}
\quad P\text{-a.s.}
\]
Consequently, \ref{Hocc2} holds with $\theta=\upsilon-1$, uniformly in
$P$ and $\varepsilon$.
\end{proposition}
\begin{proof}
Fix $P,\varepsilon,s$ and a coordinate $j$.  On $[s,T]$, write
$V_t=X_t^{\varepsilon,P,j}$ in the form
\[
V_t=V_s+\int_s^t b_u^j\,du+(M_t^{c,j}-M_s^{c,j})
+\int_s^t\int_{0<|z|\le1}\gamma_j(u)z\,\widetilde N(du,dz).
\]
Let $\psi_r(x)=\sqrt{x^2+r^2}$ and
$R_r(x,v)=\psi_r(x+v)-\psi_r(x)-\psi_r'(x)v$.  Step~2 of
Lemma~\ref{lem:adaptedoccupation} is pathwise and therefore remains valid
on $[s,T]$: with the same constant $c_*>0$,
\[
\int R_r(V_{u-}-a,\gamma_j(u)z)\nu(dz)
\ge c_*r^{1-\upsilon}
\one_{\{|V_{u-}-a|\le r\}} .
\tag{*}
\]
The constant is independent of $P,\varepsilon,s,j,a$, and $r$.

For fixed $r>0$, all terms below are integrable: $|\psi_r'|\le1$,
$\psi_r''\le r^{-1}$, the bracket density is bounded by \ref{HmultJ2},
and the jump remainder is dominated by
$(2r)^{-1}|\gamma_j(u)z|^2$.  Hence Itô's formula with jumps may be
applied on $[s,T]$ directly (equivalently, after localization followed by
conditional dominated convergence), giving
\begin{align*}
\psi_r(V_T-a)-\psi_r(V_s-a)
={}&\int_s^T\psi_r'(V_{u-}-a)b_u^j\,du
+\frac12\int_s^T\psi_r''(V_{u-}-a)d\langle M^{c,j}\rangle_u\\
&+\int_s^T\int R_r(V_{u-}-a,\gamma_j(u)z)\nu(dz)\,du
+Q_{s,T},
\end{align*}
where $Q_{s,T}$ is the sum of the continuous and compensated-jump
martingale increments.  These increments are square-integrable.  Hence
\[
\E^P[Q_{s,T}\mid\mathcal F_s^P]=0
\quad P\text{-a.s.}
\]
The continuous second-order term is nonnegative.  Using $(*)$, taking
conditional expectations, and using $|\psi_r'|\le1$ gives
\begin{align*}
&c_*r^{1-\upsilon}
\E^P\left[
\int_s^T\one_{\{|V_{u-}-a|\le r\}}du
\,\middle|\,\mathcal F_s^P\right]\\
&\qquad\le
\E^P\!\left[
\psi_r(V_T-a)-\psi_r(V_s-a)
+\int_s^T|b_u^j|\,du
\,\middle|\,\mathcal F_s^P\right].
\end{align*}
Since $\psi_r$ is $1$-Lipschitz, the first difference is bounded above by
$|V_T-V_s|$.  Conditional Cauchy--Schwarz and the conditional martingale
isometries, together with \ref{HmultJ2}, yield the deterministic estimate
\begin{align*}
\E^P[|V_T-V_s|\mid\mathcal F_s^P]
&\le C_b(T-s)
+C\Big(
\E^P[\langle M^{c,j}\rangle_T-\langle M^{c,j}\rangle_s
\mid\mathcal F_s^P]\Big)^{1/2}\\
&\quad
+C\left(
\E^P\left[
\int_s^T\int|\gamma_j(u)z|^2\nu(dz)\,du
\,\middle|\,\mathcal F_s^P\right]\right)^{1/2}\\
&\le C_bT+C(C_cT)^{1/2}
+C\overline\gamma
\left(T\int z^2\nu(dz)\right)^{1/2}
=:C_1 .
\end{align*}
Moreover,
\[
\E^P\left[\int_s^T|b_u^j|du\mid\mathcal F_s^P\right]\le C_bT.
\]
Consequently,
\[
\E^P\left[
\int_s^T\one_{\{|X_{u-}^{\varepsilon,P,j}-a|\le r\}}du
\,\middle|\,\mathcal F_s^P\right]
\le C r^{\upsilon-1}
\quad P\text{-a.s.}
\]
uniformly in all displayed parameters.

For the ALA interface,
$\{\operatorname{dist}(x,\Sigma)\le r\}\subset
\bigcup_{j=1}^d\{|x_j|\le r\}$.  Summing the preceding estimate over the
coordinates, with $a=0$, proves the conditional occupation bound for
$\Sigma$.  Finally Lemma~\ref{lem:conditional-occ2} yields
\[
\E^P\left(
\int_0^T
\one_{\{\operatorname{dist}(X_{t-}^{\varepsilon,P},\Sigma)\le r\}}dt
\right)^2
\le C r^{2(\upsilon-1)},
\]
which is precisely \ref{Hocc2} with $\theta=\upsilon-1$.
\end{proof}

\begin{corollary}[Strengthened rate for variable amplitudes]
\label{cor:multchaos2}
Under \ref{HmultJ2}, for $0<\lambda\le1$,
\[
\sup_{\varepsilon\in[0,1]}
\mathcal D_{N,\lambda}^{\rm reg,\varepsilon}
\le C\left\{
\tau_d(N)^{(\upsilon-1)/\upsilon}+\lambda^{2(\upsilon-1)}\right\},
\]
and
\[
\sup_{\varepsilon\in[0,1]}\mathcal D_N^\varepsilon
\le C\tau_d(N)^{(\upsilon-1)/\upsilon}.
\]
For $\upsilon=3/2$ and $d=1$, this gives
$C\{N^{-1/6}+\lambda\}$ for the regularized squared error.
\end{corollary}
\begin{proof}
Proposition~\ref{prop:multocc2} verifies \ref{Hocc2} with
$\theta=\upsilon-1$. Theorem~\ref{thm:jointquadratic} gives
\[
C\{\tau_d(N)^{(\upsilon-1)/\upsilon}+\lambda^2+
\lambda^{2(\upsilon-1)}\}.
\]
Since $0<\upsilon-1<1$ and $0<\lambda\le1$, the term $\lambda^2$ is
dominated by $\lambda^{2(\upsilon-1)}$. The selected case follows from
\eqref{eq:jointquadratic0}.
\end{proof}

\paragraph{Nonlinear example with individual control.}
Take $d=m=r=1$, $D=\R$, the truncated stable measure above, and
\[
b(t,x,\mu,\alpha)=-x+\tanh x+\int y\mu(dy)+\alpha,
\quad \sigma=1,\quad
\gamma(t,x,\mu,\alpha)=1+\tfrac14\tanh(x+\alpha).
\]
Then $3/4\le\gamma\le5/4$ and
\[
|b(x,\mu,\alpha)-b(x',\mu',\alpha')|
\le2|x-x'|+W_2(\mu,\mu')+|\alpha-\alpha'|,
\]
\[
|\gamma(x,\mu,\alpha)-\gamma(x',\mu',\alpha')|
\le\tfrac14(|x-x'|+|\alpha-\alpha'|).
\]
The bound on the mean term follows from coupling and Cauchy--Schwarz. For a
bounded individual predictable control compatible with \ref{Hprod}, and an
initial datum of order $q>4$, all forward assumptions are satisfied. One may
choose $\alpha=0$, or
$\alpha_t=\tanh(N((0,t)\times\{1/2<|z|\le1\}))$, which is bounded,
predictable, and a function only of the individual noise.
With $g(x,\mu)=x$, $f(t,x,\mu,y,z,u,\eta,\alpha)=y+x$, and the ALA
prototype, the backward assumptions are also satisfied. The direct squared
rate is $CN^{-(\upsilon-1)/(2(\upsilon+1))}$.

The exponent $\upsilon-1$ enlarges the class of coefficients but does not
improve on the exponent $1$ obtained in the affine class. It does not prove
the existence of marginal densities at each time. Moreover,
$\varepsilon=0$ removes the continuous martingale part of the forward
process; a control depending on the Brownian noise may still transmit that
noise to the drift and the backward equation. The identity $Z=0$ in
Proposition~\ref{prop:Zpurezero} retains its specific assumptions.

The measure \eqref{eq:truncatedPi} is truncated at large jumps; the process
is not an untruncated stable process. This distinction preserves the
quadratic and order-$q>4$ moments used in the manuscript. The Fourier
mechanism in this proof does not extend to jump coefficients with arbitrary
state dependence, nor to a random control that destroys the convolution in
Step 4.

To illustrate the need for an additional assumption, consider in dimension
one $X_t=\sum_{j=1}^{J_t}V_j$, $X_0=0$, where $J$ is a Poisson process with
intensity $\ell>0$ and the jumps are symmetric, bounded, and nonzero. This
process is square-integrable and its compensated form adds no drift. For every
$r>0$,
\[
\E\int_0^T\one_{\{|X_{t-}|\le r\}}dt
\ge\int_0^TP(J_{t-}=0)dt
=\frac{1-e^{-\ell T}}{\ell}>0.
\]
A bound $Cr^\theta$, $\theta>0$, is therefore impossible uniformly as
$r\downarrow0$ in this class. Infinite activity alone is not claimed to be
sufficient: Theorem~\ref{thm:levyoccupation} uses the precise form of the
measure and $\upsilon>1$.

\subsection{Linear occupation in a conjugated nonlinear class}
\label{sec:conjugate}

A smooth conjugacy makes it possible to retain occupation of order $r$ with
nonlinear coefficients. It imposes a stronger structure than mere
Lipschitz continuity and does not cover all amplitudes in
Hypothesis~\ref{HmultJ}.

\begin{hypothesis}[Conjugacy of truncated additive noise]
\label{HconjJ}
We work in dimension $d=1$, with $\sigma=0$ and no control in the forward
equation, and
\[
\Pi(dz)=c\one_{\{0<|z|\le1\}}|z|^{-1-\upsilon}dz,
\qquad c>0,\quad 1<\upsilon<2.
\]
The function $\Phi:\R\to\R$ is of class $C^3$, with
\[
0<m_\Phi\le\Phi'(v)\le M_\Phi<\infty,\qquad
\|\Phi''\|_\infty+\|\Phi'''\|_\infty<\infty.
\]
Let $\Psi=\Phi^{-1}$. The Borel function
$H:[0,T]\times\mathcal P_2(\R)\to\R$ satisfies
\[
|H(t,\mu)|\le M_H,\qquad
|H(t,\mu)-H(t,\mu')|\le L_H W_2(\mu,\mu').
\]
The coefficients are
\begin{align}
\beta(x,z)&=\Phi(\Psi(x)+z)-x,\nonumber\\
I(x)&=\int\{\Phi(\Psi(x)+z)-x-\Phi'(\Psi(x))z\}\Pi(dz),\nonumber\\
b(t,x,\mu)&=\Phi'(\Psi(x))H(t,\mu)+I(x).
\label{eq:conjcoeff}
\end{align}
Under each model $P$, $N$ is a Poisson random measure relative to the
augmented filtration used in the model, with compensator $\Pi(dz)dt$, and
\[
L_t=\int_0^t\int z\,\widetilde N^P(ds,dz)
\]
has independent increments relative to that filtration: for every
deterministic $0\le s<t\le T$,
\[
L_t-L_s\ \text{is independent of }\mathcal F_s^P\text{ under }P.
\]
The initial variable $\xi$ is independent of the full process $L$, and
$\sup_P\E^P|\xi|^q<\infty$ for some $q>4$. The martingale representation
framework, \ref{Hprod}, \ref{HB}, and \ref{HD} are maintained; the operator
is the ALA prototype. The backward data satisfy the moment and control
conditions of \ref{HqJ}.
\end{hypothesis}

\begin{lemma}[Conjugate representation]
\label{lem:conjrepresentation}
Under \ref{HconjJ}, the forward equation with coefficients
\eqref{eq:conjcoeff} is well posed under every $P\in\Pcal$.  Moreover, if
\[
V_t=\Psi(\xi)+\int_0^t
H(s,\Phi_\#\Lcal^P(V_s))\,ds+L_t,
\]
then
\[
X_t^P=\Phi(V_t),\qquad
V_t=\Psi(\xi)+m_t^P+L_t,\qquad
m_t^P:=\int_0^tH(s,\Phi_\#\Lcal^P(V_s))\,ds .
\]
For fixed $P$, the function $m^P$ is deterministic.  In particular, for
every deterministic $0\le s\le t\le T$,
\[
V_t=V_s+(m_t^P-m_s^P)+(L_t-L_s),
\]
and $L_t-L_s$ is independent of $\mathcal F_s^P$ under $P$, as required
in \ref{HconjJ}.
\end{lemma}
\begin{proof}
\emph{Step 1: coefficient estimates and well-posedness.}
The lower bound on $\Phi'$ implies that $\Phi$ is a bijection from
$\R$ onto $\R$ and that $|\Psi(x)-\Psi(x')|\le m_\Phi^{-1}|x-x'|$.
For $p>\upsilon$, set $J_p=\int|z|^p\Pi(dz)=2c/(p-\upsilon)$.
With $v=\Psi(x)$, Taylor's formula in integral form gives
\[
\Phi(v+z)-\Phi(v)-\Phi'(v)z
=z^2\int_0^1(1-u)\Phi''(v+uz)du.
\]
Thus $I$ is absolutely defined, Borel measurable, and
\begin{align*}
|I(x)|&\le\tfrac12\|\Phi''\|_\infty J_2,\\
|I(x)-I(x')|&\le\frac{\|\Phi'''\|_\infty J_2}{2m_\Phi}|x-x'|,\\
|\beta(x,z)|&\le M_\Phi|z|,\\
|\beta(x,z)-\beta(x',z)|
&\le\frac{\|\Phi''\|_\infty}{m_\Phi}|x-x'|\,|z|.
\end{align*}
The last inequality follows from
$\beta(x,z)=z\int_0^1\Phi'(\Psi(x)+uz)du$. Consequently,
\begin{align*}
|b(t,x,\mu)-b(t,x',\mu')|
&\le\left(\frac{M_H\|\Phi''\|_\infty}{m_\Phi}
+\frac{\|\Phi'''\|_\infty J_2}{2m_\Phi}\right)|x-x'|
+M_\Phi L_H W_2(\mu,\mu'),\\
\int|\beta(x,z)-\beta(x',z)|^2\Pi(dz)
&\le m_\Phi^{-2}\|\Phi''\|_\infty^2J_2|x-x'|^2,\\
\int|\beta(x,z)|^q\Pi(dz)&\le M_\Phi^qJ_q.
\end{align*}
These bounds verify the forward conditions of \ref{HF} and \ref{HqJ};
the forward well-posedness theorem of the manuscript applies to the limiting
system and to the particles. The bounds are uniform in $P$.

\emph{Step 2: identification by conjugacy.}
For a law $\nu$, denote by $\Phi_\#\nu$ its pushforward under $\Phi$.
For every coupling $(V,V')$,
$\E|\Phi(V)-\Phi(V')|^2\le M_\Phi^2\E|V-V'|^2$; hence
\[
W_2(\Phi_\#\nu,\Phi_\#\nu')\le M_\Phi W_2(\nu,\nu').
\]
The auxiliary equation
\[
V_t=\Psi(\xi)+\int_0^tH(s,\Phi_\#\Lcal^P(V_s))ds+L_t
\]
is therefore a McKean--Vlasov equation with Lipschitz coefficients and an
initial datum of order $q$. Its law is deterministic under $P$. Write
$m_t^P=\int_0^tH(s,\Phi_\#\Lcal^P(V_s))ds$; then
$V_t=\Psi(\xi)+m_t^P+L_t$. Itô's formula with jumps applied to $\Phi(V)$
gives
\begin{align*}
\Phi(V_t)=\xi
&+\int_0^t\Phi'(V_{s-})H(s,\Phi_\#\Lcal^P(V_s))ds\\
&+\int_0^t\int\{\Phi(V_{s-}+z)-\Phi(V_{s-})\}\widetilde N(ds,dz)\\
&+\int_0^t\int\{\Phi(V_{s-}+z)-\Phi(V_{s-})-\Phi'(V_{s-})z\}\Pi(dz)ds.
\end{align*}
We justify this identity by truncation, without assuming convergence of an
approximating nonlinear equation. Set
\[
j_n=\int_{|z|\le1/n}z^2\Pi(dz)
=\frac{2c}{2-\upsilon}n^{-(2-\upsilon)},\qquad
L_t^n=\int_0^t\int_{|z|>1/n}z\widetilde N(ds,dz),
\qquad V_t^n=\Psi(\xi)+m_t^P+L_t^n.
\]
The shift $m^P$ remains that of the limiting solution. By Doob's inequality,
\[
\E^P\sup_{t\le T}|V_t^n-V_t|^2\le4Tj_n\longrightarrow0.
\]
Write $B_\Phi(v,z)=\Phi(v+z)-\Phi(v)$ and
$R_\Phi(v,z)=B_\Phi(v,z)-\Phi'(v)z$. The integral formulas give
\begin{align*}
|B_\Phi(v,z)-B_\Phi(w,z)|&\le\|\Phi''\|_\infty|v-w||z|,\\
|R_\Phi(v,z)-R_\Phi(w,z)|&\le\tfrac12\|\Phi'''\|_\infty|v-w|z^2,\\
|R_\Phi(v,z)|&\le\tfrac12\|\Phi''\|_\infty z^2.
\end{align*}
The measure restricted to $|z|>1/n$ is finite. The change-of-variable rule
across its jumps gives Itô's formula for $\Phi(V^n)$ with the integrals
restricted to these marks. For the martingale part, Doob's inequality and the
isometry yield
\begin{align*}
&\E^P\sup_{t\le T}\left|
\int_0^t\int_{|z|>1/n}B_\Phi(V_{s-}^n,z)\widetilde N(ds,dz)
-\int_0^t\int B_\Phi(V_{s-},z)\widetilde N(ds,dz)\right|^2\\
&\qquad\le8\|\Phi''\|_\infty^2J_2
\int_0^T\E^P|V_{s-}^n-V_{s-}|^2ds+8M_\Phi^2Tj_n\longrightarrow0.
\end{align*}
For the compensated remainder, with
$I_\Phi(v)=\int R_\Phi(v,z)\Pi(dz)$ and
$I_\Phi^n(v)=\int_{|z|>1/n}R_\Phi(v,z)\Pi(dz)$,
\[
|I_\Phi^n(V_{s-}^n)-I_\Phi(V_{s-})|
\le\tfrac12\|\Phi'''\|_\infty J_2|V_{s-}^n-V_{s-}|
+\tfrac12\|\Phi''\|_\infty j_n.
\]
Cauchy--Schwarz in time implies convergence in $\mathbb S^2$ of their
drift integrals. The term containing $H$ satisfies
\[
\E^P\sup_{t\le T}\left|\int_0^t
[\Phi'(V_{s-}^n)-\Phi'(V_{s-})]
H(s,\Phi_\#\Lcal^P(V_s))ds\right|^2
\le TM_H^2\|\Phi''\|_\infty^2
\int_0^T\E^P|V_{s-}^n-V_{s-}|^2ds\longrightarrow0.
\]
Finally, $\Phi(V^n)\to\Phi(V)$ in $\mathbb S^2$ because $\Phi$ is
$M_\Phi$-Lipschitz. A subsequence converges uniformly almost surely for each
term; the integral identity therefore passes to the limit simultaneously for
all $t$. It identifies $X=\Phi(V)$ as a solution of \eqref{eq:conjcoeff}.
Forward uniqueness identifies this solution with the one already constructed;
the laws at times $s$ and $s-$ coincide for $ds$-almost every $s$.
The final increment representation follows from
$V_t=\Psi(\xi)+m_t^P+L_t$.  The deterministic character of $m^P$ follows
because $\Lcal^P(V_t)$ is an unconditional law under the fixed model.
The independence of $L_t-L_s$ from $\mathcal F_s^P$ is precisely the
independent-increment requirement stated in \ref{HconjJ}.
\end{proof}

\begin{proposition}[Second-moment occupation in the conjugated class]
\label{prop:conjocc2}
Under \ref{HconjJ}, there exists $C<\infty$, independent of
$P\in\Pcal$, $a\in\R$, deterministic $s\in[0,T]$, and $r\in(0,1]$, such
that
\[
\E^P\left[\int_s^T\one_{\{|X_{t-}^P-a|\le r\}}dt
\,\middle|\,\mathcal F_s^P\right]\le Cr
\quad P\text{-a.s.}
\]
Thus \ref{Hocc2} holds with $\theta=1$.
\end{proposition}
\begin{proof}
By Lemma~\ref{lem:conjrepresentation},
\[
V_t=V_s+(m_t^P-m_s^P)+(L_t-L_s),\qquad t\ge s,
\]
where the first two terms on the right are $\mathcal F_s^P$-measurable and
$L_t-L_s$ is independent of $\mathcal F_s^P$.  The Fourier calculation for
the truncated stable increment, now over an interval of length $t-s$,
gives a density $q_{t-s}$ satisfying
\[
\|q_{t-s}\|_\infty
\le C\{1+(t-s)^{-1/\upsilon}\}.
\]
Hence the conditional law of $V_t$ given $\mathcal F_s^P$ has the same
density bound, because the remaining term is only a random translation.

Since $X_t=\Phi(V_t)$ and
$\Psi'=(\Phi'\circ\Psi)^{-1}$ satisfies
$|\Psi'|\le m_\Phi^{-1}$, the conditional density of $X_t$ satisfies
\[
\|p_{s,t}^{P}(\cdot,\omega)\|_\infty
\le C\{1+(t-s)^{-1/\upsilon}\}
\]
for $P$-almost every $\omega$.  Therefore
\[
P(|X_t^P-a|\le r\mid\mathcal F_s^P)
\le Cr\{1+(t-s)^{-1/\upsilon}\}
\quad P\text{-a.s.}
\]
Conditional Tonelli and $\upsilon>1$ yield
\[
\E^P\left[\int_s^T\one_{\{|X_{t-}^P-a|\le r\}}dt
\,\middle|\,\mathcal F_s^P\right]
\le Cr\int_s^T\{1+(t-s)^{-1/\upsilon}\}dt
\le C_T r .
\]
As before, $X_t$ and $X_{t-}$ give the same time integral.  Finally,
Lemma~\ref{lem:conditional-occ2} yields the quadratic occupation bound
$Cr^2$, uniformly in $P$, which is \ref{Hocc2} with $\theta=1$.
\end{proof}

\begin{theorem}[Order-one occupation and approximation in the conjugated class]
\label{thm:conjoccupation}
Under \ref{HconjJ}, the forward equation is well posed. For $t>0$, its law
has a density $p_t^P$ such that
\begin{equation}
\sup_P\|p_t^P\|_\infty\le C_T(1+t^{-1/\upsilon}),\qquad
\sup_{P,a\in\R}\E^P\int_0^T\one_{\{|X_{t-}^P-a|\le r\}}dt\le C_T r.
\label{eq:conjocc}
\end{equation}
The selected limiting and particle backward systems are well posed and, for
$0<\lambda\le1$,
\begin{equation}
\mathcal D_N\le C\tau_1(N)^{1/2},\qquad
\mathcal D_{N,\lambda}^{\rm reg}\le C\{\tau_1(N)^{1/2}+\lambda^2\}.
\label{eq:conjchaos}
\end{equation}
In particular, $\tau_1(N)=N^{-1/2}$ gives a squared error bounded by
$C(N^{-1/4}+\lambda^2)$. The forward noise is purely discontinuous.
\end{theorem}
\begin{proof}
By Lemma~\ref{lem:conjrepresentation}, the forward equation is well posed,
$X^P=\Phi(V)$, and $V_t=\Psi(\xi)+m_t^P+L_t$ with deterministic $m^P$.

\emph{Step 1: Fourier bound for the additive noise.}
Symmetry gives
\[
\E e^{iuL_t}=e^{-t\psi(u)},\qquad
\psi(u)=2c\int_0^1(1-\cos(uz))z^{-1-\upsilon}dz.
\]
For $|u|\ge1$, the change of variable $w=|u|z$ yields
\[
\psi(u)=2c|u|^\upsilon\int_0^{|u|}(1-\cos w)w^{-1-\upsilon}dw
\ge c_*|u|^\upsilon,
\quad c_*=2c\int_0^1(1-\cos w)w^{-1-\upsilon}dw>0.
\]
The integral defining $c_*$ is finite because $\upsilon<2$. Hence
\[
\int_\R e^{-t\psi(u)}du
\le2+2\int_1^\infty e^{-tc_*u^\upsilon}du
\le C(1+t^{-1/\upsilon}).
\]
Fourier inversion yields a density $q_t$ of $L_t$, with
$\|q_t\|_\infty\le(2\pi)^{-1}\int e^{-t\psi(u)}du$.
Independence of $\Psi(\xi)$ and $L$, together with the deterministic
character of $m_t^P$, gives the density of $V_t$:
\[
p_{V,t}^P(v)=\int q_t(v-m_t^P-w)\Lcal^P(\Psi(\xi))(dw),
\qquad \|p_{V,t}^P\|_\infty\le C(1+t^{-1/\upsilon}).
\]
This convolution is used only for the limiting solution, not for the
particles whose empirical drift is random.

\emph{Step 2: return to the state variable and occupation.}
The change of variable $x=\Phi(v)$ gives
\[
p_t^P(x)=p_{V,t}^P(\Psi(x))\Psi'(x),\qquad
\|p_t^P\|_\infty\le m_\Phi^{-1}C(1+t^{-1/\upsilon}).
\]
By Tonelli, for every $a\in\R$ and $r>0$,
\begin{align*}
\E^P\int_0^T\one_{\{|X_t^P-a|\le r\}}dt
&=\int_0^T\int_{a-r}^{a+r}p_t^P(x)dxdt\\
&\le 2Cr\int_0^T(1+t^{-1/\upsilon})dt\\
&=2Cr\left(T+\frac{\upsilon}{\upsilon-1}T^{1-1/\upsilon}\right).
\end{align*}
For every càdlàg path, replacing $X_t$ by $X_{t-}$ does not change the time
integral; Tonelli gives the same identity after taking expectations. This
proves \eqref{eq:conjocc}, uniformly in $P$ and $a$.

\emph{Step 3: application to the backward system.}
In the scalar prototype, $\Sigma=\{0\}$. The preceding bound verifies
\ref{HoccD} with $\theta=1$, while Proposition~\ref{prop:conjocc2} verifies
\ref{Hocc2} with $\theta=1$. The growth of $A^0$ and the forward moments
ensure $\E\int|A^0(X_{s-},\mu_{s-})|^2ds<\infty$, and the same property
holds for the particles. The constructions in Theorem~\ref{thm:Bwell} and
Proposition~\ref{prop:directwell} therefore give existence and uniqueness
with the full martingale representation. Theorem~\ref{thm:jointquadratic}
and its case $\lambda=0$ apply and yield \eqref{eq:conjchaos}, including the
norm of $\mathsf K$.
\end{proof}

\paragraph{Explicit nonlinear example.}
Take $\Phi(v)=v+\delta\sin v$, $0<\delta<1$, and
$H(t,\mu)=h_0\int\tanh y\,\mu(dy)$. Then
$m_\Phi=1-\delta$, $M_\Phi=1+\delta$, $M_H\le|h_0|$ and
$L_H\le|h_0|$. With $v=\Psi(x)$ and
$\psi(1)=\int(1-\cos z)\Pi(dz)$, symmetry of $\Pi$ gives
\[
\beta(x,z)=z+\delta\{\sin(v+z)-\sin v\},\qquad
b(t,x,\mu)=(1+\delta\cos v)H(t,\mu)-\delta\psi(1)\sin v.
\]
These coefficients depend nonlinearly on the state; for $z\ne0$ in the
support of $\Pi$, the amplitude $\beta(x,z)$ is not constant in $x$. With
$g(x,\mu)=x$, $f(t,x,\mu,y,z,u,\eta,\alpha)=x+y$, and the ALA prototype,
we obtain a concrete system covered by the theorem.

The conjugacy procedure and Fourier inversion are classical. This result
provides an explicit occupation criterion in a nonlinear class; it is not a
new general heat-kernel estimate. Occupation of order $r$ for arbitrary
nonlinear coefficients satisfying only \ref{HF} remains beyond its scope.

\section{Scope, limitations, and conclusion}

The term ``monotone source'' refers to the vector-valued maximal monotone
operator $A(\cdot,\mu)$.  Its scalar projection
$\langle\varrho,A^0(x,\mu)\rangle$ is not asserted to be monotone in $x$ and
does not act on the backward variable.  Monotonicity is used for the
resolvent, the domination $|A^\lambda|\le |A^0|$, and graph closure.

The jump compensator $\Pi(de)dt$ is common to all models.  This makes
$\mathbb H_\Pi^2$ and the Poisson isometry uniform in $P$ and permits the
product construction.  A model-dependent compensator would require
additional comparability assumptions or a different quasi-sure framework.
Likewise, the presence of jumps alone does not imply the interface occupation
assumptions: compound Poisson laws may retain atoms.  The general results
therefore remain subject to the occupation hypotheses stated in each
theorem.  Section~8 verifies them for affine truncated-stable noise,
predictable nondegenerate amplitudes, and a conjugated scalar nonlinear
class; no claim is made for arbitrary nonlinear coefficients.

The two principal quantitative levels are
\[
\mathcal D_{N,\lambda}^{\rm reg}
\le C\{\tau_d(N)^{\theta/(\theta+2)}+\lambda^2+\lambda^\theta\}
\]
under first-moment occupation, and
\[
\mathcal D_{N,\lambda}^{\rm reg}
\le C\{\tau_d(N)^{\theta/(\theta+1)}
+\lambda^2+\lambda^{2\theta}\}
\]
under the stronger second-moment condition.  The improvement comes from
controlling the cumulative source itself; occupation is required only for
the limiting forward process.  For $\theta=1$ this changes the particle
exponent from $1/3$ to $1/2$ and the cumulative regularization term from
$\lambda$ to $\lambda^2$.  For predictable stable-like amplitudes the
corresponding second-moment rate is
$\tau_d(N)^{(\upsilon-1)/\upsilon}+\lambda^{2(\upsilon-1)}$.

The results are modelwise and establish neither quasi-sure aggregation nor
optimality of the rates.  Natural extensions include model-dependent jump
compensators and occupation estimates of order $r$ for broader nonlinear
jump dynamics. Density and heat-kernel methods such as
\cite{FournierPrintems2010,ChenZhang2016}, together with jump-FBSDE
discretization techniques such as \cite{BouchardElie2008}, provide natural
starting points for related extensions, but these questions require tools
beyond the arguments developed here.

\section*{Statements and Declarations}

\paragraph{Competing interests.}
The authors declare that they have no competing interests.

\paragraph{Funding.}
This research received no specific grant from any funding agency in the public, commercial, or not-for-profit sectors.

\paragraph{Use of generative artificial intelligence.}
Generative artificial-intelligence tools assisted with drafting,
reorganization, and the exposition of calculations. The authors reviewed the
resulting text and remain responsible for validating the arguments and
references.

\paragraph{Data availability.}
This work is theoretical and does not rely on any dataset.

\end{document}